\documentclass[10pt, a4paper]{article}
\usepackage{cmap}
\usepackage[main=english,russian]{babel}
\usepackage[utf8]{inputenc}
\usepackage[text={8in,10in},centering]{geometry}
\usepackage{amsfonts,amsmath,amsxtra,amsthm,amssymb,latexsym, tabularray}
\usepackage{graphicx}
\usepackage{verbatim}
\usepackage{cite, color}
\usepackage{subcaption}
\usepackage{float}
\usepackage{comment}
\newtheorem{lemma}{Lemma}
\newtheorem{theorem}{Theorem}

\newtheorem{proposition}{Proposition}
\newtheorem{corollary}{Corollary}

\newtheorem{example}{Example}

\def\s{\sigma}

\def\C{{\mathbb C}}

\newcommand{\Id}{\operatorname{Id}\nolimits}

\newcommand{\ad}{\operatorname{ad}\nolimits}
\newcommand{\const}{\operatorname{const}\nolimits}

\newcommand{\SU}{\operatorname{SU}\nolimits}

\newcommand{\tr}{\operatorname{tr}\nolimits}

\newcommand{\su}{\mathfrak{su}}

\newcommand{\arcosh}{\operatorname{arcosh}}

\newcommand{\be}[1]{\begin{equation}\label{#1}}
\newcommand{\ee}{\end{equation}}

\begin{document}

\title{The Lorentzian Problem on the Group $\SU(2)$}
\date{}
\author{A.Z. Ali, Yu.L. Sachkov}

\maketitle


\section{The Lie Group $\SU(2)$ and the Lie Algebra $\su(2)$}

Denote by $\C^{2 \times 2}$ the space of $2 \times 2$ complex matrices. The Lie group $\SU(2)$ is defined as
$$
\SU(2) = \{q \in \C^{2 \times 2} \mid qq^* = q^*q = \Id, \ \det q = 1\},
$$
where for a matrix $q = \left(\begin{array}{cc} a & b \\ c & d \end{array}\right)$ we denote $q^* = \left(\begin{array}{cc} \bar a & \bar c \\ \bar b & \bar d \end{array}\right)$. Thus,
$$
\SU(2) = \left\{ \left.\left(\begin{array}{cc} z_1 & z_2 \\ -\bar z_2 & \bar z_1 \end{array}\right) \in \C^{2 \times 2} \,\right|\, |z_1|^2 + |z_2|^2 = 1\right\}.
$$
The Lie algebra of this Lie group is
$$
\su(2) = \left\{X \in \C^{2 \times 2} \mid X+X^*=0, \ \tr X = 0\right\}.
$$
This Lie algebra is spanned by the elements $i\s_j$, $j = 1, 2, 3$, where
$$
\s_1 = \left(\begin{array}{cc} 0 & 1 \\ 1 & 0 \end{array}\right), \qquad
\s_2 = \left(\begin{array}{cc} 0 & -i \\ i & 0 \end{array}\right), \qquad
\s_3 = \left(\begin{array}{cc} 1 & 0 \\ 0 & -1 \end{array}\right)
$$
are the Pauli matrices.

Define left-invariant vector fields on the group $\SU(2)$:
$$
X_j(q) = q\left(-\frac{i}{2} \s_j\right), \qquad j = 1, 2, 3, \quad q \in \SU(2).
$$

Since
$$
[\sigma_1,\sigma_2] = 2i\sigma_3,\quad [\sigma_3,\sigma_1] = 2i \sigma_2,\quad [\sigma_2,\sigma_3] = 2i\sigma_1,
$$
the following relations hold:

$$
[X_1,X_2] = X_3, \quad [X_3,X_1] = X_2,\quad [X_2,X_3] = X_1.
$$

We pose the Lorentzian problem on the group $\SU(2)$ with the orthonormal frame $X_1$, $X_2$, $X_3$:
\begin{align}
&\dot q = u_1 X_1 + u_2 X_2 + u_3 X_3, \qquad q \in \SU(2), \label{pr1}\\
&u_1 \geq \sqrt{u_2^2 + u_3^2}, \label{pr2}\\
&q(0) = q_0 = \Id, \qquad q(t_1) = q_1, \label{pr3}\\
&\int_0^{t_1} \sqrt{u_1^2 - u_2^2 - u_3^2} dt \to \max. \label{pr4}
\end{align}

\section{Controllability}

Since $\SU(2)$ is a compact and connected Lie group, and also $Lie(\Gamma) = L$, where $\Gamma $ and $L$, we can use Theorem 6.2 from \cite{lie_control}:

\begin{theorem}
A left-invariant system $\Gamma \subset L$ is completely controllable on a compact connected Lie group $G$ if and only if $Lie(\Gamma)=L$.
\end{theorem}

This implies the complete controllability of the system \eqref{pr1}, \eqref{pr2}.

\section{Absence of Global Optimality}

\begin{proposition}
\label{prop_periodic_traj}
For any $q_0 \in \SU(2)$, there exists a timelike periodic trajectory passing through $q_0$.
\end{proposition}

\begin{proof}
Consider the system with $u_1 \equiv 1$, $u_2 \equiv 0$, $u_3 \equiv 0$. Let us find the flow of the field $X_1(q)$ and show that it is given by periodic functions. This will be the desired trajectory. Indeed,

    \begin{align*}
    & \sigma_1 = \begin{pmatrix}
    0 & 1 \\
    1 & 0
    \end{pmatrix},\ \sigma_1^2 = \begin{pmatrix}
    0 & 1\\
    1 & 0
    \end{pmatrix}\begin{pmatrix}
    0 & 1\\
    1 & 0
    \end{pmatrix} = \begin{pmatrix}
    1 & 0\\
    0 & 1
    \end{pmatrix} \Rightarrow \sigma_1^{2k-1} = \begin{pmatrix}
    0 & 1\\
    1 & 0
    \end{pmatrix},\ \sigma_1^{2k} = \begin{pmatrix}
    1 & 0\\
    0 & 1
    \end{pmatrix},\ k \in \mathbb{N}.
    \end{align*}

    \begin{align*}
    & q(t) = \left(-\frac{i}{2}\right)q_0e^{\sigma_{1}t} = \left( -\frac{i}{2} \right)q_0\left( E + t\sigma_1 + \frac{t^2}{2}\sigma_1^2 + ... + \frac{t^n}{n!}\sigma_1^n + ... \right) = \\
    & = \left( -\frac{i}{2} \right)q_0\left( E + t\begin{pmatrix}
    0 & 1\\
    1 & 0
    \end{pmatrix} + \frac{t^2}{2}\begin{pmatrix}
    1 & 0\\
    0 & 1
    \end{pmatrix} + ... + \frac{t^{2k-1}}{(2k-1)!}\begin{pmatrix}
    0 & 1\\
    1 & 0
    \end{pmatrix} + \frac{t^{2k}}{(2k)!}\begin{pmatrix}
    1 & 0\\
    0 & 1
    \end{pmatrix} + ... \right) = \\
    & = \left( -\frac{i}{2} \right)q_0\begin{pmatrix}
    1 + \frac{t^2}{2} + ... + \frac{t^{2k}}{(2k)!} + ... & t + \frac{t^3}{3!} + ... + \frac{t^{2k-1}}{(2k-1)!} + ... \\
    t + \frac{t^3}{3!} + ... + \frac{t^{2k-1}}{(2k-1)!} + ... & 1 + \frac{t^2}{2} + ... + \frac{t^{2k}}{(2k)!} + ...
    \end{pmatrix} = \left( -\frac{i}{2} \right)q_0\begin{pmatrix}
    \cos{t} & \sin{t}\\
    \sin{t} & \cos{t}
    \end{pmatrix}.
    \end{align*}
\end{proof}

\begin{corollary}
For any $q_0,\ q_1 \in \SU(2)$, there is no Lorentzian longest curve.
\end{corollary}

\begin{proof}
Denote the periodic trajectory obtained in Proposition \ref{prop_periodic_traj} as $\gamma_0$. Its length is positive and equals the period $2\pi$. $l(\gamma_0) = 2\pi$.

Assume there exists a timelike trajectory $\gamma_1$ connecting $q_0$ with $q_1$. Its length $0 < l(\gamma_1) < +\infty$.

Construct a sequence of timelike trajectories $\omega_k = k\gamma_0 + \gamma_1$, $k = 0,1,2, ...$, connecting $q_0$ with $q_1$. They are structured as follows: traverse the periodic trajectory $\gamma_0$ $k$ times and $\gamma_1$ once. The length $l(\omega_k) = 2k\pi + l(\gamma_1)$.

The limit of the sequence of lengths is infinity, therefore an optimal trajectory does not exist.
\end{proof}

\section{Investigation Using the Pontryagin Maximum Principle}

\subsection{Normal Extremals}

Naturally parameterized normal extremals are trajectories of the Hamiltonian system with the Hamiltonian $H = \frac 12 (-h_1^2 + h_2^2+h_3^2)$, belonging to the level surface $\{H = - \frac 12\}$.
In this case, the extremal controls are equal to $u_1 = - h_1$, $u_2 = h_2$, $u_3 = h_3$. Here $h_i(\lambda) = \langle \lambda, X_i\rangle$.

The system for the vertical variables is written as follows:

\begin{equation}
\label{vert_obsh}
\begin{cases}
\dot{h}_1 = \{H,h_1 \} = \frac{1}{2}\{-h_1^2 + h_2^2 + h_3^2,h_1\} = h_2\{h_2,h_1\} + h_3\{h_3,h_1\} = -h_2h_3 + h_3h_2 = 0,\\
\dot{h}_2 = \{H,h_2\} = \frac{1}{2}\{-h_1^2 + h_2^2 + h_3^2,h_2\} = -h_1\{h_1,h_2\} + h_3\{h_3,h_2\} = -h_1h_3 - h_3h_1 = -2h_1h_3,\\
\dot{h}_3 = \{H,h_3\} = \frac{1}{2}\{-h_1^2 + h_2^2 + h_3^2,h_3\} = -h_1\{h_1,h_3\} + h_2\{h_2,h_3\} = h_1h_2 + h_1h_2 = 2h_1h_2,
\end{cases}
\end{equation}
since $[X_1,X_2] = X_3$, $[X_3,X_1] = X_2$, $[X_2,X_3] = X_1$.

We introduce a parameterization on the lower sheet of the two-sheeted hyperboloid:

\begin{align*}
& -h_1^2 + h_2^2 + h_3^2 = -1, \quad h_1 < 0,\\
& h_1 = -\ch{\theta},\ h_2 = \cos{\varphi}\sh{\theta},\ h_3 = \sin{\varphi}\sh{\theta}.
\end{align*}

Then

\begin{equation}
\label{vert_norm}
\begin{cases}
\dot{h}_1 = -\dot{\theta}\sh{\theta},\\
\dot{h}_2 = -\dot{\varphi}\sin{\varphi}\sh{\theta} + \dot{\theta}\cos{\varphi}\ch{\theta},\\ 
\dot{h}_3 = \dot{\varphi}\cos{\varphi}\sh{\theta} + \dot{\theta}\sin{\varphi}\ch{\theta},
\end{cases}
\end{equation}

\begin{equation}
\label{norm_comb}
2h_1h_3 = -\sin{\varphi}\sh{2\theta}, \quad 2h_1h_2 = -\cos{\varphi}\sh{2\theta}.
\end{equation}

Using (\ref{vert_obsh}), (\ref{vert_norm}) and (\ref{norm_comb}), we obtain:

\begin{equation*}
\begin{cases}
\dot{h}_1 = \dot{\theta}\sh{\theta} = 0,\\
\dot{h}_2 = -\dot{\varphi}\sin{\varphi}\sh{\theta} + \dot{\theta}\cos{\varphi}\ch{\theta} = \sin{\varphi}\sh{2\theta},\\ 
\dot{h}_3 = \dot{\varphi}\cos{\varphi}\sh{\theta} + \dot{\theta}\sin{\varphi}\ch{\theta} = -\cos{\varphi}\sh{2\theta}.
\end{cases}
\end{equation*}

From the first equation we get $\dot{\theta} = 0$, i.e., $\theta \equiv \theta_0$. From the second and third:

\begin{equation*}
\dot{\varphi} = -2\ch{\theta_0} \Leftrightarrow \varphi = -2t\ch{\theta_0} + \varphi_0.
\end{equation*}

Let us write out the horizontal part of the system:

$$
\dot{q} = -\frac{i}{2} q\left( -h_1\sigma_1 + h_2\sigma_2 + h_3\sigma_3 \right) \Leftrightarrow \begin{cases}
\dot{z}_1 = -\frac{i}{2}\left( z_1h_3 + z_2(-h_1 + ih_2) \right),\\
\dot{z}_2 = -\frac{i}{2}\left( -z_1(h_1 + ih_2) - z_2h_3 \right).
\end{cases}
$$

\begin{equation*}
\begin{cases}
\dot{x}_1 = 1/2\sh{\theta_0}\sin{(-2t\ch{\theta_0} + \varphi_0)}y_1 + 1/2\sh{\theta_0}\cos{(-2t\ch{\theta_0} + \varphi_0)}x_2 + 1/2\ch{\theta_0}y_2,\\
\dot{y}_1 = -1/2\sh{\theta_0}\sin{(-2t\ch{\theta_0} + \varphi_0)}x_1-1/2\ch{\theta_0}x_2 + 1/2\sh{\theta_0}\cos{(-2t\ch{\theta_0} + \varphi_0)}y_2,\\
\dot{x}_2 = 1/2\ch{\theta_0}y_1 - 1/2\sh{\theta_0}\cos{(-2t\ch{\theta_0} + \varphi_0)}x_1 - 1/2\sh{\theta_0}\sin{(-2t\ch{\theta_0} + \varphi_0)}y_2 ,\\
\dot{y}_2 = -1/2\ch{\theta_0}x_1 - 1/2\sh{\theta_0}\cos{(-2t\ch{\theta_0} + \varphi_0)}y_1 + 1/2\sh{\theta_0}\sin{(-2t\ch{\theta_0} + \varphi_0)}x_2.
\end{cases}
\end{equation*}

\begin{theorem}
\label{norm_traj_form}
Normal extremals with initial conditions $(\theta_0,\varphi_0,x_{10},y_{10},x_{20},y_{20})$ are given by the following formulas:
\begin{equation}
\label{norm_traj}
    \begin{cases}
    x_1(t) = x_{10}(\alpha_1 + \frac{h_1}{m}\alpha_2) + y_{10}(\frac{a}{m}\alpha_3 + \frac{c}{m}\alpha_2) + x_{20}(\frac{c}{m}\alpha_3 - \frac{a}{m}\alpha_2) + y_{20}(\frac{h_1}{m}\alpha_3 - \alpha_4),\\
    y_1(t) = -x_{10}(\frac{a}{m}\alpha_3 + \frac{c}{m}\alpha_2) + y_{10}(\alpha_1 + \frac{h_1}{m}\alpha_2) + x_{20}(-\frac{h_1}{m}\alpha_3 + \alpha_4) + y_{20}(\frac{c}{m}\alpha_3 - \frac{a}{m}\alpha_2),\\
    x_2(t) = x_{10}( \frac{a}{m}\alpha_2 - \frac{c}{m}\alpha_3 ) + y_{10}( -\alpha_4 + \frac{h_1}{m}\alpha_3 ) + x_{20}( \frac{h_1}{m}\alpha_2 + \alpha_1 ) + y_{20}( -\frac{c}{m}\alpha_2 - \frac{a}{m}\alpha_3 ),\\
    y_2(t) = x_{10}( \alpha_4 - \frac{h_1}{m}\alpha_3 ) + y_{10}( \frac{a}{m}\alpha_2 - \frac{c}{m}\alpha_3 ) + x_{20}(\frac{c}{m}\alpha_2 + \frac{a}{m}\alpha_3 )+ y_{20}( \frac{h_1}{m}\alpha_2 + \alpha_1 ),
    \end{cases}
    \end{equation}
where
\begin{align*}
    & \alpha_1 = \cos{(h_1t)} \cos{\left( \frac{t}{2}m \right)}, \ \alpha_2 = \sin{(h_1t)} \sin{\left( \frac{t}{2}m \right)}, \ \alpha_3 = \cos{(h_1t)}\sin{\left( \frac{t}{2}m \right)},\ \alpha_4 = \sin{\left( h_1 t \right)}\cos{\left( \frac{t}{2}m \right)},\\
    & h_1 = -\ch{\theta_0}, \quad  a = \sin{\varphi_0}\sh{\theta_0}, \quad c = \cos{\varphi_0}\sh{\theta_0}, \quad m = \sqrt{\ch{(2\theta_0)}}.
    \end{align*}
\end{theorem}

\begin{proof}
We use the following method for solving a non-autonomous system of differential equations from \cite{notes}: \begin{enumerate}
\item[(a)]
Write the horizontal subsystem of the Hamiltonian system as $\dot q = V_t(q) := - h_1 X_1 + h_2 X_2 + h_3 X_3$
\item[(b)]
Prove that $\frac{d}{dt} V_t = \ad f (V_t)$ for some field $f$.
\item[(c)]
Conclude from this that $V_t = e^{t \ad f} g$ for some field $g$.
\item[(d)]
Use formula (19.6) from \cite{notes}
$$
q(t) = \stackrel{\longrightarrow}{\exp} \int_0^t V_{\tau} dt (q_0) 
= \stackrel{\longrightarrow}{\exp} \int_0^t e^{\tau \ad f} g dt (q_0)
= e^{-tf} \circ e^{t(f+g)} (q_0).
$$
\end{enumerate}

\begin{itemize}
    \item[$(a)$] 
    \begin{align*}
    & \dot{q} = -h_1X_1 + h_2X_2 + h_3X_3 = -\frac{i}{2} q\left( -h_1\sigma_1 + h_2\sigma_2 + h_3\sigma_3 \right) = \\
    & = -\frac{i}{2} q\left( \ch{\theta_0}\sigma_1 + \cos{(-2t\ch{\theta_0} + \varphi_0)}\sh{\theta_0}\sigma_2 + \sin{(-2t\ch{\theta_0}+\varphi_0)}\sh_{\theta_0}\sigma_3 \right) = V_t(q).
    \end{align*}

    \item[$(b)$] 

    \begin{align*}
    & h_1 = -\ch{\theta_0},\quad \dot{h}_1 = 0,\\
    & h_2 = \cos{(-2t\ch{\theta_0}+\varphi_0)}\sh{\theta_0},\quad \dot{h}_2 = 2\ch{\theta_0}\sin{2t\ch{(\theta_0+\varphi_0)}}\sh{\theta_0} = -2h_1h_3,\\
    & h_3 = \sin{(-2t\ch{\theta_0}+\varphi_0)}\sh{\theta_0},\quad \dot{h}_3 = (-2\ch{\theta_0})\cos{(-2t\ch_{\theta_0}+\varphi_0)}\sh{\theta_0} = 2h_1h_2.
    \end{align*}

    \begin{equation*}
    \frac{d}{dt}V_t = -\frac{i}{2} q\left( -2h_1h_3\sigma_2 + 2h_1h_2\sigma_3 \right) = - 2h_1h_3X_2 + 2h_1h_2X_3.
    \end{equation*}

    Take $f = 2h_1X_1 = -2\ch{\theta_0}X_1$ and obtain the required relation:
    \begin{equation*}
    \frac{d}{dt}V_t = -2h_1h_3X_2 + 2h_1h_2X_3 = [f,V_t].
    \end{equation*}

    \item[$(c)$]

    We need to find $g$ such that $V_t = e^{t\ad{f}}g$.

    Substitute $t = 0$, then $g = V_0$:

    \begin{equation*}
    g = V_0 = -h_1\vert_{t=0}X_1 + h_2\vert_{t=0}X_2 + h_3\vert_{t=0}X_3 = \ch{\theta_0}X_1 + \cos{\varphi_0}\sh{\theta_0}X_2 + \sin{\varphi_0}\sh{\theta_0}X_3 = -h_1X_1 + cX_2 + aX_3.
    \end{equation*}

    \item[$(d)$]

    $$
    q(t) = \stackrel{\longrightarrow}{\exp} \int_0^t V_{\tau} dt (q_0) 
    = \stackrel{\longrightarrow}{\exp} \int_0^t e^{\tau \ad f} g dt (q_0)
    = e^{-tf} \circ e^{t(f+g)} (q_0).
    $$

    Let us compute the flows of the fields $-f$ and $f+g$, and then their composition.
    
    First, $-f$:

    $$
    \dot{q} = -f = -2h_1X_1 = ih_1 q\sigma_1 \Rightarrow q(t) = q_0e^{ih_1t\sigma_1}.
    $$

    Then

     \begin{align*}
    & q(t) = q_0\left( E + ih_1t\sigma_1 - \frac{(h_1t)^2}{2}\sigma_1^2 - i\frac{(h_1t)^3}{3!}\sigma_1^3 + \frac{(h_1t)^4}{4!}\sigma_1^4 + ... \right) = \\
    & = q_0\left( E + ih_1t\sigma_1 - \frac{(h_1t)^2}{2}E - i\frac{(h_1t)^3}{3!}\sigma_1 + \frac{(h_1t)^4}{4!}E + ... \right) = \\
    & = q_0\begin{pmatrix}
    1 - \frac{(h_1t)^2}{2} + \frac{(h_1t)^4}{4!} + ... & i\left( h_1t - \frac{(h_1t)^3}{3!} + \frac{(h_1t)^5}{5!} + ... \right) \\
    i\left( h_1t - \frac{(h_1t)^3}{3!} + \frac{(h_1t)^5}{5!} + ... \right) & 1 - \frac{(h_1t)^2}{2} + \frac{(h_1t)^4}{4!} + ...
    \end{pmatrix} = q_0 \begin{pmatrix}
    \cos{(h_1t)} & i\sin{(h_1t)} \\
    i\sin{(h_1t)} & \cos{(h_1t)}
    \end{pmatrix},
     \end{align*}

     since

    \begin{align*}
    & \sigma_1 = \begin{pmatrix}
    0 & 1 \\
    1 & 0
    \end{pmatrix},\ \sigma_1^2 = \begin{pmatrix}
    0 & 1\\
    1 & 0
    \end{pmatrix}\begin{pmatrix}
    0 & 1\\
    1 & 0
    \end{pmatrix} = \begin{pmatrix}
    1 & 0\\
    0 & 1
    \end{pmatrix} \Rightarrow \sigma_1^{2k-1} = \begin{pmatrix}
    0 & 1\\
    1 & 0
    \end{pmatrix},\ \sigma_1^{2k} = \begin{pmatrix}
    1 & 0\\
    0 & 1
    \end{pmatrix},\ k \in \mathbb{N}.
    \end{align*}

    Now $f+g$:

    \begin{align*}
    & \dot{q} = f + g = -2\ch{\theta_0}X_1 + \ch{\theta_0}X_1 + \cos{\varphi_0}\sh{\theta_0}X_2 + \sin{\varphi_0}\sh{\theta_0}X_3 = -\ch{\theta_0}X_1 + \cos{\varphi_0}\sh{\theta_0}X_2 + \sin{\varphi_0}\sh{\theta_0}X_3 = \\
    & = \left( -\frac{i}{2} \right)q\left( -\ch{\theta_0}\begin{pmatrix}
    0 & 1\\
    1 & 0
    \end{pmatrix} + \cos{\varphi_0}\sh{\theta_0}\begin{pmatrix}
    0 & -i\\
    i & 0
    \end{pmatrix} + \sin{\varphi_0}\sh{\theta_0}\begin{pmatrix}
    1 & 0\\
    0 & -1
    \end{pmatrix} \right) = \\
    & = \left( -\frac{i}{2} \right)q\begin{pmatrix}
    \sin{\varphi_0}\sh{\theta_0} & -\ch{\theta_0} -i\cos{\varphi_0}\sh{\theta_0} \\
    -\ch{\theta_0}+ i\cos{\varphi_0}\sh{\theta_0} & -\sin{\varphi_0}\sh{\theta_0}
    \end{pmatrix} = \left( -\frac{i}{2} \right)qA.
    \end{align*}

    Then

    \begin{align*}
    & q(t) = q_0\left( E - t\frac{i}{2}A + \frac{t^2}{2}\left( -\frac{i}{2} \right)^2 A^2 + \frac{t^3}{3!}\left( -\frac{i}{2} \right)^3A^3 + ... + \frac{t^{2k}}{(2k)!}\left( -\frac{i}{2} \right)^{2k}A^{2k} + \frac{t^{2k+1}}{(2k+1)!}\left( -\frac{i}{2} \right)^{2k+1}A^{2k+1} + ... \right) = \\
    & = q_0\left(  E - t\frac{i}{2}A + \frac{t^2}{2}\left( -\frac{i}{2} \right)^2 \ch{(2\theta_0)}E  + ... + \frac{t^{2k}}{(2k)!}\left( -\frac{i}{2} \right)^{2k} \ch^{k}{(2\theta_0)}E + \frac{t^{2k+1}}{(2k+1)!}\left( -\frac{i}{2} \right)^{2k+1} \ch^k{(2\theta_0)} A + ... \right) = \\
    & = q_0 \begin{pmatrix}
   \cos{\left( \frac{t}{2}\sqrt{\ch{2\theta_0}} \right)} -i\frac{a}{\sqrt{\ch{2\theta_0}}}\sin{\left( \frac{t}{2}\sqrt{\ch{2\theta_0}}  \right)} & -\frac{c + ih_1}{\sqrt{\ch{2\theta_0}}}\sin{\left( \frac{t}{2}\sqrt{\ch{2\theta_0}} \right)} \\
    -\frac{-c + ih_1}{\sqrt{\ch{2\theta_0}}}\sin{\left( \frac{t}{2}\sqrt{\ch{2\theta_0}} \right)} & \cos{\left( \frac{t}{2}\sqrt{\ch{2\theta_0}} \right)} + i\frac{a}{\sqrt{\ch{2\theta_0}}}\sin{\left( \frac{t}{2}\sqrt{\ch{2\theta_0}}  \right)}
    \end{pmatrix},
    \end{align*}

    since
    \begin{align*}
    & A = \begin{pmatrix}
    \sin{\varphi_0}\sh{\theta_0} & -\ch{\theta_0} -i\cos{\varphi_0}\sh{\theta_0} \\
    -\ch{\theta_0}+ i\cos{\varphi_0}\sh{\theta_0} & -\sin{\varphi_0}\sh{\theta_0}
    \end{pmatrix},\\
    & A^2 = \begin{pmatrix}
    \sin{\varphi_0}\sh{\theta_0} & -\ch{\theta_0} -i\cos{\varphi_0}\sh{\theta_0} \\
    -\ch{\theta_0}+ i\cos{\varphi_0}\sh{\theta_0} & -\sin{\varphi_0}\sh{\theta_0}
    \end{pmatrix}^2 =  \begin{pmatrix}
    \ch{2\theta_0} & 0 \\
    0 & \ch{2\theta_0}
    \end{pmatrix},\\
    & A^3 = \ch{2\theta_0}A,\ A^4 = \ch{2\theta_0}A^2 = \ch^2{2\theta_0}E,\ A^5 = A^4A = \ch^2{2\theta_0}A,... \\
    & A^{2k} = \ch^k{(2\theta_0)}E,\ A^{2k+1} = \ch^k{(2\theta_0)}A,\ k = 0,1,2,3,....
    \end{align*}

    Composition:

    \begin{align*}
    & e^{-tf}\circ e^{t(f+g)}(q_0) = \\
    & = q_0 \begin{pmatrix}
   \cos{\left( \frac{t}{2}\sqrt{\ch{2\theta_0}} \right)} -i\frac{a}{\sqrt{\ch{2\theta_0}}}\sin{\left( \frac{t}{2}\sqrt{\ch{2\theta_0}}  \right)} & -\frac{c + ih_1}{\sqrt{\ch{2\theta_0}}}\sin{\left( \frac{t}{2}\sqrt{\ch{2\theta_0}} \right)} \\
    -\frac{-c + ih_1}{\sqrt{\ch{2\theta_0}}}\sin{\left( \frac{t}{2}\sqrt{\ch{2\theta_0}} \right)} & \cos{\left( \frac{t}{2}\sqrt{\ch{2\theta_0}} \right)} + i\frac{a}{\sqrt{\ch{2\theta_0}}}\sin{\left( \frac{t}{2}\sqrt{\ch{2\theta_0}}  \right)}
    \end{pmatrix} \begin{pmatrix}
    \cos{(h_1t)} & i\sin{(h_1t)} \\
    i\sin{(h_1t)} & \cos{(h_1t)}
    \end{pmatrix}
     = \\
    & = q_0 \begin{pmatrix}
   \cos{\left( \frac{t}{2}m \right)} -i\frac{a}{m}\sin{\left( \frac{t}{2}m  \right)} & -\frac{c + ih_1}{m}\sin{\left( \frac{t}{2}m \right)} \\
    -\frac{-c + ih_1}{m}\sin{\left( \frac{t}{2}m \right)} & \cos{\left( \frac{t}{2}m \right)} + i\frac{a}{m}\sin{\left( \frac{t}{2}m  \right)}
    \end{pmatrix} \begin{pmatrix}
    \cos{(h_1t)} & i\sin{(h_1t)} \\
    i\sin{(h_1t)} & \cos{(h_1t)}
    \end{pmatrix}
     = q_0\begin{pmatrix}
    b_{11} & b_{12} \\
    b_{21} & b_{22}
    \end{pmatrix} = \\
    & = \begin{pmatrix}
    z_{10}b_{11} + z_{20}b_{21} & z_{10}b_{12} + z_{20}b_{22} \\
    -\bar{z}_{20}b_{11} + \bar{z}_1b_{21} & -\bar{z}_{20}b_{12} + \bar{z}_1b_{22}
    \end{pmatrix} = \begin{pmatrix}
    (x_{10}+iy_{10})b_{11} + (x_{20}+iy_{20})b_{21} & (x_{10}+iy_{10})b_{12} + (x_{20} + iy_{20})b_{22} \\
    -(x_{20}-iy_{20})b_{11} + (x_{10} - iy_{10})b_{21} & -(x_{20}-iy_{20})b_{12} + (x_{10} - iy_{10})b_{22}
    \end{pmatrix},
    \end{align*}

 where
    \begin{align*}
    & b_{11} = \cos{\left(h_1t\right)} \cos{\left( \frac{t}{2}m \right)} - i \frac{a}{m}\sin{\left( \frac{t}{2}m \right)}\cos{\left(h_1 t\right)} -i \frac{c+ih_1}{m}\sin{\left(\frac{t}{2}m \right)}\sin{\left(h_1 t \right)} = \\
    & = f\left(h_1t\right)f\left( \frac{t}{2}m \right) + \frac{h_1}{m}g\left( \frac{t}{2}m \right)g\left( \frac{t}{2}m \right) - i\frac{a}{m}g\left( \frac{t}{2}m \right)f\left(h_1t\right) - i\frac{c}{m}g\left( \frac{t}{2}m \right)g\left( \frac{t}{2}m \right) = \alpha_1 + \frac{h_1}{m}\alpha_2 - i\frac{a}{m}\alpha_3 - i\frac{c}{m}\alpha_2,\\
    & b_{12} = i\cos{\left(\frac{t}{2}m\right)}\sin{\left(h_1 t\right)} + \frac{a}{m}\sin{\left( \frac{t}{2}m \right)}\sin{\left( h_1 t \right)} - \frac{c + ih_1}{m}\sin{\left(\frac{t}{2}m\right)}\cos{\left(h_1 t\right)} = \\
    & = \frac{a}{m}g\left( h_1 t \right)g\left( \frac{t}{2}m \right) - \frac{c}{m}g\left( \frac{t}{2}m \right)f\left(h_1t\right) + if\left( \frac{t}{2}m \right)g\left(h_1t\right) - i\frac{h_1}{m}g\left( \frac{t}{2}m \right)f\left(h_1t\right) = \frac{a}{m}\alpha_2 - \frac{c}{m}\alpha_3 + i\alpha_4 - i\frac{h_1}{m}\alpha_3,\\
    & b_{21} = -\frac{-c+ih_1}{m}\sin{\left(\frac{t}{2}m\right)}\cos{\left(h_1 t\right)} + i\sin{\left(h_1 t\right)}\cos{\left(\frac{t}{2}m\right)} - \frac{a}{m}\sin{\left(\frac{t}{2}m\right)}\sin{\left(h_1t\right)} = \\
    & = \frac{c}{m}g\left( \frac{t}{2}m \right)f\left(h_1t\right) - \frac{a}{m}g\left( \frac{t}{2}m \right)g\left(h_1t\right) - i\frac{h_1}{m}g\left( \frac{t}{2}m \right)f\left(h_1t\right) + if\left( \frac{t}{2}m \right)g\left(h_1t\right) = \frac{c}{m}\alpha_3 - \frac{a}{m}\alpha_2 - i \frac{h_1}{m}\alpha_3 + i\alpha_4,\\
    & b_{22} = -i\frac{-c+ih_1}{m}\sin{\left(\frac{t}{2}m\right)}\sin{\left(h_1 t\right)} + \cos{\left( \frac{t}{2}m \right)}\cos{\left(h_1 t\right)} + i\frac{a}{m}\sin{\left(\frac{t}{2}m\right)}\cos{\left(h_1 t\right)} = \\
    & = \frac{h_1}{m}g\left( \frac{t}{2}m \right)g\left(h_1t\right) + f\left( \frac{t}{2}m \right)f\left(h_1t\right) + i\frac{c}{m}g\left( \frac{t}{2}m \right)g\left(h_1t\right) + i\frac{a}{m}g\left( \frac{t}{2}m \right)f\left(h_1t\right) = \frac{h_1}{m}\alpha_2 + \alpha_1 + i\frac{c}{m}\alpha_2 + i\frac{a}{m}\alpha_3,
    \end{align*}

    \begin{align*}
    & \alpha_1 = \cos{(h_1t)} \cos{\left( \frac{t}{2}m \right)}, \ \alpha_2 = \sin{(h_1t)} \sin{\left( \frac{t}{2}m \right)}, \ \alpha_3 = \cos{(h_1t)}\sin{\left( \frac{t}{2}m \right)},\ \alpha_4 = \sin{\left( h_1 t \right)}\cos{\left( \frac{t}{2}m \right)},\\
    & h_1 = -\ch{\theta_0}, \quad  a = \sin{\varphi_0}\sh{\theta_0}, \quad c = \cos{\varphi_0}\sh{\theta_0}, \quad m = \sqrt{\ch{(2\theta_0)}}.
    \end{align*}

    Thus, the solution is:

    \begin{equation*}
    \begin{cases}
    x_1(t) = x_{10}(\alpha_1 + \frac{h_1}{m}\alpha_2) + y_{10}(\frac{a}{m}\alpha_3 + \frac{c}{m}\alpha_2) + x_{20}(\frac{c}{m}\alpha_3 - \frac{a}{m}\alpha_2) + y_{20}(\frac{h_1}{m}\alpha_3 - \alpha_4),\\
    y_1(t) = -x_{10}(\frac{a}{m}\alpha_3 + \frac{c}{m}\alpha_2) + y_{10}(\alpha_1 + \frac{h_1}{m}\alpha_2) + x_{20}(-\frac{h_1}{m}\alpha_3 + \alpha_4) + y_{20}(\frac{c}{m}\alpha_3 - \frac{a}{m}\alpha_2),\\
    x_2(t) = x_{10}( \frac{a}{m}\alpha_2 - \frac{c}{m}\alpha_3 ) + y_{10}( -\alpha_4 + \frac{h_1}{m}\alpha_3 ) + x_{20}( \frac{h_1}{m}\alpha_2 + \alpha_1 ) + y_{20}( -\frac{c}{m}\alpha_2 - \frac{a}{m}\alpha_3 ),\\
    y_2(t) = x_{10}( \alpha_4 - \frac{h_1}{m}\alpha_3 ) + y_{10}( \frac{a}{m}\alpha_2 - \frac{c}{m}\alpha_3 ) + x_{20}(\frac{c}{m}\alpha_2 + \frac{a}{m}\alpha_3 )+ y_{20}( \frac{h_1}{m}\alpha_2 + \alpha_1 ).
    \end{cases}
    \end{equation*}

\end{itemize}
\end{proof}

\subsection{Abnormal Extremals}

Abnormal extremals are trajectories of the Hamiltonian system with the Hamiltonian $H = \frac{1}{2} (-h_1^2 + h_2^2 + h_3^2)$, satisfying the conditions $H = 0$, $h_1 < 0$.

According to (\ref{vert_obsh}), the vertical part of the Hamiltonian system is:

    \begin{equation*}
    \begin{cases}
    \dot{h}_1 = 0,\\
    \dot{h}_2 = -2h_1h_3,\\
    \dot{h}_3 = 2h_1h_2.
    \end{cases}
    \end{equation*}

    We introduce a parameterization on the cone $\{H = 0, \ h_1 < 0\}$:

    \begin{align*}
    & -h_1^2 + h_2^2 + h_3^2 = 0,\\
    & h_1 = -r, \ h_2 = r\cos{\varphi}, \ h_3 = r\sin{\varphi},\ r > 0.
    \end{align*}

    \begin{equation}
    \label{vert_anorm}
    \begin{cases}
    \dot{h}_1 = -\dot{r},\\
    \dot{h}_2 = \dot{r}\cos{\varphi} - \dot{\varphi}r\sin{\varphi},\\
    \dot{h}_3 = \dot{r}\sin{\varphi} + \dot{\varphi}r \cos{\varphi}.
    \end{cases}
    \end{equation}

    \begin{equation}
    \label{anorm_comb}
    2h_1h_3 = -2r^2\sin{\varphi},\quad 2h_1h_2 = -2r^2\cos{\varphi}.
    \end{equation}

    Using (\ref{vert_obsh}), (\ref{vert_anorm}) and (\ref{anorm_comb}), the system is rewritten as follows:

    \begin{equation*}
    \begin{cases}
    \dot{h}_1 = -\dot{r} = 0,\\
    \dot{h}_2 = \dot{r}\cos{\varphi} - \dot{\varphi}r\sin{\varphi} = -2h_1h_3 = 2r^2\sin{\varphi},\\
    \dot{h}_3 = \dot{r}\sin{\varphi} + \dot{\varphi}r \cos{\varphi} = 2h_1h_2 = -2r^2\cos{\varphi}.
    \end{cases}
    \end{equation*}

    We obtain:

    \begin{equation*}
    \begin{cases}
    r \equiv \const = r_0 > 0,\\
    \varphi = -2r_0 t + \varphi_0.
    \end{cases}
    \end{equation*}

    Thus,

    $$
    h_1 = -r_0,\ h_2 = r_0\cos{(-2r_0t + \varphi_0)}, \ h_3 = r_0\sin{(-2r_0t + \varphi_0)},\ r_0 > 0.
    $$

    Let us write out the horizontal system:

    \begin{equation*}
    \dot{q} = -\frac{i}{2} q\left( -h_1\sigma_1 + h_2\sigma_2 + h_3\sigma_3 \right) \Leftrightarrow 
    \begin{cases}
    \dot{z}_1 = \left( -\frac{i}{2} \right)(z_1h_3 + z_2(-h_1 +ih_2)),\\
    \dot{z}_2 = \left( -\frac{i}{2} \right)(-z_1(h_1 + ih_2) - z_2h_3).
    \end{cases}
    \end{equation*}

    \begin{equation*}
    \begin{cases}
    \dot{x}_1 = \frac{r_0}{2}\left( y_1\sin{(-2r_0t+\varphi_0)} + x_2\cos{(-2r_0t+\varphi_0)} + y_2 \right),\\
    \dot{y}_1 = \frac{r_0}{2}\left( -x_1\sin{(-2r_0t+\varphi_0)} - x_2 + y_2\cos{(-2r_0t+\varphi_0)} \right),\\
    \dot{x}_2 = \frac{r_0}{2}\left( -x_1\cos{(-2r_0t+\varphi_0)} + y_1 -y_2\sin{(-2r_0t+\varphi_0)} \right),\\
    \dot{y}_2 = \frac{r_0}{2}\left( -x_1 - y_1\cos{(-2r_0t+\varphi_0)} +  x_2\sin{(-2r_0t+\varphi_0)} \right).
    \end{cases}
    \end{equation*}

    \begin{theorem}
    \label{anorm_traj_form}
    Abnormal extremals with initial conditions $(r_0,\varphi_0,x_{10},y_{10},x_{20},y_{20})$ are given by the following formulas:
    \begin{equation}
    \label{anorm_traj}
    \begin{cases}
    x_1(t) = x_{10}(\alpha_1 + \frac{h_1}{m}\alpha_2) + y_{10}(\frac{a}{m}\alpha_3 + \frac{c}{m}\alpha_2) + x_{20}(\frac{c}{m}\alpha_3 - \frac{a}{m}\alpha_2) + y_{20}(\frac{h_1}{m}\alpha_3 - \alpha_4),\\
    y_1(t) = -x_{10}(\frac{a}{m}\alpha_3 + \frac{c}{m}\alpha_2) + y_{10}(\alpha_1 + \frac{h_1}{m}\alpha_2) + x_{20}(-\frac{h_1}{m}\alpha_3 + \alpha_4) + y_{20}(\frac{c}{m}\alpha_3 - \frac{a}{m}\alpha_2),\\
    x_2(t) = x_{10}( \frac{a}{m}\alpha_2 - \frac{c}{m}\alpha_3 ) + y_{10}( -\alpha_4 + \frac{h_1}{m}\alpha_3 ) + x_{20}( \frac{h_1}{m}\alpha_2 + \alpha_1 ) + y_{20}( -\frac{c}{m}\alpha_2 - \frac{a}{m}\alpha_3 ),\\
    y_2(t) = x_{10}( \alpha_4 - \frac{h_1}{m}\alpha_3 ) + y_{10}( \frac{a}{m}\alpha_2 - \frac{c}{m}\alpha_3 ) + x_{20}(\frac{c}{m}\alpha_2 + \frac{a}{m}\alpha_3 )+ y_{20}( \frac{h_1}{m}\alpha_2 + \alpha_1 ),
    \end{cases}
    \end{equation}

    where 

    \begin{align*}
    & \alpha_1 = \cos{(h_1t)} \cos{\left( \frac{t}{2}m \right)}, \ \alpha_2 = \sin{(h_1t)} \sin{\left( \frac{t}{2}m \right)}, \ \alpha_3 = \cos{(h_1t)}\sin{\left( \frac{t}{2}m \right)},\ \alpha_4 = \sin{\left( h_1 t \right)}\cos{\left( \frac{t}{2}m \right)},\\
    & h_1 = -r_0 < 0,\ m = r_0\sqrt{2},\ a = r_0\sin{\varphi_0},\ c = r_0 \cos{\varphi_0}.
    \end{align*}
    \end{theorem}

\begin{proof}
We use the following method for solving a non-autonomous system of differential equations from \cite{notes}, analogously to how we did it when deriving the normal extremals in Theorem \ref{norm_traj_form}:

    \begin{enumerate}
\item[$(a)$]
Write the horizontal subsystem of the Hamiltonian system as $\dot q = V_t(q) := - h_1 X_1 + h_2 X_2 + h_3 X_3$
\item[$(b)$]
Prove that $\frac{d}{dt} V_t = \ad f (V_t)$ for some field $f$.
\item[$(c)$]
Conclude from this that $V_t = e^{t \ad f} g$ for some field $g$.
\item[$(d)$]
Use formula (19.6) from \cite{notes}
$$
q(t) = \stackrel{\longrightarrow}{\exp} \int_0^t V_{\tau} dt (q_0) 
= \stackrel{\longrightarrow}{\exp} \int_0^t e^{\tau \ad f} g dt (q_0)
= e^{-tf} \circ e^{t(f+g)} (q_0).
$$
\end{enumerate}

\begin{enumerate}
    \item[$(a)$]
    \begin{align*}
    & \dot{q} = V_t(q):= -h_1X_1 + h_2X_2 + h_3X_3 = -\frac{i}{2}q(-h_1\sigma_1 + h_2\sigma_2 + h_3\sigma_3) = \\
    & = -\frac{ir_0}{2}q\left( \sigma_1 + \cos{(-2r_0t+\varphi_0)}\sigma_2 + \sin{(-2r_0t+\varphi_0)}\sigma_3 \right).
    \end{align*}
    \item[$(b)$] 
    $$
    \frac{d}{dt}V_t = -\frac{ir_0}{2}q\left( 2r_0\sin{(-2r_0t+\varphi_0)}\sigma_2 - 2r_0\cos{(-2r_0t+\varphi_0)}\sigma_3 \right) = -2h_1h_3X_2 + 2h_1h_2X_3.
    $$

    Let $f = 2h_1X_1$,
    $$
    [f,V_t] = [2h_1X_1,-h_1X_1 + h_2X_2 + h_3X_3] = 2h_1h_2[X_1,X_2] + 2h_1h_3[X_1,X_3] = 2h_1h_2X_3 - 2h_1h_3X_2 = \frac{d}{dt}V_t.
    $$
    \item[$(c)$] 
    We need to find $g$ such that $V_t = e^{t\ad{f}}g$.

    Substitute $t = 0$, then $g = V_0$:

    $$
    V_0 = -h_1\vert_{t = 0}X_1 + h_2\vert_{t = 0}X_2 + h_3\vert_{t = 0}X_3 = r_0X_1 + r_0\cos{\varphi_0}X_2 + r_0\sin{\varphi_0}X_3 = g.
    $$
    
    \item[$(d)$] 

    $$
    q(t) = \stackrel{\longrightarrow}{\exp} \int_0^t V_{\tau} dt (q_0) 
    = \stackrel{\longrightarrow}{\exp} \int_0^t e^{\tau \ad f} g dt (q_0)
    = e^{-tf} \circ e^{t(f+g)} (q_0).
    $$

    Let us compute the flows of the fields $-f$ and $f+g$, and then their composition.

    First, $-f$:

    \begin{align*}
    & \dot{q} = -2h_1X_1 = -\frac{i}{2}2r_0\sigma_1 = -ir_0\sigma_1 \Rightarrow q(t) = q_0e^{ih_1t\sigma_1}.
    \end{align*}

    Then:

    \begin{align*}
    & e^{-tf}(q_0) = \left[ E + (-ir_0)t\sigma_1 + (-ir_0)^2\frac{t^2}{2}\sigma_1^2 + (-ir_0)^3\frac{t^3}{3!}\sigma_1^3 + (-ir_0)^4 \frac{t^4}{4!}\sigma_1^4 + \right. ... + \\
    & \left. + (-ir_0)^{4k}\frac{t^{4k}}{(4k)!}\sigma_1^{4k} + (-ir_0)^{4k+1}\frac{t^{4k+1}}{(4k+1)!}\sigma_1^{4k+1} + (-ir_0)^{4k+2}\frac{t^{4k+2}}{(4k+2)!}\sigma_1^{4k+2} + (-ir_0)^{4k+3}\frac{t^{4k+3}}{(4k+3)!}\sigma_1^{4k+3} + ... \right] (q_0) = \\
    & = \left[ \cos{(tr_0)}E - i\sin{(tr_0)}\sigma_1 \right](q_0) = q_0\begin{pmatrix}
    \cos{(tr_0)} & -i\sin{(tr_0)}\\
    -i\sin{(tr_0)} & \cos{(tr_0)}
    \end{pmatrix} = q_0\begin{pmatrix}
        \cos{(h_1t)} & i\sin{(h_1t)} \\
        i\sin{(h_1t)} & \cos{(h_1t)}
    \end{pmatrix},
    \end{align*}

    since

    \begin{align*}
    & \sigma_1 = \begin{pmatrix}
    0 & 1 \\
    1 & 0
    \end{pmatrix},\ \sigma_1^2 = \begin{pmatrix}
    0 & 1\\
    1 & 0
    \end{pmatrix}\begin{pmatrix}
    0 & 1\\
    1 & 0
    \end{pmatrix} = \begin{pmatrix}
    1 & 0\\
    0 & 1
    \end{pmatrix} \Rightarrow \sigma_1^{2k-1} = \begin{pmatrix}
    0 & 1\\
    1 & 0
    \end{pmatrix},\ \sigma_1^{2k} = \begin{pmatrix}
    1 & 0\\
    0 & 1
    \end{pmatrix},\ k \in \mathbb{N}.
    \end{align*}

    Now $f+g$:

    \begin{align*}
    & f+g = -2r_0X_1 + r_0X_1 + r_0\cos{\varphi_0}X_2 + r_0\sin{\varphi_0}X_3 = \frac{-ir_0}{2}\left( -\sigma_1 + \cos{\varphi_0}\sigma_2 + \sin{\varphi_0}\sigma_3 \right) = \\
    & = \frac{-ir_0}{2}\left( \begin{pmatrix}
    0 & -1\\
    -1 & 0
    \end{pmatrix} + \cos{\varphi_0}\begin{pmatrix}
    0 & -i\\
    i & 0
    \end{pmatrix} + \sin{\varphi_0}\begin{pmatrix}
    1 & 0\\
    0 & -1
    \end{pmatrix} \right) = \frac{-ir_0}{2}\begin{pmatrix}
    \sin{\varphi_0} & -1 - i\cos{\varphi_0}\\
    -1 + i\cos{\varphi_0} & -\sin{\varphi_0}
    \end{pmatrix} = \frac{-ir_0}{2}A.
    \end{align*}

    Since

    $$
    A^2 = \begin{pmatrix}
    \sin^2{\varphi_0} + 1 + \cos^2{\varphi_0} & -\sin{\varphi_0} - i\sin{\varphi_0}\cos{\varphi_0} + \sin{\varphi_0} + i\sin{\varphi_0}\cos{\varphi_0} \\
    -\sin{\varphi_0} + i\sin{\varphi_0}\cos{\varphi_0} + \sin{\varphi_0} - i\sin{\varphi_0}\cos{\varphi_0} & 1 + \cos^2{\varphi_0} + \sin^2{\varphi_0}
    \end{pmatrix} = 2E,
    $$

    then:

    \begin{align*}
    & e^{t(f+g)}(q_0) = \left[E - i\frac{r_0t}{2}A - \frac{(r_0t)^2}{2^2 2!}2E + i\frac{(r_0t)^3}{2^3 3!}2A + \frac{(r_0t)^4}{2^4 4!}4E + ... + \frac{(r_0t)^{4k}}{2^{4k} (4k)!}2^{k}E - i\frac{(r_0t)^{4k+1}}{2^{4k+1} (4k+1)!}2^{k}A - \right. \\
    & \left. - \frac{(r_0t)^{4k+2}}{2^{4k+2} (4k+2)!}2^{k+1}E + i\frac{(r_0t)^{4k+3}}{2^{4k+3} (4k+3)!}2^{k+1}A + ... \right](q_0)  = \left[ \cos{\left( \frac{r_0t}{\sqrt{2}} \right)}E - \frac{i}{\sqrt{2}}\sin{\left( \frac{r_0t}{\sqrt{2}} \right)}A \right](q_0) = \\
    & = q_0 \begin{pmatrix}
    \cos{\left( \frac{r_0t}{\sqrt{2}} \right)}-\frac{i}{\sqrt{2}}\sin{\varphi_0}\sin{\left( \frac{r_0t}{\sqrt{2}} \right)} & \frac{\sin{\left( \frac{r_0t}{\sqrt{2}} \right)}}{\sqrt{2}}\left( -\cos{\varphi_0} + i \right) \\
    \frac{\sin{\left( \frac{r_0t}{\sqrt{2}} \right)}}{\sqrt{2}}\left( \cos{\varphi_0} + i \right) & \cos{\left( \frac{r_0t}{\sqrt{2}} \right)} + \frac{i}{\sqrt{2}}\sin{\varphi_0}\sin{\left( \frac{r_0t}{\sqrt{2}} \right)}
    \end{pmatrix} = \\
    & = q_0\begin{pmatrix}
    \cos{(mt/2)} - ia/m\sin{(mt/2)} & \frac{-c - ih_1}{m}\sin{(mt/2)}\\
    \frac{c - ih_1}{m}\sin{(mt/2)} & \cos{(mt/2)} + i\frac{a}{m}\sin{(mt/2)}
    \end{pmatrix},
    \end{align*}

    where

    $$
    m = r_0\sqrt{2}, \ a = r_0\sin{\varphi_0},\ c = r_0\cos{\varphi_0}.
    $$

    Composition:

    \begin{align*}
    & e^{-tf} \circ e^{t(f+g)}(q_0) = q_0\begin{pmatrix}
    \cos{(mt/2)} - ia/m\sin{(mt/2)} & \frac{-c - ih_1}{m}\sin{(mt/2)}\\
    \frac{c - ih_1}{m}\sin{(mt/2)} & \cos{(mt/2)} + i\frac{a}{m}\sin{(mt/2)}
    \end{pmatrix}\begin{pmatrix}
        \cos{(h_1t)} & i\sin{(h_1t)} \\
        i\sin{(h_1t)} & \cos{(h_1t)}
    \end{pmatrix} = \\
    & = q_0\begin{pmatrix}
    \alpha_1+h_1/m\alpha_2+i(-a/m\alpha_3-c/m\alpha_2) & a/m\alpha_2-c/m\alpha_3+i(\alpha_4-h_1/m\alpha_3)\\
    c/m\alpha_3 - a/m\alpha_2 + i(-h_1/m\alpha_3 + \alpha_4) & h_1/m\alpha_2 + \alpha_1 + i(c/m\alpha_2 + a/m\alpha_3)
    \end{pmatrix} = \\
    & = \begin{pmatrix}
        z_{10} & z_{20} \\
        -\bar{z}_{20} & \bar{z}_{10}
    \end{pmatrix}\begin{pmatrix}
    b_1 & b_2\\
    -\bar{b}_2 & \bar{b}_1
    \end{pmatrix} = \begin{pmatrix}
    z_{10}b_1 - z_{20}\bar{b}_2 & z_{10}b_2 + z_{20}\bar{b}_1 \\
    -b_1\bar{z}_{20} - b_2\bar{z}_{10} & -b_2\bar{z}_{20} + \bar{b}_1\bar{z}_{10}
    \end{pmatrix},
    \end{align*}

    where
    
    \begin{align*}
    & z_{10}b_1 - z_{20}\bar{b}_2 = (x_{10}+iy_{10})(\alpha_1+h_1/m\alpha_2+i(-a/m\alpha_3-c/m\alpha_2)) - (x_{20}+iy_{20})(a/m\alpha_2-c/m\alpha_3-i(\alpha_4-h_1/m\alpha_3)) = \\
    & = x_{10}(\alpha_1 + h_1/m\alpha_2) + y_{10}(a/m\alpha_3 + c/m\alpha_2)+x_{20}(c/m\alpha_3 - a/m\alpha_2) + y_{20}(h_1/m\alpha_3 - \alpha_4) + \\
    & + i( x_{10}(-a/m\alpha_3 - c/m\alpha_2) + y_{10}(\alpha_1 + h_1/m\alpha_2) + x_{20}(\alpha_4 - h_1/m\alpha_3) + y_{20}(c/m\alpha_3 - a/m\alpha_2) ).
    \end{align*}

    \begin{align*}
    & z_{10}b_2 + z_{20}\bar{b}_1 = (x_{10}+iy_{10})(a/m\alpha_2-c/m\alpha_3+i(\alpha_4-h_1/m\alpha_3)) + (x_{20}+iy_{20})(h_1/m\alpha_2 + \alpha_1 + i(c/m\alpha_2 + a/m\alpha_3)) = \\
    & = x_{10}(a/m\alpha_2 - c/m\alpha_3)+y_{10}(h_1/m\alpha_3-\alpha_4) + x_{20}(h_1/m\alpha_2 + \alpha_1) + y_{20}(-c/m\alpha_2 - a/m\alpha_3) + \\
    & + i(x_{10}(\alpha_4 - h_1/m\alpha_3) + y_{10}(a/m\alpha_2 - c/m\alpha_3) + x_{20}(c/m\alpha_2 + a/m\alpha_3) + y_{20}(h_1/m\alpha_2 + \alpha_1)).
    \end{align*}
    
\end{enumerate}

    Thus, the solution is:

    \begin{equation*}
    \begin{cases}
    x_1(t) = x_{10}(\alpha_1 + \frac{h_1}{m}\alpha_2) + y_{10}(\frac{a}{m}\alpha_3 + \frac{c}{m}\alpha_2) + x_{20}(\frac{c}{m}\alpha_3 - \frac{a}{m}\alpha_2) + y_{20}(\frac{h_1}{m}\alpha_3 - \alpha_4),\\
    y_1(t) = -x_{10}(\frac{a}{m}\alpha_3 + \frac{c}{m}\alpha_2) + y_{10}(\alpha_1 + \frac{h_1}{m}\alpha_2) + x_{20}(-\frac{h_1}{m}\alpha_3 + \alpha_4) + y_{20}(\frac{c}{m}\alpha_3 - \frac{a}{m}\alpha_2),\\
    x_2(t) = x_{10}( \frac{a}{m}\alpha_2 - \frac{c}{m}\alpha_3 ) + y_{10}( -\alpha_4 + \frac{h_1}{m}\alpha_3 ) + x_{20}( \frac{h_1}{m}\alpha_2 + \alpha_1 ) + y_{20}( -\frac{c}{m}\alpha_2 - \frac{a}{m}\alpha_3 ),\\
    y_2(t) = x_{10}( \alpha_4 - \frac{h_1}{m}\alpha_3 ) + y_{10}( \frac{a}{m}\alpha_2 - \frac{c}{m}\alpha_3 ) + x_{20}(\frac{c}{m}\alpha_2 + \frac{a}{m}\alpha_3 )+ y_{20}( \frac{h_1}{m}\alpha_2 + \alpha_1 ),
    \end{cases}
    \end{equation*}

    where 

    \begin{align*}
    & \alpha_1 = \cos{(h_1t)} \cos{\left( \frac{t}{2}m \right)}, \ \alpha_2 = \sin{(h_1t)} \sin{\left( \frac{t}{2}m \right)}, \ \alpha_3 = \cos{(h_1t)}\sin{\left( \frac{t}{2}m \right)},\ \alpha_4 = \sin{\left( h_1 t \right)}\cos{\left( \frac{t}{2}m \right)},\\
    & h_1 = -r_0 < 0,\ m = r_0\sqrt{2},\ a = r_0\sin{\varphi_0},\ c = r_0 \cos{\varphi_0}.
    \end{align*}
\end{proof}

\subsection{Conjugate Points}

Let us write the normal extremal trajectories for the initial point $x_{10} = 1$, $y_{10} = 0$, $x_{20} = 0$, $y_{20} = 0$:

$$
\begin{cases}
    x_1(t) = \alpha_1 + \frac{h_1}{m}\alpha_2 = \cos{(h_1t)} \cos{\left( \frac{t}{2}m \right)} + \frac{h_1}{m}\sin{(h_1t)} \sin{\left( \frac{t}{2}m \right)},\\
    y_1(t) = -\frac{a}{m}\alpha_3 - \frac{c}{m}\alpha_2 = - \frac{a}{m}\cos{(h_1t)}\sin{\left( \frac{t}{2}m \right)} - \frac{c}{m}\sin{(h_1t)} \sin{\left( \frac{t}{2}m \right)},\\
    x_2(t) = \frac{a}{m}\alpha_2 - \frac{c}{m}\alpha_3 =  \frac{a}{m}\sin{(h_1t)} \sin{\left( \frac{t}{2}m \right)} - \frac{c}{m}\cos{(h_1t)}\sin{\left( \frac{t}{2}m \right)},\\
    y_2(t) = \alpha_4 - \frac{h_1}{m}\alpha_3 = \sin{\left( h_1 t \right)}\cos{\left( \frac{m}{2}t \right)} - \frac{h_1}{m}\cos{(h_1t)}\sin{\left( \frac{t}{2}m \right)}.
\end{cases}
$$

$$
\begin{cases}
x_1(t,\theta_0,\varphi_0) = \cos{(\ch{\theta_0}t)} \cos{\left( \frac{t}{2}\sqrt{\ch{(2\theta_0)}} \right)} + \frac{\ch{\theta_0}}{\sqrt{\ch{(2\theta_0)}}}\sin{(\ch{\theta_0}t)} \sin{\left( \frac{t}{2}\sqrt{\ch{(2\theta_0)}} \right)},\\
y_1(t, \theta_0, \varphi_0) = - \frac{\sin{\varphi_0}\sh{\theta_0}}{\sqrt{\ch{(2\theta_0)}}}\cos{(\ch{\theta_0}t)}\sin{\left( \frac{t}{2}\sqrt{\ch{(2\theta_0)}} \right)} + \frac{\cos{\varphi_0}\sh{\theta_0}}{\sqrt{\ch{(2\theta_0)}}}\sin{(\ch{\theta_0}t)} \sin{\left( \frac{t}{2}\sqrt{\ch{(2\theta_0)}} \right)},\\
x_2(t, \theta_0, \varphi_0) = -\frac{\sin{\varphi_0}\sh{\theta_0}}{\sqrt{\ch{(2\theta_0)}}}\sin{(\ch{\theta_0}t)} \sin{\left( \frac{t}{2}\sqrt{\ch{(2\theta_0)}} \right)} - \frac{\cos{\varphi_0}\sh{\theta_0}}{\sqrt{\ch{(2\theta_0)}}}\cos{(\ch{\theta_0}t)}\sin{\left( \frac{t}{2}\sqrt{\ch{(2\theta_0)}} \right)},\\
y_2(t, \theta_0, \varphi_0) = -\sin{\left( \ch{\theta_0} t \right)}\cos{\left( \frac{\sqrt{\ch{(2\theta_0)}}}{2}t \right)} + \frac{\ch{\theta_0}}{\sqrt{\ch{(2\theta_0)}}}\cos{(\ch{\theta_0}t)}\sin{\left( \frac{t}{2}\sqrt{\ch{(2\theta_0)}} \right)}.
\end{cases}
$$

Let us compute the derivatives with respect to $t$, $\varphi_0$ and $\theta_0$.

\begin{align*}
& \partial_t x_1 = \left( -h_1 + \frac{h_1}{2} \right)\sin{(h_1 t)}\cos{\left( \frac{t}{2}m \right)} + \left( - \frac{m}{2} + \frac{h_1^2}{m} \right)\cos{(h_1t)}\sin{\left( \frac{t}{2}m \right)} = -\frac{h_1}{2}\alpha_4(t) + \left( \frac{h_1^2}{m} - \frac{m}{2} \right)\alpha_3(t) ,\\
& \partial_t y_1 = \frac{ah_1}{m}\sin{(h_1t)}\sin{\left( \frac{t}{2}m \right)} - \frac{a}{2}\cos{(h_1t)}\cos{\left( \frac{t}{2}m \right)} - \frac{ch_1}{m}\cos{(h_1t)}\sin{\left( \frac{t}{2}m \right)} - \frac{c}{2}\sin{(h_1t)}\cos{\left( \frac{t}{2}m \right)} = \\
& = -\frac{a}{2}\alpha_1(t) + \frac{ah_1}{m}\alpha_2(t) - \frac{ch_1}{m}\alpha_3(t) - \frac{c}{2}\alpha_4(t),\\
& \partial_t x_2 = 
\frac{ah_1}{m}\cos{(h_1t)}\sin{\left( \frac{t}{2}m \right)} + \frac{a}{2}\sin{(h_1t)}\cos{\left( \frac{t}{2}m \right)} + \frac{ch_1}{m}\sin{(h_1 t)}\sin{\left( \frac{t}{2}m \right)} - \frac{c}{2}\cos{(h_1t)}\cos{\left( \frac{t}{2}m \right)} = \\
& = -\frac{c}{2}\alpha_1(t) + \frac{ch_1}{m}\alpha_2(t) + \frac{ah_1}{m}\alpha_3(t) + \frac{a}{2}\alpha_4(t),\\
& \partial_t y_2 = \left( h_1 - \frac{h_1}{2}\right)\cos{(h_1t)}\cos{\left( \frac{t}{2}m \right)} + \left( - \frac{m}{2} + \frac{h_1^2}{m} \right)\sin{(h_1t)}\sin{\left( \frac{t}{2}m \right)} = \frac{h_1}{2}\alpha_1(t) + \left( \frac{h_1^2}{m} - \frac{m}{2} \right)\alpha_2(t).
\end{align*}

$$
\begin{cases}
\partial_{\varphi_0}x_1 = 0,\\
\partial_{\varphi_0}y_1 = - \frac{\cos{\varphi_0}\sh{\theta_0}}{\sqrt{\ch{(2\theta_0)}}}\cos{(\ch{\theta_0}t)}\sin{\left( \frac{t}{2}\sqrt{\ch{(2\theta_0)}} \right)} - \frac{\sin{\varphi_0}\sh{\theta_0}}{\sqrt{\ch{(2\theta_0)}}}\sin{(\ch{\theta_0}t)} \sin{\left( \frac{t}{2}\sqrt{\ch{(2\theta_0)}} \right)} = - \frac{c}{m}\alpha_3(t) + \frac{a}{m}\alpha_2(t),\\
\partial_{\varphi_0}x_2 = -\frac{\cos{\varphi_0}\sh{\theta_0}}{\sqrt{\ch{(2\theta_0)}}}\sin{(\ch{\theta_0}t)} \sin{\left( \frac{t}{2}\sqrt{\ch{(2\theta_0)}} \right)} + \frac{\sin{\varphi_0}\sh{\theta_0}}{m}\cos{(\ch{\theta_0}t)}\sin{\left( \frac{t}{2}m \right)} = \frac{c}{m}\alpha_2(t) + \frac{a}{m}\alpha_3(t),\\
\partial_{\varphi_0}y_2 = 0.
\end{cases}
$$

\begin{align*}
\partial_{\theta_0}x_1 &  = -t\frac{\sh^3{\theta_0}}{\ch{(2\theta_0)}}\sin{(\ch{\theta_0}t)} \cos{\left( \frac{t}{2}\sqrt{\ch{(2\theta_0)}} \right)} - \frac{\sh{\theta_0}}{\ch^{3/2}{(2\theta_0)}}\sin{(\ch{\theta_0}t)} \sin{\left( \frac{t}{2}\sqrt{\ch{(2\theta_0)}} \right)} = \\
& = t\frac{\sh^3{\theta_0}}{m^2}\alpha_4(t) + \frac{\sh{\theta_0}}{m^3}\alpha_2(t);
\end{align*}

\begin{align*}
\partial_{\theta_0}y_1
& = \left( - \frac{\sin{\varphi_0}\ch{\theta_0}}{\ch^{3/2}{(2\theta_0)}} + t\frac{\cos{\varphi_0}\sh^2{\theta_0}}{\sqrt{\ch{(2\theta_0)}}} \right)\cos{(\ch{\theta_0}t)}\sin{\left( \frac{t}{2}\sqrt{\ch{(2\theta_0)}} \right)} + \\
& + \left( t\frac{\sin{\varphi_0}\sh^2{\theta_0}}{\sqrt{\ch{(2\theta_0)}}} + \frac{\cos{\varphi_0}\ch{\theta_0}}{\ch^{3/2}{(2\theta_0)}} \right)\sin{(\ch{\theta_0}t)} \sin{\left( \frac{t}{2}\sqrt{\ch{(2\theta_0)}} \right)} -\\
& - \frac{\sin{\varphi_0}\sh{\theta_0}}{\sqrt{\ch{(2\theta_0)}}}\cos{(\ch{\theta_0}t)}\cos{\left( \frac{t}{2}\sqrt{\ch{(2\theta_0)}} \right)} \frac{t\sh{(2\theta_0)}}{2\sqrt{\ch{(2\theta_0)}}} + \frac{\cos{\varphi_0}\sh{\theta_0}}{\sqrt{\ch{(2\theta_0)}}}\sin{(\ch{\theta_0}t)} \cos{\left( \frac{t}{2}\sqrt{\ch{(2\theta_0)}} \right)}\frac{t\sh{(2\theta_0)}}{2\sqrt{\ch{(2\theta_0)}}} = \\
& = - t\frac{a \th{(2\theta_0)}}{2}\alpha_1(t) - \left( t\frac{a\sh{\theta_0}}{m} - \frac{\cos{\varphi_0}h_1}{m^3} \right)\alpha_2(t) + \left(  \frac{\sin{\varphi_0}h_1}{m^3} + t\frac{c\sh{\theta_0}}{m} \right)\alpha_3(t) - t\frac{c \th{(2\theta_0)}}{2}\alpha_4(t);
\end{align*}

\begin{align*}
\partial_{\theta_0}x_2 &  = \left( -\frac{\sin{\varphi_0}\ch{\theta_0}}{\ch^{3/2}{(2\theta_0)}} + t\frac{\cos{\varphi_0}\sh^2{\theta_0}}{\sqrt{\ch{(2\theta_0)}}} \right)\sin{(\ch{\theta_0}t)} \sin{\left( \frac{t}{2}\sqrt{\ch{(2\theta_0)}} \right)} - \\
& - \left( t \frac{\sin{\varphi_0}\sh^2{\theta_0}}{\sqrt{\ch{(2\theta_0)}}} + \frac{\cos{\varphi_0}\ch{\theta_0}}{\ch^{3/2}{(2\theta_0)}} \right)\cos{(\ch{\theta_0}t)}\sin{\left( \frac{t}{2}\sqrt{\ch{(2\theta_0)}} \right)} -\\
& - \frac{\sin{\varphi_0}\sh{\theta_0}}{\sqrt{\ch{(2\theta_0)}}}\sin{(\ch{\theta_0}t)} \cos{\left( \frac{t}{2}\sqrt{\ch{(2\theta_0)}} \right)}\frac{t\sh{(2\theta_0)}}{2\sqrt{\ch{(2\theta_0)}}} - \frac{\cos{\varphi_0}\sh{\theta_0}}{\sqrt{\ch{(2\theta_0)}}}\cos{(\ch{\theta_0}t)}\cos{\left( \frac{t}{2}\sqrt{\ch{(2\theta_0)}} \right)}\frac{t\sh{(2\theta_0)}}{2\sqrt{\ch{(2\theta_0)}}} = \\
& = - t\frac{c \th{(2\theta_0)}}{2}\alpha_1(t) + \left( -\frac{\sin{\varphi_0}h_1}{m^3} - t\frac{c\sh{\theta_0}}{m} \right)\alpha_2(t) - \left( t \frac{a\sh{\theta_0}}{m} - \frac{\cos{\varphi_0}h_1}{m^3} \right)\alpha_3(t) + t\frac{a \th{(2\theta_0)}}{2}\alpha_4(t);
\end{align*}

\begin{align*}
\partial_{\theta_0}y_2 & = \left( -t\sh{\theta_0} + \frac{t\ch{\theta_0}\tanh{(2\theta_0)}}{2} \right)\cos{(t\ch{\theta_0})}\cos{\left( \frac{t}{2}\sqrt{\ch{(2\theta_0)}} \right)} - \frac{\sh{\theta_0}}{\ch^{3/2}{(2\theta_0)}}\cos{(\ch{\theta_0}t)} \sin{\left( \frac{t}{2}\sqrt{\ch{(2\theta_0)}} \right)} + \\
& + \left( \frac{t\sh{(2\theta_0)}}{2\sqrt{\ch{(2\theta_0)}}} - \frac{t\sh{\theta_0}\ch{\theta_0}}{\sqrt{\ch{(2\theta_0)}}}\right)\sin{(t\ch{\theta_0})}\sin{\left( \frac{t}{2}\sqrt{\ch{(2\theta_0)}} \right)} = \\
& = -t\frac{\sh^3{\theta_0}}{\ch{(2\theta_0)}}\cos{(t\ch{\theta_0})}\cos{\left( \frac{t}{2}\sqrt{\ch{(2\theta_0)}} \right)}- \frac{\sh{\theta_0}}{\ch^{3/2}{(2\theta_0)}}\cos{(\ch{\theta_0}t)} \sin{\left( \frac{t}{2}\sqrt{\ch{(2\theta_0)}} \right)} = \\
& = -t\frac{\sh^3{\theta_0}}{\ch{(2\theta_0)}}\alpha_1(t) - \frac{\sh{\theta_0}}{\ch^{3/2}{(2\theta_0)}}\alpha_3(t).
\end{align*}

Let us compute the determinant of the Jacobian matrix (since we are interested in the chart containing the group identity $x_1 = 1$, $y_1 = x_2 = y_2 = 0$):

\begin{equation*}
\begin{pmatrix}
\partial_t y_1 & \partial_{\theta_0} y_1 & \partial_{\varphi_0} y_1 \\
\partial_t x_2 & \partial_{\theta_0} x_2 & \partial_{\varphi_0} x_2 \\
\partial_t y_2 & \partial_{\theta_0} y_2 & 0 \\
\end{pmatrix}.
\end{equation*}

\begin{proposition}
\label{det_vych}
The determinant of the Jacobian matrix of the exponential map $(t, \theta_0, \varphi_0) \rightarrow (y_1,x_2,y_2)$ equals
\begin{align*}
& \sin{\left( \frac{t}{2}\sqrt{\ch{(2\theta_0)}} \right)}\frac{\sh{\theta_0}}{2\ch^2{(2\theta_0)}\sqrt{\ch{(2\theta_0)}}}\left( \sqrt{\ch{(2\theta_0)}}\sh^2{\theta_0}t\cos{\left( \frac{t}{2}\sqrt{\ch{(2\theta_0)}} \right)} + \sin{\left( \frac{t}{2}\sqrt{\ch{(2\theta_0)}} \right)} \right) \cdot \\
& \cdot \left( \sqrt{\ch{(2\theta_0)}}\cos{(t\ch{\theta_0})}\cos{\left( \frac{t}{2}\sqrt{\ch{(2\theta_0)}} \right)} + \ch{\theta_0}\sin{(t\ch{\theta_0})}\sin{\left( \frac{t}{2}\sqrt{\ch{(2\theta_0)}} \right)} \right).
\end{align*}
\end{proposition}
\begin{proof}
Expand the determinant along the third row:

\begin{align*}
& \partial_t y_2 \begin{vmatrix}
\partial_{\theta_0}y_1 & \partial_{\varphi_0}y_1 \\
\partial_{\theta_0}x_2 & \partial_{\varphi_0}x_2
\end{vmatrix} - \partial_{\theta_0}y_2 \begin{vmatrix}
\partial_t y_1 & \partial_{\varphi_0} y_1 \\
\partial_{t}x_2 & \partial_{\varphi_0}x_2
\end{vmatrix} = \partial_t y_2 \left( \partial_{\theta_0}y_1 \partial_{\varphi_0}x_2 - \partial_{\theta_0}x_2 \partial_{\varphi_0}y_1 \right) - \partial_{\theta_0}y_2 \left( \partial_{t}y_1 \partial_{\varphi_0}x_2 - \partial_{t}x_2 \partial_{\varphi_0}y_1 \right) = \\
& = \left( \frac{h_1}{2}\alpha_1(t) + \left( \frac{h_1^2}{m} - \frac{m}{2} \right)\alpha_2(t) \right)\begin{vmatrix}
\partial_{\theta_0}y_1 & - \frac{c}{m}\alpha_3(t) + \frac{a}{m}\alpha_2(t) \\
\partial_{\theta_0}x_2 & \frac{c}{m}\alpha_2(t) + \frac{a}{m}\alpha_3(t)
\end{vmatrix} + \left( t\frac{\sh^3{\theta_0}}{m^2}\alpha_1(t) - \frac{\sh{\theta_0}}{m^3}\alpha_3(t) \right) \begin{vmatrix}
\partial_{t}y_1 & - \frac{c}{m}\alpha_3(t) + \frac{a}{m}\alpha_2(t) \\
\partial_{t}x_2 & \frac{c}{m}\alpha_2(t) + \frac{a}{m}\alpha_3(t)
\end{vmatrix} = \\
& = \left( \frac{h_1}{2}\alpha_1(t) + \left( \frac{h_1^2}{m} - \frac{m}{2} \right)\alpha_2(t) \right) \left( \left( \frac{c}{m}\partial_{\theta_0}y_1 - \frac{a}{m}\partial_{\theta_0}x_2 \right)\alpha_2(t) + \left( \frac{a}{m}\partial_{\theta_0}y_1 + \frac{c}{m}\partial_{\theta_0}x_2 \right) \alpha_3(t) \right) + \\
& + \left( t\frac{\sh^3{\theta_0}}{m^2}\alpha_1(t) - \frac{\sh{\theta_0}}{m^3}\alpha_3(t) \right)\left( \left( \frac{c}{m}\partial_t y_1 - \frac{a}{m}\partial_t x_2 \right)\alpha_2(t) + \left( \frac{a}{m}\partial_t y_1 + \frac{c}{m}\partial_t x_2 \right)\alpha_3(t) \right) = \\
& = \left( \frac{h_1}{2}\left( \frac{c}{m}\partial_{\theta_0}y_1 - \frac{a}{m}\partial_{\theta_0}x_2 \right) + t\frac{\sh^3{\theta_0}}{m^2}\left( \frac{c}{m}\partial_t y_1 - \frac{a}{m}\partial_t x_2 \right) \right) \alpha_1(t)\alpha_2(t) + \\
& + \left( \frac{h_1}{2}\left( \frac{a}{m}\partial_{\theta_0}y_1 + \frac{c}{m}\partial_{\theta_0}x_2 \right) + t\frac{\sh^3{\theta_0}}{m^2}\left( \frac{a}{m}\partial_t y_1 + \frac{c}{m}\partial_t x_2 \right) \right)\alpha_1(t)\alpha_3(t) + \\
& + \left( \left( \frac{h_1^2}{m} - \frac{m}{2} \right)\left( \frac{c}{m}\partial_{\theta_0}y_1 - \frac{a}{m}\partial_{\theta_0}x_2 \right)  \right) \alpha_2(t)\alpha_2(t) + \\
& + \left( \left( \frac{h_1^2}{m} - \frac{m}{2} \right)\left( \frac{a}{m}\partial_{\theta_0}y_1 + \frac{c}{m}\partial_{\theta_0}x_2 \right) - \frac{\sh{\theta_0}}{m^3}\left( \frac{c}{m}\partial_t y_1 - \frac{a}{m}\partial_t x_2 \right) \right)\alpha_2(t)\alpha_3(t) + \\
& + \left( -\frac{\sh{\theta_0}}{m^3}\left( \frac{a}{m}\partial_ty_1 + \frac{c}{m}\partial_tx_2 \right) \right) \alpha_3(t)\alpha_3(t) = \\
& = \left[ \frac{h_1}{2}\left( \frac{c}{m}\left( \left(  \frac{\sin{\varphi_0}h_1}{m^3} + t\frac{c\sh{\theta_0}}{m} \right)\alpha_3(t) - \left( t\frac{a\sh{\theta_0}}{m} - \frac{\cos{\varphi_0}h_1}{m^3} \right)\alpha_2(t) - t\frac{a \th{(2\theta_0)}}{2}\alpha_1(t) - t\frac{c \th{(2\theta_0)}}{2}\alpha_4(t) \right) - \right. \right. \\
& - \left. \left. \frac{a}{m}\left( \left( -\frac{\sin{\varphi_0}h_1}{m^3} - t\frac{c\sh{\theta_0}}{m} \right)\alpha_2(t) - \left( t \frac{a\sh{\theta_0}}{m} - \frac{\cos{\varphi_0}h_1}{m^3} \right)\alpha_3(t) + t\frac{a \th{(2\theta_0)}}{2}\alpha_4(t) - t\frac{c \th{(2\theta_0)}}{2}\alpha_1(t) \right) \right) + \right. \\
& + \left. t\frac{\sh^3{\theta_0}}{m^2}\left( \frac{c}{m}\left( -\frac{a}{2}\alpha_1(t) + \frac{ah_1}{m}\alpha_2(t) - \frac{ch_1}{m}\alpha_3(t) - \frac{c}{2}\alpha_4(t) \right) - \frac{a}{m}\left( -\frac{c}{2}\alpha_1(t) + \frac{ch_1}{m}\alpha_2(t) + \frac{ah_1}{m}\alpha_3(t) + \frac{a}{2}\alpha_4(t) \right) \right) \right] \alpha_1(t)\alpha_2(t) + 
\end{align*}

\begin{align*}
& + \left[ \frac{h_1}{2}\left( \frac{a}{m}\left( \left(  \frac{\sin{\varphi_0}h_1}{m^3} + t\frac{c\sh{\theta_0}}{m} \right)\alpha_3(t) - \left( t\frac{a\sh{\theta_0}}{m} - \frac{\cos{\varphi_0}h_1}{m^3} \right)\alpha_2(t) - t\frac{a \th{(2\theta_0)}}{2}\alpha_1(t) - t\frac{c \th{(2\theta_0)}}{2}\alpha_4(t) \right) + \right. \right. \\
& + \left. \left. \frac{c}{m}\left( \left( -\frac{\sin{\varphi_0}h_1}{m^3} - t\frac{c\sh{\theta_0}}{m} \right)\alpha_2(t) - \left( t \frac{a\sh{\theta_0}}{m} - \frac{\cos{\varphi_0}h_1}{m^3} \right)\alpha_3(t) + t\frac{a \th{(2\theta_0)}}{2}\alpha_4(t) - t\frac{c \th{(2\theta_0)}}{2}\alpha_1(t) \right) \right) + \right. \\
& + \left. t\frac{\sh^3{\theta_0}}{m^2}\left( \frac{a}{m}\left( -\frac{a}{2}\alpha_1(t) + \frac{ah_1}{m}\alpha_2(t) - \frac{ch_1}{m}\alpha_3(t) - \frac{c}{2}\alpha_4(t) \right) + \frac{c}{m}\left( -\frac{c}{2}\alpha_1(t) + \frac{ch_1}{m}\alpha_2(t) + \frac{ah_1}{m}\alpha_3(t) + \frac{a}{2}\alpha_4(t) \right) \right) \right]\alpha_1(t)\alpha_3(t) + \\
& + \left[ \left( \frac{h_1^2}{m} - \frac{m}{2} \right)\left( \frac{c}{m}\left( - t\frac{a \th{(2\theta_0)}}{2}\alpha_1(t) - \left( t\frac{a\sh{\theta_0}}{m} - \frac{\cos{\varphi_0}h_1}{m^3} \right)\alpha_2(t) + \left(  \frac{\sin{\varphi_0}h_1}{m^3} + t\frac{c\sh{\theta_0}}{m} \right)\alpha_3(t) - t\frac{c \th{(2\theta_0)}}{2}\alpha_4(t) \right) - \right. \right. \\
& \left. \left. - \frac{a}{m}\left( - t\frac{c \th{(2\theta_0)}}{2}\alpha_1(t) + \left( -\frac{\sin{\varphi_0}h_1}{m^3} - t\frac{c\sh{\theta_0}}{m} \right)\alpha_2(t) - \left( t \frac{a\sh{\theta_0}}{m} - \frac{\cos{\varphi_0}h_1}{m^3} \right)\alpha_3(t) + t\frac{a \th{(2\theta_0)}}{2}\alpha_4(t) \right) \right)  \right] \alpha_2(t)\alpha_2(t) + \\
& + \left[ \left( \frac{h_1^2}{m} - \frac{m}{2} \right)\left( \frac{a}{m}\left( - t\frac{a \th{(2\theta_0)}}{2}\alpha_1(t) - \left( t\frac{a\sh{\theta_0}}{m} - \frac{\cos{\varphi_0}h_1}{m^3} \right)\alpha_2(t) + \left(  \frac{\sin{\varphi_0}h_1}{m^3} + t\frac{c\sh{\theta_0}}{m} \right)\alpha_3(t) - t\frac{c \th{(2\theta_0)}}{2}\alpha_4(t) \right) + \right. \right. \\
& \left. \left. + \frac{c}{m}\left( - t\frac{c \th{(2\theta_0)}}{2}\alpha_1(t) + \left( -\frac{\sin{\varphi_0}h_1}{m^3} - t\frac{c\sh{\theta_0}}{m} \right)\alpha_2(t) - \left( t \frac{a\sh{\theta_0}}{m} - \frac{\cos{\varphi_0}h_1}{m^3} \right)\alpha_3(t) + t\frac{a \th{(2\theta_0)}}{2}\alpha_4(t) \right) \right) - \right. \\
& \left. - \frac{\sh{\theta_0}}{m^3}\left( \frac{c}{m}\left( -\frac{a}{2}\alpha_1(t) + \frac{ah_1}{m}\alpha_2(t) - \frac{ch_1}{m}\alpha_3(t) - \frac{c}{2}\alpha_4(t) \right) - \frac{a}{m}\left( -\frac{c}{2}\alpha_1(t) + \frac{ch_1}{m}\alpha_2(t) + \frac{ah_1}{m}\alpha_3(t) + \frac{a}{2}\alpha_4(t) \right) \right) \right] \alpha_2(t)\alpha_3(t) + \\
& + \left[ -\frac{\sh{\theta_0}}{m^3}\left( \frac{a}{m}\left( -\frac{a}{2}\alpha_1(t) + \frac{ah_1}{m}\alpha_2(t) - \frac{ch_1}{m}\alpha_3(t) - \frac{c}{2}\alpha_4(t) \right) + \frac{c}{m}\left( -\frac{c}{2}\alpha_1(t) + \frac{ch_1}{m}\alpha_2(t) + \frac{ah_1}{m}\alpha_3(t) + \frac{a}{2}\alpha_4(t) \right) \right) \right] \alpha_3(t)\alpha_3(t) = \\
& = \left[ \left( \frac{h_1}{2}\left( \frac{c}{m}( - t\frac{a \th{(2\theta_0)}}{2} ) - \frac{a}{m}( - t\frac{c \th{(2\theta_0)}}{2} ) \right) + t\frac{\sh^3{\theta_0}}{m^2}\left( \frac{c}{m}( -\frac{a}{2} ) - \frac{a}{m}( -\frac{c}{2} )  \right) \right)\alpha_1(t) + \right. \\
& + \left. \left( \frac{h_1}{2}\left( \frac{c}{m}( \frac{\cos{\varphi_0}h_1}{m^3} - t\frac{a\sh{\theta_0}}{m} ) - \frac{a}{m}( -\frac{\sin{\varphi_0}h_1}{m^3} - t\frac{c\sh{\theta_0}}{m} ) \right) + t\frac{\sh^3{\theta_0}}{m^2}\left( \frac{c}{m}( \frac{ah_1}{m}  ) - \frac{a}{m}( \frac{ch_1}{m} )  \right) \right)\alpha_2(t) + \right. \\
& + \left. \left( \frac{h_1}{2}\left( \frac{c}{m}( \frac{\sin{\varphi_0}h_1}{m^3} + t\frac{c\sh{\theta_0}}{m} ) - \frac{a}{m}( \frac{\cos{\varphi_0}h_1}{m^3} - t \frac{a\sh{\theta_0}}{m} ) \right) + t\frac{\sh^3{\theta_0}}{m^2}\left( \frac{c}{m}( - \frac{ch_1}{m} ) - \frac{a}{m}( \frac{ah_1}{m} )  \right) \right)\alpha_3(t) + \right. \\
& + \left. \left( \frac{h_1}{2}\left( \frac{c}{m}( - t\frac{c \th{(2\theta_0)}}{2} ) - \frac{a}{m}( t\frac{a \th{(2\theta_0)}}{2} ) \right) + t\frac{\sh^3{\theta_0}}{m^2}\left( \frac{c}{m}( - \frac{c}{2} ) - \frac{a}{m}( \frac{a}{2} )  \right) \right)\alpha_4(t)
\right]\alpha_1(t)\alpha_2(t) + \\
& + \left[ \left( \frac{h_1}{2}\left( \frac{a}{m}( - t\frac{a \th{(2\theta_0)}}{2} ) + \frac{c}{m}( - t\frac{c \th{(2\theta_0)}}{2} ) \right) + t\frac{\sh^3{\theta_0}}{m^2}\left( \frac{a}{m}( -\frac{a}{2} ) + \frac{c}{m}( -\frac{c}{2} ) \right) \right)\alpha_1(t) + \right. \\
& + \left. \left( \frac{h_1}{2}\left( \frac{a}{m}( \frac{\cos{\varphi_0}h_1}{m^3} - t\frac{a\sh{\theta_0}}{m} ) + \frac{c}{m}( -\frac{\sin{\varphi_0}h_1}{m^3} - t\frac{c\sh{\theta_0}}{m} ) \right) + t\frac{\sh^3{\theta_0}}{m^2}\left( \frac{a}{m}( \frac{ah_1}{m} ) + \frac{c}{m}( \frac{ch_1}{m} ) \right) \right)\alpha_2(t) + \right. \\
& + \left. \left( \frac{h_1}{2}\left( \frac{a}{m}( \frac{\sin{\varphi_0}h_1}{m^3} + t\frac{c\sh{\theta_0}}{m} ) + \frac{c}{m}( \frac{\cos{\varphi_0}h_1}{m^3} - t \frac{a\sh{\theta_0}}{m} ) \right) + t\frac{\sh^3{\theta_0}}{m^2}\left( \frac{a}{m}( - \frac{ch_1}{m} ) + \frac{c}{m}( \frac{ah_1}{m} ) \right) \right)\alpha_3(t) + \right. \\
& + \left. \left( \frac{h_1}{2}\left( \frac{a}{m}( - t\frac{c \th{(2\theta_0)}}{2} ) + \frac{c}{m}( t\frac{a \th{(2\theta_0)}}{2} ) \right) + t\frac{\sh^3{\theta_0}}{m^2}\left( \frac{a}{m}( - \frac{c}{2} ) + \frac{c}{m}( \frac{a}{2} ) \right) \right)\alpha_4(t)
\right]\alpha_1(t)\alpha_3(t) + \\
& + \left( \left( \frac{h_1^2}{m} - \frac{m}{2} \right)\left( h_1\frac{c\cos{\varphi_0}+a\sin{\varphi_0}}{m^4}\alpha_2(t) + t\frac{(a^2+c^2)\sh{\theta_0}}{m^2} \alpha_3(t) - t\frac{(a^2+c^2)\th{(2\theta_0)}}{2m}\alpha_4(t) \right)  \right) \alpha_2(t)\alpha_2(t) + \\
& + \left( \left( \frac{h_1^2}{m} - \frac{m}{2} \right)\left( -t\frac{(a^2+c^2)\th{(2\theta_0)}}{2m}\alpha_1(t) + h_1\frac{c\cos{\varphi_0} + a\sin{\varphi_0}}{m^4}\alpha_2(t) + h_1\frac{c\cos{\varphi_0} + a\sin{\varphi_0}}{m^4}\alpha_3(t) \right) - \right. \\
& \left. - \frac{\sh{\theta_0}}{m^3}\left( -h_1\frac{a^2+c^2}{m^2}\alpha_3(t) - \frac{a^2 + c^2}{2m}\alpha_4(t) \right) \right)\alpha_2(t)\alpha_3(t) + \\
& + \left( -\frac{\sh{\theta_0}}{m^3}\left( -\frac{a^2+c^2}{2m}\alpha_1(t) + h_1\frac{a^2+c^2}{m^2}\alpha_2(t) \right) \right) \alpha_3(t)\alpha_3(t) =
\end{align*}

\begin{align*}
& = 
\end{align*}

\begin{align*}
& m = \sqrt{\ch{(2\theta_0)}}, \quad  h_1 = -\ch{\theta_0}, \quad a = \sin{\varphi_0}\sh{\theta_0}, \quad c = \cos{\varphi_0}\sh{\theta_0},\\
& m^2 = \ch{(2\theta_0)}, \quad m^4 = \ch^2{(2\theta_0)},\\
& \frac{h_1^2}{m} - \frac{m}{2} = \frac{2h_1^2 - m^2}{2m} = \frac{2\ch^2{\theta_0} - \ch{(2\theta_0)}}{2m} = \frac{2\ch^2{\theta_0} - (2\ch^2{\theta_0} - 1)}{2m} = \frac{1}{2m},\\
& c^2 + a^2 = \cos^2{\varphi_0}\sh^2{\theta_0} + \sin^2{\varphi_0}\sh^2{\theta_0} = \sh^2{\theta_0},\\
& c\cos{\varphi_0} + a\sin{\varphi_0} = \cos^2{\varphi_0}\sh{\theta_0} + \sin^2{\varphi_0}\sh{\theta_0} = \sh{\theta_0},\\
& c\sin{\varphi_0} - a\cos{\varphi_0} = \sin{\varphi_0}\cos{\varphi_0}\sh{\theta_0} - \sin{\varphi_0}\cos{\varphi_0}\sh{\theta_0} = 0.
\end{align*}

\begin{align*}
& = 
\end{align*}

\begin{align*}
& = 
\end{align*}

\begin{align*}
& \frac{\sh{\theta_0}h_1^2}{2m^4} = \frac{\sh{\theta_0}\ch^2{\theta_0}}{2\ch^2{(2\theta_0)}},\\
& -t\frac{\sh^2{\theta_0}h_1\th{(2\theta_0)}}{4m} - t\frac{\sh^5{\theta_0}}{2m^3} = t\frac{\sh^3{\theta_0}\ch^2{\theta_0} - \sh^5{\theta_0}}{2\ch{(2\theta_0)}\sqrt{\ch{(2\theta_0)}}} = t\frac{\sh^3{\theta_0}}{2\ch{(2\theta_0)}\sqrt{\ch{(2\theta_0)}}},\\
& -t\frac{\sh^2{\theta_0}\th{(2\theta_0)}}{4m^2} + \frac{\sh^3{\theta_0}}{2m^4} = -t\frac{\sh^2{\theta_0}\sh{(2\theta_0)}}{4\ch^2{(2\theta_0)}} + \frac{\sh^3{\theta_0}}{2\ch^2{(2\theta_0)}} = \frac{(-t\ch{\theta_0}+1)\sh^3{\theta_0}}{2\ch^2{(2\theta_0)}},\\
& \frac{\sh{\theta_0}h_1}{2m^5} = -\frac{\sh{\theta_0}\ch{\theta_0}}{2\ch^2{(2\theta_0)}\sqrt{\ch{(2\theta_0)}}}.
\end{align*}

\begin{align*}
& = \sin{\left( \frac{t}{2}m \right)}\frac{\sh{\theta_0}}{32\ch^2{(2\theta_0)}\sqrt{\ch{(2\theta_0)}}}U(t, \theta_0, \varphi_0),
\end{align*}

\begin{align*}
& U(t,\theta_0,\varphi_0) = 4t\cos{(t\ch{\theta_0})}\cos^2{\left(\frac{t}{2}\sqrt{\ch{(2\theta_0)}}\right)}\ch^4{\theta_0} + 4t\cos{(t\ch{\theta_0})}\cos^2{\left( \frac{t}{2}\sqrt{\ch{(2\theta_0)}} \right)}\ch^2{\theta_0}\left( 3\ch{(2\theta_0)} - 5 \right) + \\
& + \cos{(t\ch{\theta_0})}\left( t\cos^2{\left( \frac{t}{2}\sqrt{\ch{(2\theta_0)}} \right)}(\ch{(2\theta_0)} - 3)^2 + 8\sqrt{\ch{(2\theta_0)}}\sin{\left(t\sqrt{\ch{(2\theta_0)}}\right)} \right) + \\
& + 8\ch{\theta_0}\sin{(t\ch{\theta_0})}\left( 1-\cos{\left(t\sqrt{\ch{(2\theta_0)}} \right)} + t\sqrt{\ch{(2\theta_0)}}\sin{\left(t\sqrt{\ch{(2\theta_0)}} \right)}\sh^2{\theta_0} \right).
\end{align*}

Consider the function $U(t, \theta_0, \varphi_0)$ and simplify the expression.

\begin{align*}
& U(t, \theta_0, \varphi_0)  = 4th_1^4\cos{(th_1)}\cos^2{\left( \frac{t}{2}m \right)} + 4th_1^2(3m^2-5)\cos{(th_1)}\cos^2{\left( \frac{t}{2}m \right)} + t(m^2-3)^2\cos{(th_1)}\cos^2{\left( \frac{t}{2}m \right)} + \\
& + 8m\cos{(th_1)}\sin{(tm)} + 8h_1\sin{(th_1)} - 8h_1\sin{(th_1)}\cos{(tm)} + 8th_1m\sh^2{\theta_0}\sin{(th_1)}\sin{(tm)} = \\
& = t\left( (4h_1^4 + 4h_1^2(3m^2-5) + (m^2-3)^2 )\cos{(th_1)}\cos^2{\left( \frac{t}{2}m \right)} + 8h_1m\sh^2{\theta_0}\sin{(th_1)}\sin{(tm)} \right) + \\
& + 8m\cos{(th_1)}\sin{(tm)}+ 8h_1\sin{(th_1)} - 8h_1\sin{(th_1)}\cos{(tm)}. 
\end{align*}

Since

\begin{align*}
& 4h_1^4 + 4h_1^2(3m^2-5) + (m^2-3)^2 = 4\ch^4{\theta_0} + 4\ch^2{\theta_0}(3\ch{(2\theta_0)} - 5) + (\ch{(2\theta_0)} - 3)^2 = 4\ch^4{\theta_0} + 4\ch^2{\theta_0}(6\ch^2{\theta_0} - 8) + \\
& + (2\ch^2{\theta_0}-4)^2 = 4\ch^4{\theta_0} + 24\ch^4{\theta_0} - 32\ch^2{\theta_0} + 4\ch^4{\theta_0} - 16\ch^2{\theta_0} + 16 = 16( 2\ch^4{\theta_0} - 3\ch^2{\theta_0} + 1 ) = \\
& = 16 (4\ch^4{\theta_0} - 4\ch^2{\theta_0} + 1 - 2\ch^4{\theta_0} + \ch^2{\theta_0}) = 16\left( (2\ch^2{\theta_0} - 1)^2 - \ch^2{\theta_0}(2\ch^2{\theta_0} - 1) \right) = \\
& = 16(2\ch^2{\theta_0} - 1)(2\ch^2{\theta_0} - 1 - \ch^2{\theta_0}) =  16\ch{(2\theta_0)}\sh^2{\theta_0},
\end{align*}

and also

\begin{align*}
& 1-\cos{(tm)} = 1-\left( 1-2\sin^2{\left( \frac{t}{2}m \right)} \right) = 2\sin^2{\left( \frac{t}{2}m \right)},
\end{align*}

we continue simplifying the expression for $U$:

\begin{align*}
& U(t, \theta_0, \varphi_0) = 16m^2\sh^2{\theta_0}t\cos{(th_1)}\cos^2{\left( \frac{t}{2}m \right)} + 16h_1m\sh^2{\theta_0}t\sin{(th_1)}\sin{\left( \frac{t}{2}m \right)}\cos{\left( \frac{t}{2}m \right)} + \\
& + 16m\cos{(th_1)}\sin{\left( \frac{t}{2}m \right)}\cos{\left( \frac{t}{2}m \right)} + 16h_1\sin{(th_1)}\sin^2{\left( \frac{t}{2}m \right)} = \\
& = 16\left[ m\cos{(th_1)}\cos{\left( \frac{t}{2}m \right)}\left( m\sh^2{\theta_0}t\cos{\left( \frac{t}{2}m \right)} + \cos{\left( \frac{t}{2}m \right)} \right) + \right. \\
& + \left. h_1\sin{(th_1)}\sin{\left( \frac{t}{2}m \right)}\left( m\sh^2{\theta_0}t\cos{\left( \frac{t}{2}m \right)} + \sin{\left( \frac{t}{2}m \right)} \right) \right] = \\
& = 16\left( m\sh^2{\theta_0}t\cos{\left( \frac{t}{2}m \right)} + \sin{\left( \frac{t}{2}m \right)} \right) \left( m\cos{(th_1)}\cos{\left( \frac{t}{2}m \right)} + h_1\sin{(th_1)}\sin{\left( \frac{t}{2}m \right)} \right).
\end{align*}

\end{proof}

\begin{corollary}
\label{sopr_tochki_tk}
We have a series of conjugate points $t_k = (2k\pi)/\sqrt{\ch{(2\theta_0)}}$, $k=1, 2, ...$.
\end{corollary}

\begin{proof}
Since the factor $\sin{\left( \frac{t}{2}\sqrt{\ch{(2\theta_0)}} \right)}$ is explicitly isolated (as well as the factor $\sh{\theta_0}$, but the line $\theta_0 = 0$ does not interest us, since it is a singular set of the parameterization of the conjugate variables), at the points $t_k = (2k\pi)/\sqrt{\ch{(2\theta_0)}}$, $k=1, 2, ...$, the determinant vanishes.
\end{proof}

\begin{corollary}
Let
\begin{align*}
&   f(t) = \sqrt{\ch{(2\theta_0)}}\sh^2{\theta_0}t\cos{\left( \frac{t}{2}\sqrt{\ch{(2\theta_0)}} \right)} + \sin{\left( \frac{t}{2}\sqrt{\ch{(2\theta_0)}} \right)}; \\
& g(t) = \sqrt{\ch{(2\theta_0)}}\cos{(t\ch{\theta_0})}\cos{\left( \frac{t}{2}\sqrt{\ch{(2\theta_0)}} \right)} + \ch{\theta_0}\sin{(t\ch{\theta_0})}\sin{\left( \frac{t}{2}\sqrt{\ch{(2\theta_0)}} \right)}.
\end{align*}

Then the first conjugate time for $\ch{\theta_0} \geqslant \frac{3}{\sqrt{2}}$ is the zero of the function $f(t)$; for $1 < \ch{\theta_0} < \frac{3}{\sqrt{2}}$, the first conjugate time is the zero of the function $f(t)$ or the zero of the function $g(t)$.
\end{corollary}

\begin{proof}
The functions $f(t)$ and $g(t)$ are factors of the expression for the determinant of the Jacobian matrix of the exponential map obtained in Proposition \ref{det_vych}. In this statement, we show that the first conjugate time is a zero of one of these functions, and not the first point $t_1$ from Corollary \ref{sopr_tochki_tk}. That is, we are looking for zeros of the functions $f(t)$ and $g(t)$ belonging to the interval $\left(0, \frac{2\pi}{\sqrt{\ch{(2\theta_0)}}} \right)$.

Let us find the points where the factors vanish.

\begin{itemize}
    \item[1)] 
    $$
\sqrt{\ch{(2\theta_0)}}\sh^2{\theta_0}t\cos{\left( \frac{t}{2}\sqrt{\ch{(2\theta_0)}} \right)} = - \sin{\left( \frac{t}{2}\sqrt{\ch{(2\theta_0)}} \right)}.
    $$
    
    The case $\theta_0 = 0$ does not interest us, but we obtain the following zeros:

        $$
        0 = \sin{\left( \frac{t}{2} \right)}
        $$

        $$
        t = 2\pi k, \ k \in \mathbb{Z}.
        $$
        
    Let $\theta_0 \neq 0$.

    Consider the function $$
   f(t) = \sqrt{\ch{(2\theta_0)}}\sh^2{\theta_0}t\cos{\left( \frac{t}{2}\sqrt{\ch{(2\theta_0)}} \right)} + \sin{\left( \frac{t}{2}\sqrt{\ch{(2\theta_0)}} \right)}.
    $$ 

    Some of its values:
    $f(0) = 0$; $f\left( \frac{\pi}{\sqrt{\ch{(2\theta_0)}}} \right) = 1 > 0$; $f\left( \frac{2\pi}{\sqrt{\ch{(2\theta_0)}}} \right) = -2\pi \sh^2{\theta_0} < 0$.
    
    Since $\cos{\left( \frac{t}{2}\sqrt{\ch{(2\theta_0)}} \right)} > 0$ and $\sin{\left( \frac{t}{2}\sqrt{\ch{(2\theta_0)}} \right)} > 0$ for $t \in (0, \frac{\pi}{\sqrt{\ch{(2\theta_0)}}})$, then $f(t) > 0$ on this interval (even on the half-interval, since the right boundary can be included), that is, there are no zeros on this half-interval.

    By continuity of the function $f(t)$ on the segment $t \in \left[ \frac{\pi}{\sqrt{\ch{(2\theta_0)}}},\frac{2\pi}{\sqrt{\ch{(2\theta_0)}}} \right]$, and also taking into account its values at the endpoints, there is at least one zero on the interval $t \in \left( \frac{\pi}{\sqrt{\ch{(2\theta_0)}}}, \frac{2\pi}{\sqrt{\ch{(2\theta_0)}}} \right)$. Let us show that it will be unique, using the monotonicity of the function $f(t)$ on this interval. Indeed,

    \begin{align*}
    & f'(t) = \sqrt{\ch{(2\theta_0)}}\sh^2{\theta_0}\cos{\left( \frac{t}{2}\sqrt{\ch{(2\theta_0)}} \right)} - \sqrt{\ch{(2\theta_0)}}\sh^2{\theta_0}t\sin{\left( \frac{t}{2}\sqrt{\ch{(2\theta_0)}} \right)} + \frac{\sqrt{\ch{(2\theta_0)}}}{2}\cos{\left( \frac{t}{2}\sqrt{\ch{(2\theta_0)}} \right)} = \\
    & = \frac{\ch^{3/2}{(2\theta_0)}}{2}\cos{\left( \frac{t}{2}\sqrt{\ch{(2\theta_0)}} \right)} - \sqrt{\ch{(2\theta_0)}}\sh^2{\theta_0}t\sin{\left( \frac{t}{2}\sqrt{\ch{(2\theta_0)}} \right)}.
    \end{align*}

    Since $\cos{\left( \frac{t}{2}\sqrt{\ch{(2\theta_0)}} \right)} < 0$ and $\sin{\left( \frac{t}{2}\sqrt{\ch{(2\theta_0)}} \right)} > 0$ for $t \in \left(\frac{\pi}{\sqrt{\ch{(2\theta_0)}}}, \frac{2\pi}{\sqrt{\ch{(2\theta_0)}}} \right)$, then $f'(t) < 0$ on this interval ($f(t)$ is decreasing).

    Thus, the function $f(t)$ has a unique zero on the interval $\left(0, \frac{2\pi}{\sqrt{\ch{(2\theta_0)}}} \right)$.

    \item[2)]
    $$
    \sqrt{\ch{(2\theta_0)}}\cos{(t\ch{\theta_0})}\cos{\left( \frac{t}{2}\sqrt{\ch{(2\theta_0)}} \right)} = - \ch{\theta_0}\sin{(t\ch{\theta_0})}\sin{\left( \frac{t}{2}\sqrt{\ch{(2\theta_0)}} \right)}. 
    $$
    
        The case $\theta_0 = 0$ does not interest us, but we obtain the following zeros:

        $$
        \cos{t}\cos{\left( \frac{t}{2} \right)} + \sin{t}\sin{\left( \frac{t}{2} \right)} = 0
        $$

        $$
        \cos{\frac{t}{2}} = 0
        $$

        $$
        \frac{t}{2} = \frac{\pi}{2} + \pi k, \ k \in \mathbb{Z}
        $$

        $$
        t = \pi + 2\pi k, \ k \in \mathbb{Z}.
        $$
        
    Let $\theta_0 \neq 0$.

    Consider the function 
    $$
    g(t) = \sqrt{\ch{(2\theta_0)}}\cos{(t\ch{\theta_0})}\cos{\left( \frac{t}{2}\sqrt{\ch{(2\theta_0)}} \right)} + \ch{\theta_0}\sin{(t\ch{\theta_0})}\sin{\left( \frac{t}{2}\sqrt{\ch{(2\theta_0)}} \right)} = m\alpha_1(t) + h_1\alpha_2(t).
    $$

    Some of its values: $g(0) = \sqrt{\ch{(2\theta_0)}} > 0$; 
    
    $g\left( \frac{\pi}{\sqrt{\ch{(2\theta_0)}}} \right) = \ch{\theta_0}\sin{\left( \frac{\pi \ch{\theta_0}}{\sqrt{\ch{(2\theta_0)}}} \right)} > 0$; 
    
    $g\left(\frac{\pi}{2\ch{\theta_0}}\right) = \ch{\theta_0}\sin{\left( \frac{\pi \sqrt{\ch{(2\theta_0)}}}{4\ch{\theta_0}} \right)} > 0$;
    
    $g\left( \frac{\pi}{\ch{\theta_0}} \right) = -\sqrt{\ch{(2\theta_0)}}\cos{\left( \frac{\pi \sqrt{\ch{(2\theta_0)}}}{2\ch{\theta_0}} \right)} > 0$; 
    
    $g\left( \frac{2\pi}{\sqrt{\ch{(2\theta_0)}}} \right) = -\sqrt{\ch{(2\theta_0)}}\cos{\left( \frac{2\pi \ch{\theta_0}}{\sqrt{\ch{(2\theta_0)}}} \right)} > 0$ for $\ch{\theta_0} > \frac{3}{\sqrt{2}}$, $ < 0$ for $1 < \ch{\theta_0} < \frac{3}{\sqrt{2}}$, $= 0$ for $\ch{\theta_0} = \frac{3}{\sqrt{2}}$; 
    
    $g\left( \frac{3\pi}{2\ch{\theta_0}} \right) = -\ch{\theta_0}\sin{\left( \frac{3\pi \sqrt{\ch{(2\theta_0)}}}{4\ch{\theta_0}} \right)}> 0$ for $\ch{\theta_0} > \frac{3}{\sqrt{2}}$, $ < 0$ for $1 < \ch{\theta_0} < \frac{3}{\sqrt{2}}$, $= 0$ for $\ch{\theta_0} = \frac{3}{\sqrt{2}}$.

    Note that the following inequalities hold:
    $$
    \frac{\sqrt{2}}{2} < \frac{\ch{\theta_0}}{\sqrt{\ch{(2\theta_0)}}} < 1 \Leftrightarrow 1 < \frac{\sqrt{\ch{(2\theta_0)}}}{\ch{\theta_0}} < \sqrt{2}, \quad \theta_0 \neq 0.
    $$
    moreover, equality to $1$ is achieved only for $\theta_0 = 0$ (a case that does not interest us).

    From these inequalities, we can draw conclusions about the following arrangement of points for all $\theta_0 \neq 0$:
    $$
    0 < \frac{\pi}{2\sqrt{\ch{(2\theta_0)}}} < \frac{\pi}{2\ch{\theta_0}} < \frac{\pi}{\sqrt{\ch{(2\theta_0)}}} < \frac{\pi}{\ch{\theta_0}} < \frac{2\pi}{\sqrt{\ch{(2\theta_0)}}} < \frac{2\pi}{\ch{\theta_0}}.
    $$

    Since $\frac{3}{2}\sqrt{2} > 2$ and the following inequalities hold:

    $$
    \frac{3\pi}{2\sqrt{\ch{(2\theta_0)}}} < \frac{3\pi}{2\ch{\theta_0}} < \frac{3\pi\sqrt{2}}{2\sqrt{\ch{(2\theta_0)}}},
    $$

    depending on $\theta_0$, $\frac{3\pi}{2\ch{\theta_0}}$ can be either greater than $\frac{2\pi}{\sqrt{\ch{(2\theta_0)}}}$ or less than $\frac{2\pi}{\sqrt{\ch{(2\theta_0)}}}$, or equal to it. Namely, for $1 < \ch{\theta_0} < \frac{3}{\sqrt{2}}$ we have $\frac{3\pi}{2\ch{\theta_0}} < \frac{2\pi}{\sqrt{\ch{(2\theta_0)}}}$, for $\frac{3}{\sqrt{2}} < \ch{\theta_0}$ we have $\frac{3\pi}{2\ch{\theta_0}} > \frac{2\pi}{\sqrt{\ch{(2\theta_0)}}}$, and for $\ch{\theta_0} = \frac{3}{\sqrt{2}}$ we have equality.

    $\sin{(t\ch{\theta_0})}$ is positive on the interval $(0,\frac{\pi}{\ch{\theta_0}})$, negative on the interval $\left(\frac{\pi}{\ch{\theta_0}}, \frac{2\pi}{\ch{\theta_0}}\right)$;

    $\sin{\left( \frac{t}{2}\sqrt{\ch{(2\theta_0)}} \right)}$ is positive on the interval $\left(0,\frac{2\pi}{\sqrt{\ch{(2\theta_0)}}}\right)$;
    
    $\cos{(t\ch{\theta_0})}$ is positive on the intervals $\left(0,\frac{\pi}{2\ch{\theta_0}}\right)$ and $\left(\frac{3\pi}{2\ch{\theta_0}},
    \frac{2\pi}{\ch{\theta_0}}\right)$, negative on the interval $\left(\frac{\pi}{2\ch{\theta_0}},\frac{3\pi}{2\ch{\theta_0}}\right)$;

    $\cos{\left( \frac{t}{2}\sqrt{\ch{(2\theta_0)}} \right)}$ is positive on the interval $\left(0, \frac{\pi}{\sqrt{\ch{(2\theta_0)}}}\right)$, negative on the interval $\left(\frac{\pi}{\sqrt{\ch{(2\theta_0)}}}, \frac{2\pi}{\sqrt{\ch{(2\theta_0)}}}\right)$.

    Hence, the product $\sin{(t\ch{\theta_0})}\sin{\left( \frac{t}{2}\sqrt{\ch{(2\theta_0)}} \right)}$ is positive on the interval $\left(0, \frac{\pi}{\ch{\theta_0}}\right)$, negative on the interval $\left( \frac{\pi}{\ch{\theta_0}}, \frac{2\pi}{\sqrt{\ch{(2\theta_0)}}} \right)$;

    the product $\cos{(t\ch{\theta_0})}\cos{\left( \frac{t}{2}\sqrt{\ch{(2\theta_0)}} \right)}$ is positive on the intervals $\left(0, \frac{\pi}{2\ch{\theta_0}}\right)$ and $\left( \frac{\pi}{\sqrt{\ch{(2\theta_0)}}}, \frac{3\pi}{2\ch{\theta_0}} \right)$ (or on the interval $\left( \frac{\pi}{\sqrt{\ch{(2\theta_0)}}}, \frac{2\pi}{\sqrt{\ch{(2\theta_0)}}} \right)$), negative on the intervals $\left( \frac{\pi}{2\ch{\theta_0}}, \frac{\pi}{\sqrt{\ch{(2\theta_0)}}}  \right)$ and $\left( \frac{3\pi}{2\ch{\theta_0}}, \frac{2\pi}{\sqrt{\ch{(2\theta_0)}}} \right)$ (or the second interval does not exist).

    Thus, the function $g(t)$ is positive on the intervals $\left( 0, \frac{\pi}{2\ch{\theta_0}} \right)$ and $\left( \frac{\pi}{\sqrt{\ch{(2\theta_0)}}}, \frac{\pi}{\ch{\theta_0}} \right)$, negative on the interval $\left( \frac{3\pi}{2\ch{\theta_0}}, \frac{2\pi}{\sqrt{\ch{(2\theta_0)}}} \right)$ (if the corresponding inequality holds).

    \begin{align*}
    & g'(t) = -\sqrt{\ch{(2\theta_0)}}\ch{\theta_0}\sin{(t\ch{\theta_0})}\cos{\left( \frac{t}{2}\sqrt{\ch{(2\theta_0)}} \right)} - \frac{\ch{(2\theta_0)}}{2}\cos{(t\ch{\theta_0})}\sin{\left( \frac{t}{2}\sqrt{\ch{(2\theta_0)}} \right)} + \\
    & + \ch^2{\theta_0}\cos{(t\ch{\theta_0})}\sin{\left( \frac{t}{2}\sqrt{\ch{(2\theta_0)}} \right)} + \frac{\ch{\theta_0}\sqrt{\ch{(2\theta_0)}}}{2}\sin{(t\ch{\theta_0})}\cos{\left( \frac{t}{2}\sqrt{\ch{(2\theta_0)}} \right)} = \\
    & = -\frac{\ch{\theta_0}\sqrt{\ch{(2\theta_0)}}}{2}\sin{(t\ch{\theta_0})}\cos{\left( \frac{t}{2}\sqrt{\ch{(2\theta_0)}} \right)} + \frac{1}{2}\cos{(t\ch{\theta_0})}\sin{\left( \frac{t}{2}\sqrt{\ch{(2\theta_0)}} \right)}.
    \end{align*}

    The product $-\sin{(t\ch{\theta_0})}\cos{\left( \frac{t}{2}\sqrt{\ch{(2\theta_0)}} \right)}$ is positive on the interval $\left( \frac{\pi}{\sqrt{\ch{(2\theta_0)}}}, \frac{\pi}{\ch{\theta_0}}
    \right)$, negative on the intervals $\left( 0, \frac{\pi}{\sqrt{\ch{(2\theta_0)}}} \right)$ and $\left( \frac{\pi}{\ch{\theta_0}}, \frac{2\pi}{\sqrt{\ch{(2\theta_0)}}} \right)$;

    the product $\cos{(t\ch{\theta_0})}\sin{\left( \frac{t}{2}\sqrt{\ch{(2\theta_0)}} \right)}$ is positive on the intervals $\left( 0, \frac{\pi}{2\ch{\theta_0}} \right)$ and $\left( \frac{3\pi}{2\ch{\theta_0}}, \frac{2\pi}{\sqrt{\ch{(2\theta_0)}}} \right)$, negative on the interval $\left( \frac{\pi}{2\ch{\theta_0}}, \frac{3\pi}{2\ch{\theta_0}} \right)$.

    Thus, $g'(t)$ is negative on the intervals $\left( \frac{\pi}{2\ch{\theta_0}}, \frac{\pi}{\sqrt{\ch{(2\theta_0)}}} \right)$ and $\left( \frac{\pi}{\ch{\theta_0}}, \frac{3\pi}{2\ch{\theta_0}} \right)$. On the remaining intervals, it may change sign.

    Some of its values: $g'(0) = 0$; 
    
    $g'\left( \frac{\pi}{\sqrt{\ch{(2\theta_0)}}} \right) = \frac{1}{2}\cos{\left( \frac{\pi \ch{\theta_0}}{\sqrt{\ch{(2\theta_0)}}} \right)} < 0$; 
    
    $g'\left(\frac{\pi}{2\ch{\theta_0}}\right) = -\frac{\ch{\theta_0}\sqrt{\ch{(2\theta_0)}}}{2}\cos{\left( \frac{\pi \sqrt{\ch{(2\theta_0)}}}{4\ch{\theta_0}} \right)} < 0$;
    
    $g'\left( \frac{\pi}{\ch{\theta_0}} \right) = -\frac{1}{2}\sin{\left( \frac{\pi \sqrt{\ch{(2\theta_0)}}}{2\ch{\theta_0}} \right)} < 0$; 
    
    $g'\left( \frac{2\pi}{\sqrt{\ch{(2\theta_0)}}} \right) = \frac{\ch{\theta_0}\sqrt{\ch{(2\theta_0)}}}{2}\sin{\left( \frac{2\pi \ch{\theta_0}}{\sqrt{\ch{(2\theta_0)}}} \right)} < 0$ ; 
    
    $g'\left( \frac{3\pi}{2\ch{\theta_0}} \right) = \frac{\ch{\theta_0}\sqrt{\ch{(2\theta_0)}}}{2}\cos{\left( \frac{3\pi \sqrt{\ch{(2\theta_0)}}}{4\ch{\theta_0}} \right)} < 0$.

    Consider three cases:

    \begin{itemize}
        \item[-] $\ch{\theta_0} > \frac{3}{\sqrt{2}}$.

        On the segments $\left[ 0, \frac{\pi}{2\ch{\theta_0}}\right]$ and $\left[ \frac{\pi}{\sqrt{\ch{(2\theta_0)}}}, \frac{\pi}{\ch{\theta_0}} \right]$, the function $g(t)$ is positive, and on the two remaining segments $\left[ \frac{\pi}{2\ch{\theta_0}}, \frac{\pi}{\sqrt{\ch{(2\theta_0)}}} \right]$ and $\left[  \frac{\pi}{\ch{\theta_0}}, \frac{2\pi}{\sqrt{\ch{(2\theta_0)}}}\right]$, the function is strictly decreasing. Since at the endpoints of the segments the function is positive, one can conclude that it is positive on them. Hence, there are no roots on the entire segment $\left[0, \frac{2\pi}{\sqrt{\ch{(2\theta_0)}}} \right]$.

        \item[-] $1 < \ch{\theta_0} < \frac{3}{\sqrt{2}}$.

        On the segments $\left[ 0, \frac{\pi}{2\ch{\theta_0}}\right]$ and $\left[ \frac{\pi}{\sqrt{\ch{(2\theta_0)}}}, \frac{\pi}{\ch{\theta_0}} \right]$, the function $g(t)$ is positive, on the segment $\left[ \frac{3\pi}{2\ch{\theta_0}}, \frac{2\pi}{\sqrt{\ch{(2\theta_0)}}} \right]$ it is negative, and on the two remaining segments $\left[ \frac{\pi}{2\ch{\theta_0}}, \frac{\pi}{\sqrt{\ch{(2\theta_0)}}} \right]$ and $\left[  \frac{\pi}{\ch{\theta_0}}, \frac{2\pi}{\sqrt{\ch{(2\theta_0)}}}\right]$, the function is strictly decreasing. Since at the endpoints of the segment $\left[ \frac{\pi}{2\ch{\theta_0}}, \frac{\pi}{\sqrt{\ch{(2\theta_0)}}} \right]$ the function is positive, one can conclude that it is positive on it. At the endpoints of the segment $\left[  \frac{\pi}{\ch{\theta_0}}, \frac{2\pi}{\sqrt{\ch{(2\theta_0)}}}\right]$, the function takes values of different signs, i.e., we have exactly one root on this interval. Hence, there is exactly one root on the entire segment $\left[0, \frac{2\pi}{\sqrt{\ch{(2\theta_0)}}} \right]$.

        \item[-] $\ch{\theta_0} = \frac{3}{\sqrt{2}}$.
        The only difference from the two previous cases is that for this value of $\ch{\theta_0}$, $\frac{3\pi}{2\ch{\theta_0}} = \frac{2\pi}{\sqrt{\ch{(2\theta_0)}}}$ and $g(t)$ vanishes at this point. Thus, the right endpoint of the segment $\left[0, \frac{2\pi}{\sqrt{\ch{(2\theta_0)}}} \right]$ is the unique zero of the function $g(t)$.
    \end{itemize}
\end{itemize}

Thus, the first conjugate time for $\ch{\theta_0} \geqslant \frac{3}{\sqrt{2}}$ is the zero of the function $f(t)$.
    
For $1 < \ch{\theta_0} < \frac{3}{\sqrt{2}}$, the first conjugate time is the zero of the function $f(t)$ or the zero of the function $g(t)$.

\end{proof}


\subsection{One-Parameter Subgroups}

\begin{proposition}
    \label{odnopar}
    Normal extremal trajectories are one-parameter subgroups only for $\theta_0 = 0$.
    \end{proposition}

    \begin{proof}

    The property of a one-parameter subgroup is that $\forall t_1$, $t_2 \in \mathbb{R}$, $q(t_1)\cdot q(t_2) = q(t_1 + t_2)$ holds. Here and below, multiplication $\cdot$ denotes matrix multiplication. Thus, we need to find out when the following identity holds:

    \begin{align*}
    & q_0\begin{pmatrix}
    b_{11}(t_1) & b_{12}(t_1)\\
    b_{21}(t_1) & b_{22}(t_1)
    \end{pmatrix}
    \begin{pmatrix}
    b_{11}(t_2) & b_{12}(t_2)\\
    b_{21}(t_2) & b_{22}(t_2)
    \end{pmatrix} =
    q_0\begin{pmatrix}
    b_{11}(t_1 + t_2) & b_{12}(t_1 + t_2) \\
    b_{21}(t_1 + t_2) & b_{22}(t_1 + t_2)
    \end{pmatrix}.
    \end{align*}
    
    First, consider the case $\theta_0 = 0$. Then $h_1 = -1$, $m = 1$, $a = 0$, $c = 0$. Here, by direct computation, we show that such trajectories are indeed one-parameter subgroups.

    \begin{align*}
    & b_{11}(t_1)b_{11}(t_2) + b_{12}(t_1)b_{21}(t_2) \big \vert_{h_1 = -1,\ m = 1,\ a = 0,\ c = 0} = \\
    & = \left( \alpha_1(t_1) + \frac{h_1}{m}\alpha_2(t_1) - i\frac{a}{m}\alpha_3(t_1) - i\frac{c}{m}\alpha_2(t_1) \right)\left( \alpha_1(t_2) + \frac{h_1}{m}\alpha_2(t_2) - i\frac{a}{m}\alpha_3(t_2) - i\frac{c}{m}\alpha_2(t_2) \right) + \\
    & + \left( \frac{a}{m}\alpha_2(t_1) - \frac{c}{m}\alpha_3(t_1) + i\alpha_4(t_1) - i\frac{h_1}{m}\alpha_3(t_1) \right)\left( \frac{c}{m}\alpha_3(t_2) - \frac{a}{m}\alpha_2(t_2) - i \frac{h_1}{m}\alpha_3(t_2) + i\alpha_4(t_2) \right)\big \vert_{h_1 = -1,\ m = 1,\ a = 0,\ c = 0} = \\
    & = \left( \alpha_1(t_1) -\alpha_2(t_1) \right)\left( \alpha_1(t_2) -\alpha_2(t_2) \right) + \left( i\alpha_4(t_1) + i\alpha_3(t_1) \right)\left( i \alpha_3(t_2) + i\alpha_4(t_2) \right) = \alpha_1(t_1)\alpha_1(t_2) - \alpha_1(t_1)\alpha_2(t_2) - \\
    & - \alpha_2(t_1)\alpha_1(t_2) + \alpha_2(t_1)\alpha_2(t_2) - \alpha_4(t_1)\alpha_4(t_2) - \alpha_4(t_1)\alpha_3(t_2) - \alpha_3(t_1)\alpha_3(t_2) - \alpha_3(t_1)\alpha_4(t_2) =\\
    & = \cos{(t_1)} \cos{\left( \frac{t_1}{2} \right)} \cos{(t_2)} \cos{\left( \frac{t_2}{2} \right)} - \cos{(t_1)}\sin{\left( \frac{t_1}{2}\right)}\cos{(t_2)}\sin{\left( \frac{t_2}{2} \right)} - \\
    & - \sin{\left( t_1 \right)}\cos{\left( \frac{t_1}{2} \right)}\sin{\left( t_2 \right)}\cos{\left( \frac{t_2}{2} \right)} + \sin{(t_1)} \sin{\left( \frac{t_1}{2} \right)}\sin{(t_2)} \sin{\left( \frac{t_2}{2} \right)} + \\
    & +
    \sin{(t_1)} \sin{\left( \frac{t_1}{2} \right)} \cos{(t_2)} \cos{\left( \frac{t_2}{2} \right)} + \sin{\left( t_1 \right)}\cos{\left( \frac{t_1}{2} \right)}\cos{(t_2)}\sin{\left( \frac{t_2}{2} \right)}\\
    & + \cos{(t_1)}\sin{\left( \frac{t_1}{2}\right)}\sin{\left( t_2 \right)}\cos{\left( \frac{t_2}{2} \right)} +
    \cos{(t_1)} \cos{\left( \frac{t_1}{2} \right)}\sin{(t_2)} \sin{\left( \frac{t_2}{2} \right)} = \\
    & = \cos{(t_1)} \cos{(t_2)}\left( \cos{\left(\frac{t_1}{2}\right)} \cos{\left( \frac{t_2}{2} \right)} - \sin{\left(\frac{t_1}{2}\right)} \sin{\left( \frac{t_2}{2} \right)}  \right) - \sin{(t_1)} \sin{(t_2)}\left( \cos{\left(\frac{t_1}{2}\right)} \cos{\left( \frac{t_2}{2} \right)} - \sin{\left(\frac{t_1}{2}\right)} \sin{\left( \frac{t_2}{2} \right)}  \right) + \\
    & +  \sin{\left( t_1 \right)}\cos{(t_2)}\left( \sin{\left( \frac{t_1}{2} \right)}\cos{\left( \frac{t_2}{2} \right)} + \cos{\left( \frac{t_1}{2} \right)}\sin{\left( \frac{t_2}{2} \right)} \right) + \cos{(t_1)}\sin{\left( t_2 \right)}\left( \sin{\left( \frac{t_1}{2}\right)}\cos{\left( \frac{t_2}{2} \right)} + \cos{\left( \frac{t_1}{2} \right)}\sin{\left( \frac{t_2}{2} \right)} \right) = \\
    & = \left( \cos{(t_1)} \cos{(t_2)} - \sin{(t_1)} \sin{(t_2)}  \right) \left( \cos{\left(\frac{t_1}{2}\right)} \cos{\left( \frac{t_2}{2} \right)} - \sin{\left(\frac{t_1}{2}\right)} \sin{\left( \frac{t_2}{2} \right)}  \right) + \\
    & + \left( \sin{\left( t_1 \right)}\cos{(t_2)} + \cos{(t_1)}\sin{\left( t_2 \right)}  \right) \left( \sin{\left( \frac{t_1}{2}\right)}\cos{\left( \frac{t_2}{2} \right)} + \cos{\left( \frac{t_1}{2} \right)}\sin{\left( \frac{t_2}{2} \right)} \right) = \\
    & = \cos{\left( t_1 + t_2 \right)}\cos{\left( \frac{t_1+t_2}{2} \right)} + \sin{\left( t_1 + t_2 \right)}\sin{\left( \frac{t_1 + t_2}{2} \right)} = \cos{\left(h_1\left( t_1 + t_2 \right)\right)} \cos{\left( \frac{t_1 + t_2}{2}m \right)} - \\
    & - i \frac{a}{m}\sin{\left( \frac{t_1+t_2}{2}m \right)}\cos{\left(h_1 \left(t_1 + t_2\right) \right)} -i \frac{c+ih_1}{m}\sin{\left(\frac{t_1 + t_2}{2}m \right)}\sin{\left(h_1 \left(t_1 + t_2 \right) \right)} \big \vert_{h_1 = -1,\ m = 1,\ a = 0,\ c = 0} = \\
    & = b_{11}(t_1 + t_2) \big \vert_{h_1 = -1,\ m = 1,\ a = 0,\ c = 0}.
    \end{align*}

    \begin{align*}
    & b_{11}(t_1)b_{12}(t_2) + b_{12}(t_1)b_{22}(t_2) \big \vert_{h_1 = -1,\ m = 1,\ a = 0,\ c = 0} = \\
    & = \left( \alpha_1(t_1) + \frac{h_1}{m}\alpha_2(t_1) - i\frac{a}{m}\alpha_3(t_1) - i\frac{c}{m}\alpha_2(t_1) \right)\left( \frac{a}{m}\alpha_2(t_2) - \frac{c}{m}\alpha_3(t_2) + i\alpha_4(t_2) - i\frac{h_1}{m}\alpha_3(t_2) \right) + \\
    & + \left( \frac{a}{m}\alpha_2(t_1) - \frac{c}{m}\alpha_3(t_1) + i\alpha_4(t_1) - i\frac{h_1}{m}\alpha_3(t_1) \right) \left( \frac{h_1}{m}\alpha_2(t_2) + \alpha_1(t_2) + i\frac{c}{m}\alpha_2(t_2) + i\frac{a}{m}\alpha_3(t_2) \right) \big \vert_{h_1 = -1,\ m = 1,\ a = 0,\ c = 0} = \\
    & = \left( \alpha_1(t_1) - \alpha_2(t_1) \right)\left( i\alpha_4(t_2) + i\alpha_3(t_2) \right) + \left( i\alpha_4(t_1) + i\alpha_3(t_1) \right) \left( -\alpha_2(t_2) + \alpha_1(t_2) \right) = i\alpha_1(t_1)\alpha_4(t_2) + i\alpha_1(t_1)\alpha_3(t_2) - \\
    & - i\alpha_2(t_1)\alpha_4(t_2) - i\alpha_2(t_1)\alpha_3(t_2) -i\alpha_4(t_1)\alpha_2(t_2) + i\alpha_4(t_1)\alpha_1(t_2) - i\alpha_3(t_1)\alpha_2(t_2) + i\alpha_3(t_1)\alpha_1(t_2) = \\
    & = i\left[ -\cos{(t_1)} \cos{\left( \frac{t_1}{2} \right)}\sin{\left( t_2 \right)}\cos{\left( \frac{t_2}{2} \right)} + \cos{(t_1)} \cos{\left( \frac{t_1}{2} \right)}\cos{(t_2)}\sin{\left( \frac{t_2}{2} \right)} -  \right. \\
    & \left. - \sin{(t_1)} \sin{\left( \frac{t_1}{2} \right)}\sin{\left( t_2 \right)}\cos{\left( \frac{t_2}{2} \right)} + \sin{(t_1)} \sin{\left( \frac{t_1}{2} \right)}\cos{(t_2)}\sin{\left( \frac{t_2}{2} \right)} - \right. \\
    & \left. - \sin{\left( t_1 \right)}\cos{\left( \frac{t_1}{2} \right)}\sin{(t_2)} \sin{\left( \frac{t_2}{2} \right)} - \sin{\left( t_1 \right)}\cos{\left( \frac{t_1}{2} \right)}\cos{(t_2)} \cos{\left( \frac{t_2}{2} \right)} -  \right. \\
    & \left. + \cos{(t_1)}\sin{\left( \frac{t_1}{2} \right)}\sin{(t_2)} \sin{\left( \frac{t_2}{2} \right)} + \cos{(t_1)}\sin{\left( \frac{t_1}{2} \right)}\cos{(t_2)} \cos{\left( \frac{t_2}{2} \right)} \right] = \\
    & = i\left[ \cos{t_1}\sin{t_2}\left( \sin{\left( \frac{t_1}{2} \right)}\sin{\left( \frac{t_2}{2} \right)} - \cos{\left( \frac{t_1}{2} \right)}\cos{\left( \frac{t_2}{2} \right)} \right) + \sin{t_1}\cos{t_2}\left( \sin{\left( \frac{t_1}{2} \right)}\sin{\left( \frac{t_2}{2} \right)} - \cos{\left( \frac{t_1}{2} \right)}\cos{\left( \frac{t_2}{2} \right)} \right) + \right. \\
    & \left. + \cos{t_1}\cos{t_2}\left( \sin{\left( \frac{t_1}{2} \right)}\cos{\left( \frac{t_2}{2} \right)} + \cos{\left( \frac{t_1}{2} \right)}\sin{\left( \frac{t_2}{2} \right)} \right) - \sin{t_1}\sin{t_2}\left( \sin{\left( \frac{t_1}{2} \right)}\cos{\left( \frac{t_2}{2} \right)} + \cos{\left( \frac{t_1}{2} \right)}\sin{\left( \frac{t_2}{2} \right)} \right) \right] = \\
    & = i\left[ \left( \sin{\left( \frac{t_1}{2} \right)}\sin{\left( \frac{t_2}{2} \right)} - \cos{\left( \frac{t_1}{2} \right)}\cos{\left( \frac{t_2}{2} \right)} \right)\left( \sin{t_1}\cos{t_2} + \cos{t_1}\sin{t_2} \right) + \right. \\
    & + \left. \left( \sin{\left( \frac{t_1}{2} \right)}\cos{\left( \frac{t_2}{2} \right)} + \cos{\left( \frac{t_1}{2} \right)}\sin{\left( \frac{t_2}{2} \right)}  \right) \left( \cos{t_1}\cos{t_2} - \sin{t_1}\sin{t_2} \right) \right] = \\ 
    & = -i\cos{\left( \frac{t_1+t_2}{2} \right)}\sin{(t_1+t_2)} + i\sin{\left( \frac{t_1 + t_2}{2} \right)}\cos{(t_1 + t_2)} = i\cos{\left(\frac{t_1 + t_2}{2}m\right)}\sin{\left(h_1 (t_1 + t_2)\right)} + \\
    & + \frac{a}{m}\sin{\left( \frac{t_1 + t_2}{2}m \right)}\sin{\left( h_1 (t_1 + t_2) \right)} - \frac{c + ih_1}{m}\sin{\left(\frac{t_1 + t_2}{2}m\right)}\cos{\left(h_1 (t_1 + t_2)\right)} \big \vert_{h_1 = -1,\ m = 1,\ a = 0,\ c = 0} = \\
    & = b_{12}(t_1 + t_2)\big \vert_{h_1 = -1,\ m = 1,\ a = 0,\ c = 0}.
    \end{align*}

    \begin{align*}
    & b_{21}(t_1)b_{11}(t_2) + b_{22}(t_1)b_{21}(t_2) = \left( \frac{c}{m}\alpha_3(t_1) - \frac{a}{m}\alpha_2(t_1) - i \frac{h_1}{m}\alpha_3(t_1) + i\alpha_4(t_1) \right)\left( \alpha_1(t_2) + \frac{h_1}{m}\alpha_2(t_2) - i\frac{a}{m}\alpha_3(t_2) - i\frac{c}{m}\alpha_2(t_2) \right) + \\
    & + \left( \frac{h_1}{m}\alpha_2(t_1) + \alpha_1(t_1) + i\frac{c}{m}\alpha_2(t_1) + i\frac{a}{m}\alpha_3(t_1) \right)\left( \frac{c}{m}\alpha_3(t_2) - \frac{a}{m}\alpha_2(t_2) - i \frac{h_1}{m}\alpha_3(t_2) + i\alpha_4(t_2) \right)\big \vert_{h_1 = -1,\ m = 1,\ a = 0,\ c = 0} = \\
    & = \left( i \alpha_3(t_1) + i\alpha_4(t_1) \right)\left( \alpha_1(t_2) - \alpha_2(t_2) \right) + \left( -\alpha_2(t_1) + \alpha_1(t_1) \right)\left( i\alpha_3(t_2) + i\alpha_4(t_2) \right) = i\alpha_3(t_1)\alpha_1(t_2) - i\alpha_3(t_1)\alpha_2(t_2) + \\
    & + i\alpha_4(t_1)\alpha_1(t_2) - i\alpha_4(t_1)\alpha_2(t_2) - i\alpha_2(t_1)\alpha_3(t_2) - i\alpha_2(t_1)\alpha_4(t_2) + i\alpha_1(t_1)\alpha_3(t_2) + i\alpha_1(t_1)\alpha_4(t_2) = \\
    & = b_{11}(t_1)b_{12}(t_2) + b_{12}(t_1)b_{22}(t_2)\big \vert_{h_1 = -1,\ m=1,\ a = 0,\ c = 0} = b_{12}(t_1 + t_2)\big \vert_{h_1 = -1,\ m = 1,\ a = 0,\ c = 0} = \\
    & = -\bar{b}_{12}(t_1 + t_2)\big \vert_{h_1 = -1,\ m = 1,\ a = 0,\ c = 0}.
    \end{align*}

    \begin{align*}
    & b_{21}(t_1)b_{12}(t_2) + b_{22}(t_1)b_{22}(t_2) = \left( \frac{c}{m}\alpha_3(t_1) - \frac{a}{m}\alpha_2(t_1) - i \frac{h_1}{m}\alpha_3(t_1) + i\alpha_4(t_1) \right)\left( \frac{a}{m}\alpha_2(t_2) - \frac{c}{m}\alpha_3(t_2) + i\alpha_4(t_2) - i\frac{h_1}{m}\alpha_3(t_2) \right) + \\
    & + \left( \frac{h_1}{m}\alpha_2(t_1) + \alpha_1(t_1) + i\frac{c}{m}\alpha_2(t_1) + i\frac{a}{m}\alpha_3(t_1) \right)\left( \frac{h_1}{m}\alpha_2(t_2) + \alpha_1(t_2) + i\frac{c}{m}\alpha_2(t_2) + i\frac{a}{m}\alpha_3(t_2) \right) \big \vert_{h_1 = -1,\ m = 1,\ a = 0,\ c = 0} = \\
    & = \left( i\alpha_3(t_1) + i\alpha_4(t_1) \right)\left( i\alpha_4(t_2) + i\alpha_3(t_2) \right) + \left( -\alpha_2(t_1) + \alpha_1(t_1)  \right)\left( -\alpha_2(t_2) + \alpha_1(t_2) \right) = -\alpha_3(t_1)\alpha_4(t_2) - \alpha_3(t_1)\alpha_3(t_2) - \\
    & - \alpha_4(t_1)\alpha_4(t_2) - \alpha_4(t_1)\alpha_3(t_2) + \alpha_2(t_1)\alpha_2(t_2) - \alpha_2(t_1)\alpha_1(t_2) - \alpha_1(t_1)\alpha_2(t_2) + \alpha_1(t_1)\alpha_1(t_2) = \\
    & = b_{11}(t_1)b_{11}(t_2) + b_{12}(t_1)b_{21}(t_2) \big \vert_{h_1 = -1,\ m = 1,\ a = 0,\ c = 0} = b_{11}(t_1 + t_2) \big \vert_{h_1 = -1,\ m = 1,\ a = 0,\ c = 0} = \\
    & = \bar{b}_{11}(t_1 + t_2)\big \vert_{h_1 = -1,\ m = 1,\ a = 0,\ c = 0}.
    \end{align*}

    We will analyze the general case starting from the end, i.e., from what we should obtain, considering the element of the first row, first column:

    \begin{align*}
    & b_{11}(t_1 + t_2) = \cos{\left(h_1(t_1 + t_2)\right)} \cos{\left( \frac{t_1 + t_2}{2}m \right)} - i \frac{a}{m}\sin{\left( \frac{t_1 + t_2}{2}m \right)}\cos{\left(h_1 (t_1 + t_2) \right)} - \\
    & - i \frac{c+ih_1}{m}\sin{\left(\frac{t_1 + t_2}{2}m \right)}\sin{\left(h_1 (t_1 + t_2) \right)} = \\
    & = \left( \cos{\left( \frac{mt_1}{2} \right)}\cos{\left( \frac{mt_2}{2} \right)} - \sin{\left( \frac{mt_1}{2} \right)}\sin{\left( \frac{mt_2}{2} \right)} \right) \left( \cos{(h_1t_1)}\cos{(h_1t_2)} - \sin{(h_1t_1)}\sin{(h_1t_2)} \right) - \\
    & - i \frac{a}{m}\left( \sin{\left( \frac{mt_1}{2} \right)}\cos{\left( \frac{mt_2}{2} \right)} + \cos{\left( \frac{mt_1}{2} \right)}\sin{\left( \frac{mt_2}{2} \right)} \right)\left( \cos{(h_1t_1)}\cos{(h_1t_2)} - \sin{(h_1t_1)}\sin{(h_1t_2)} \right) - \\
    & - i\frac{c+ih_1}{m}\left( \sin{\left( \frac{mt_1}{2}\right)}\cos{\left( \frac{mt_2}{2} \right)} + \cos{\left( \frac{mt_1}{2} \right)}\sin{\left( \frac{mt_2}{2} \right)} \right)\left( \sin{(h_1t_1)}\cos{(h_1t_2)} + \cos{(h_1t_1)}\sin{(h_1t_2)} \right) = \\
    & = \cos{\left( \frac{mt_1}{2} \right)}\cos{\left( \frac{mt_2}{2} \right)}\cos{(h_1t_1)}\cos{(h_1t_2)} - \cos{\left( \frac{mt_1}{2} \right)}\cos{\left( \frac{mt_2}{2} \right)}\sin{(h_1t_1)}\sin{(h_1t_2)} - \\
    & - \sin{\left( \frac{mt_1}{2} \right)}\sin{\left( \frac{mt_2}{2} \right)}\cos{(h_1t_1)}\cos{(h_1t_2)} + \sin{\left( \frac{mt_1}{2} \right)}\sin{\left( \frac{mt_2}{2} \right)}\sin{(h_1t_1)}\sin{(h_1t_2)} - \\
    & - i\frac{a}{m}\sin{\left( \frac{mt_1}{2} \right)}\cos{\left( \frac{mt_2}{2} \right)}\cos{(h_1t_1)}\cos{(h_1t_2)} + i\frac{a}{m} \sin{\left( \frac{mt_1}{2} \right)}\cos{\left( \frac{mt_2}{2} \right)}\sin{(h_1t_1)}\sin{(h_1t_2)} - \\
    & - i\frac{a}{m}\cos{\left( \frac{mt_1}{2} \right)}\sin{\left( \frac{mt_2}{2} \right)}\cos{(h_1t_1)}\cos{(h_1t_2)} + i\frac{a}{m}\cos{\left( \frac{mt_1}{2} \right)}\sin{\left( \frac{mt_2}{2} \right)}\sin{(h_1t_1)}\sin{(h_1t_2)} - \\
    & - i\frac{c + ih_1}{m}\sin{\left( \frac{mt_1}{2}\right)}\cos{\left( \frac{mt_2}{2} \right)}\sin{(h_1t_1)}\cos{(h_1t_2)} - i\frac{c+ih_1}{m}\sin{\left( \frac{mt_1}{2}\right)}\cos{\left( \frac{mt_2}{2} \right)}\cos{(h_1t_1)}\sin{(h_1t_2)} - \\
    & - i\frac{c+ih_1}{m}\cos{\left( \frac{mt_1}{2} \right)}\sin{\left( \frac{mt_2}{2} \right)}\sin{(h_1t_1)}\cos{(h_1t_2)} - i\frac{c+ih_1}{m}\cos{\left( \frac{mt_1}{2} \right)}\sin{\left( \frac{mt_2}{2} \right)}\cos{(h_1t_1)}\sin{(h_1t_2)} = \\
    & = \alpha_1(t_1)\alpha_1(t_2) - \alpha_4(t_1)\alpha_4(t_2) - \alpha_3(t_1)\alpha_3(t_2) + \alpha_2(t_1)\alpha_2(t_2) - \\
    & - i\frac{a}{m}\left( \alpha_3(t_1)\alpha_1(t_2) - \alpha_2(t_1)\alpha_4(t_2) + \alpha_1(t_1)\alpha_3(t_2) - \alpha_4(t_1)\alpha_2(t_2) \right) - \\
    & - i\frac{c+ih_1}{m}\left( \alpha_2(t_1)\alpha_1(t_2) + \alpha_3(t_1)\alpha_4(t_2) + \alpha_4(t_1)\alpha_3(t_2) + \alpha_1(t_1)\alpha_2(t_2) \right).
    \end{align*}

    And now let us write out the element of the first row, first column of the result of the matrix product:

    \begin{align*}
    & b_{11}(t_1)b_{11}(t_2) + b_{12}(t_1)b_{21}(t_2) = \left( \alpha_1(t_1) + \frac{h_1}{m}\alpha_2(t_1) - i\frac{a}{m}\alpha_3(t_1) - i\frac{c}{m}\alpha_2(t_1) \right)\left( \alpha_1(t_2) + \frac{h_1}{m}\alpha_2(t_2) - i\frac{a}{m}\alpha_3(t_2) - i\frac{c}{m}\alpha_2(t_2) \right) + \\
    & + \left( \frac{a}{m}\alpha_2(t_1) - \frac{c}{m}\alpha_3(t_1) + i\alpha_4(t_1) - i\frac{h_1}{m}\alpha_3(t_1) \right)\left( \frac{c}{m}\alpha_3(t_2) - \frac{a}{m}\alpha_2(t_2) - i \frac{h_1}{m}\alpha_3(t_2) + i\alpha_4(t_2) \right) = \\
    & = \alpha_1(t_1)\alpha_1(t_2) + \frac{h_1}{m}\alpha_1(t_1)\alpha_2(t_2) - i\frac{a}{m}\alpha_1(t_1)\alpha_3(t_2) - i\frac{c}{m}\alpha_1(t_1)\alpha_2(t_2) + \\ 
    & + \frac{h_1}{m}\alpha_2(t_1)\alpha_1(t_2) + \frac{h_1^2}{m^2}\alpha_2(t_1)\alpha_2(t_2) - i\frac{h_1a}{m^2}\alpha_2(t_1)\alpha_3(t_2) - i\frac{h_1c}{m^2}\alpha_2(t_1)\alpha_2(t_2) - \\
    & - i\frac{a}{m}\alpha_3(t_1)\alpha_1(t_2) - i\frac{h_1a}{m^2}\alpha_3(t_1)\alpha_2(t_2) - \frac{a^2}{m^2}\alpha_3(t_1)\alpha_3(t_2) - \frac{ac}{m^2}\alpha_3(t_1)\alpha_2(t_2) - \\
    & - i\frac{c}{m}\alpha_2(t_1)\alpha_1(t_2) - i\frac{h_1c}{m^2}\alpha_2(t_1)\alpha_2(t_2) - \frac{ac}{m^2}\alpha_2(t_1)\alpha_3(t_2) - \frac{c^2}{m^2}\alpha_2(t_1)\alpha_2(t_2) + \\
    & + \frac{ac}{m^2}\alpha_2(t_1)\alpha_3(t_2) - \frac{a^2}{m^2}\alpha_2(t_1)\alpha_2(t_2) - i\frac{h_1a}{m^2}\alpha_2(t_1)\alpha_3(t_2) + i\frac{a}{m}\alpha_2(t_1)\alpha_4(t_2) - \\
    & - \frac{c^2}{m^2}\alpha_3(t_1)\alpha_3(t_2) + \frac{ac}{m^2}\alpha_3(t_1)\alpha_2(t_2) + i\frac{h_1c}{m^2}\alpha_3(t_1)\alpha_3(t_2) - i\frac{c}{m}\alpha_3(t_1)\alpha_4(t_2) + \\
    & + i\frac{c}{m}\alpha_4(t_1)\alpha_3(t_2) - i\frac{a}{m}\alpha_4(t_1)\alpha_2(t_2) - \frac{h_1}{m}\alpha_4(t_1)\alpha_3(t_2) - \alpha_4(t_1)\alpha_4(t_2) - \\
    & - i\frac{h_1c}{m^2}\alpha_3(t_1)\alpha_3(t_2) + i\frac{h_1a}{m^2}\alpha_3(t_1)\alpha_2(t_2) - \frac{h_1^2}{m^2}\alpha_3(t_1)\alpha_3(t_2) + \frac{h_1}{m}\alpha_3(t_1)\alpha_4(t_2) = \\
    & = \alpha_1(t_1)\alpha_1(t_2) + \left( \frac{h_1}{m} - i\frac{c}{m}\right)\alpha_1(t_1)\alpha_2(t_2) - i\frac{a}{m}\alpha_1(t_1)\alpha_3(t_2) + \\
    & + \left( \frac{h_1}{m} - i\frac{c}{m} \right)\alpha_2(t_1)\alpha_1(t_2) + \left( \frac{h_1^2}{m^2} - 2i\frac{h_1c}{m^2} - \frac{c^2}{m^2} - \frac{a^2}{m^2} \right)\alpha_2(t_1)\alpha_2(t_2) + \left( -i\frac{h_1a}{m^2} - \frac{ac}{m^2} + \frac{ac}{m^2} - i\frac{h_1a}{m^2} \right)\alpha_2(t_1)\alpha_3(t_2) + \\ 
    & + i\frac{a}{m}\alpha_2(t_1)\alpha_4(t_2) - \\
    & - i\frac{a}{m}\alpha_3(t_1)\alpha_1(t_2) + \left( -i\frac{h_1a}{m^2} - \frac{ac}{m^2} + \frac{ac}{m^2} + i\frac{h_1a}{m^2} \right) \alpha_3(t_1)\alpha_2(t_2) + \left( -\frac{a^2}{m^2} - \frac{c^2}{m^2} + i\frac{h_1c}{m^2} - i\frac{h_1c}{m^2} - \frac{h_1^2}{m^2} \right)\alpha_3(t_1)\alpha_3(t_2) + \\
    & + \left( -i\frac{c}{m} + \frac{h_1}{m} \right)\alpha_3(t_1)\alpha_4(t_2) - \\
    & -i\frac{a}{m}\alpha_4(t_1)\alpha_2(t_2) + \left( i\frac{c}{m} - \frac{h_1}{m} \right)\alpha_4(t_1)\alpha_3(t_2) - \alpha_4(t_1)\alpha_4(t_2) = \\
    & = \alpha_1(t_1)\alpha_1(t_2) + \frac{1+i\sh{2\theta_0}\cos{\varphi_0}}{\ch{2\theta_0}}\alpha_2(t_1)\alpha_2(t_2) - \alpha_3(t_1)\alpha_3(t_2) - \alpha_4(t_1)\alpha_4(t_2) - \\
    & - i\frac{a}{m} \left( \alpha_1(t_1)\alpha_3(t_2) - \alpha_2(t_1)\alpha_4(t_2) + \alpha_3(t_1)\alpha_1(t_2) + \alpha_4(t_1)\alpha_2(t_2) \right) - \\
    & - i\frac{c + ih_1}{m}\left( \alpha_1(t_1)\alpha_2(t_2) + \alpha_2(t_1)\alpha_1(t_2) + \alpha_3(t_1)\alpha_4(t_2) - \alpha_4(t_1)\alpha_3(t_2) \right).
    \end{align*}

    We see that the brackets before $\alpha_2(t_1)\alpha_3(t_2)$ and $\alpha_3(t_1)\alpha_2(t_2)$ vanish.

    The bracket before $\alpha_3(t_1)\alpha_3(t_2)$ simplifies to a simple form:

    $$
     -\frac{a^2}{m^2} - \frac{c^2}{m^2} + i\frac{h_1c}{m^2} - i\frac{h_1c}{m^2} - \frac{h_1^2}{m^2} = \frac{-\sin^2{\varphi_0}\sh^2{\theta_0} - \cos^2{\varphi_0}\sh^2{\theta_0} - \ch^2{\theta_0}}{m^2} = -\frac{\ch^2{\theta_0}+\sh^2{\theta_0}}{\left( \sqrt{\ch{2\theta_0}} \right)^2} = -1.
    $$

    In turn, the bracket before $\alpha_2(t_1)\alpha_2(t_2)$ simplifies only to the following form:

    $$
     \frac{h_1^2}{m^2} - 2i\frac{h_1c}{m^2} - \frac{c^2}{m^2} - \frac{a^2}{m^2} = \frac{\ch^2{\theta_0} - \cos^2{\varphi_0}\sh^2{\theta_0} - \sin^2{\varphi_0}\sh^2{\theta_0} + 2i\ch{\theta_0}\sh{\theta_0}\cos{\varphi_0}}{m^2} = \frac{1 + i\sh{2\theta_0}\cos{\varphi_0}}{\ch{2\theta_0}},
    $$
    although in the result it should be equal to $1$, from which it follows that only for $\theta_0 = 0$ can we obtain a one-parameter subgroup.

    \end{proof}

\subsection{Periodic Trajectories}

\begin{proposition}
    Normal extremals that are one-parameter subgroups are periodic with period $T = 4\pi$.
    \end{proposition}

    \begin{proof}
    By Proposition \ref{odnopar}, only trajectories with $\theta_0 = 0$ are one-parameter subgroups. Then $h_1 = -1$, $m = 1$, $a = 0$, $c = 0$ and, accordingly, the trajectories themselves look as follows:

     \begin{equation*}
    \begin{cases}
    x_1(t) = x_{10}(\alpha_1 -\alpha_2) - y_{20}(\alpha_3 + \alpha_4),\\
    y_1(t) = y_{10}(\alpha_1 - \alpha_2) + x_{20}( \alpha_3 + \alpha_4),\\
    x_2(t) = - y_{10}( \alpha_4 + \alpha_3 ) + x_{20}( -\alpha_2 + \alpha_1 ),\\
    y_2(t) = x_{10}( \alpha_4 + \alpha_3 ) + y_{20}( -\alpha_2 + \alpha_1 ),
    \end{cases}
    \end{equation*}

    where 
    $$
    \alpha_1 = \cos{(-t)} \cos{\left( \frac{t}{2} \right)}, \ \alpha_2 = \sin{(-t)} \sin{\left( \frac{t}{2} \right)}, \ \alpha_3 = \cos{(-t)}\sin{\left( \frac{t}{2} \right)},\ \alpha_4 = \sin{\left( - t \right)}\cos{\left( \frac{t}{2} \right)}.
    $$

    Then, taking into account the assumption made above and its consequences, we have:

    \begin{equation*}
    \begin{cases}
    x_1(t) = x_{10}\left(\cos{(t)} \cos{\left( \frac{t}{2} \right)} + \sin{(t)} \sin{\left( \frac{t}{2} \right)}\right) - y_{20}\left(\cos{(t)}\sin{\left( \frac{t}{2} \right)} - \sin{\left( t \right)}\cos{\left( \frac{t}{2} \right)}\right) = x_{10}\cos{\left( \frac{t}{2} \right)} + y_{20}\sin{\left( \frac{t}{2} \right)},\\
    y_1(t) = y_{10}\left( \cos{(t)} \cos{\left( \frac{t}{2} \right)} + \sin{(t)} \sin{\left( \frac{t}{2} \right)} \right) + x_{20}\left( \cos{(t)}\sin{\left( \frac{t}{2} \right)} - \sin{\left( t \right)}\cos{\left( \frac{t}{2} \right)} \right) = y_{10}\cos{\left( \frac{t}{2} \right)} - x_{20}\sin{\left( \frac{t}{2} \right)},\\
    x_2(t) = - y_{10}\left( -\sin{\left( t \right)}\cos{\left( \frac{t}{2} \right)} + \cos{(t)}\sin{\left( \frac{t}{2} \right)} \right) + x_{20}\left( \sin{(t)} \sin{\left( \frac{t}{2} \right)} + \cos{(t)} \cos{\left( \frac{t}{2} \right)} \right) = y_{10}\sin{\left( \frac{t}{2} \right)} + x_{20}\cos{\left( \frac{t}{2} \right)},\\
    y_2(t) = x_{10}\left( -\sin{\left( t \right)}\cos{\left( \frac{t}{2} \right)} + \cos{(t)}\sin{\left( \frac{t}{2} \right)} \right) + y_{20}\left( \sin{(t)} \sin{\left( \frac{t}{2} \right)} + \cos{(t)} \cos{\left( \frac{t}{2} \right)} \right) = -x_{10}\sin{\left( \frac{t}{2} \right)} + y_{20}\cos{\left( \frac{t}{2} \right)}.
    \end{cases}
    \end{equation*}

    Obviously, such trajectories are periodic with period $T = 4\pi$.

    \end{proof}

    \begin{proposition}
    There exists an infinite set of values of the parameter $\theta_0$ for which the normal extremal trajectories are periodic, without being one-parameter subgroups.
    \end{proposition}

    \begin{proof}
    Let us recall the formulas for normal trajectories:

    \begin{equation*}
    \begin{cases}
    x_1(t) = x_{10}(\alpha_1 + \frac{h_1}{m}\alpha_2) + y_{10}(\frac{a}{m}\alpha_3 + \frac{c}{m}\alpha_2) + x_{20}(\frac{c}{m}\alpha_3 - \frac{a}{m}\alpha_2) + y_{20}(\frac{h_1}{m}\alpha_3 - \alpha_4),\\
    y_1(t) = -x_{10}(\frac{a}{m}\alpha_3 + \frac{c}{m}\alpha_2) + y_{10}(\alpha_1 + \frac{h_1}{m}\alpha_2) + x_{20}(-\frac{h_1}{m}\alpha_3 + \alpha_4) + y_{20}(\frac{c}{m}\alpha_3 - \frac{a}{m}\alpha_2),\\
    x_2(t) = x_{10}( \frac{a}{m}\alpha_2 - \frac{c}{m}\alpha_3 ) + y_{10}( -\alpha_4 + \frac{h_1}{m}\alpha_3 ) + x_{20}( \frac{h_1}{m}\alpha_2 + \alpha_1 ) + y_{20}( -\frac{c}{m}\alpha_2 - \frac{a}{m}\alpha_3 ),\\
    y_2(t) = x_{10}( \alpha_4 - \frac{h_1}{m}\alpha_3 ) + y_{10}( \frac{a}{m}\alpha_2 - \frac{c}{m}\alpha_3 ) + x_{20}(\frac{c}{m}\alpha_2 + \frac{a}{m}\alpha_3 )+ y_{20}( \frac{h_1}{m}\alpha_2 + \alpha_1 ).
    \end{cases}
    \end{equation*}

    \begin{align*}
    & \alpha_1 = \cos{(h_1t)} \cos{\left( \frac{t}{2}m \right)}, \ \alpha_2 = \sin{(h_1t)} \sin{\left( \frac{t}{2}m \right)}, \ \alpha_3 = \cos{(h_1t)}\sin{\left( \frac{t}{2}m \right)},\ \alpha_4 = \sin{\left( h_1 t \right)}\cos{\left( \frac{t}{2}m \right)},\\
    & h_1 = -\ch{\theta_0}, \quad  a = \sin{\varphi_0}\sh{\theta_0}, \quad c = \cos{\varphi_0}\sh{\theta_0}, \quad m = \sqrt{\ch{(2\theta_0)}}.
    \end{align*}

    The general form of the trajectories is as follows:

    \begin{equation*}
    \begin{cases}
    x_1(t) = x_{10}\alpha_1 + \left( x_{10}\frac{h_1}{m} + y_{10}\frac{c}{m} - x_{20}\frac{a}{m} \right)\alpha_2 + \left( y_{10}\frac{a}{m} + x_{20}\frac{c}{m} + y_{20}\frac{h_1}{m} \right)\alpha_3 - y_{20}\alpha_4 = b_{11}\alpha_1 + b_{12}\alpha_2 + b_{13}\alpha_3  + b_{14}\alpha_4,\\
    y_1(t) = y_{10}\alpha_1 + \left( -x_{10}\frac{c}{m} + y_{10}\frac{h_1}{m} - y_{20}\frac{a}{m} \right)\alpha_2 + \left( -x_{10}\frac{a}{m} - x_{20}\frac{h_1}{m} + y_{20}\frac{c}{m} \right)\alpha_3 + x_{20}\alpha_4 = b_{21}\alpha_1 + b_{22}\alpha_2 + b_{23}\alpha_3  + b_{24}\alpha_4,\\
    x_2(t) = x_{20}\alpha_1 + \left( x_{10}\frac{a}{m} + x_{20}\frac{h_1}{m} - y_{20}\frac{c}{m} \right)\alpha_2 + \left( -x_{10}\frac{c}{m} + y_{10}\frac{h_1}{m} - y_{20}\frac{a}{m} \right) \alpha_3 - y_{10}\alpha_4 = b_{31}\alpha_1 + b_{32}\alpha_2 + b_{33}\alpha_3  + b_{34}\alpha_4,\\
    y_2(t) = y_{20}\alpha_1 + \left( y_{10}\frac{a}{m} + x_{20}\frac{c}{m} + y_{20}\frac{h_1}{m} \right)\alpha_2 + \left( -x_{10}\frac{h_1}{m} - y_{10}\frac{c}{m} + x_{20}\frac{a}{m} \right) \alpha_3 + x_{10}\alpha_4 = b_{41}\alpha_1 + b_{42}\alpha_2 + b_{43}\alpha_3  + b_{44}\alpha_4.
    \end{cases}
    \end{equation*}

    Let us expand $\alpha_j$, $j=1,..., 4$ as sums (differences) of cosines and sines, and then return to the trajectories:

    \begin{align*}
    & \alpha_1 = \cos{(h_1t)} \cos{\left( \frac{t}{2}m \right)} = \frac{1}{2}\left[ \cos{\left( \left( h_1 + \frac{m}{2} \right)t \right)} + \cos{\left( \left( h_1 - \frac{m}{2}\right)t \right)} \right] = \frac{1}{2}\left( \cos{(ut)} + \cos{(vt)} \right), \\
    & \alpha_2 = \sin{(h_1t)} \sin{\left( \frac{t}{2}m \right)} = \frac{1}{2}\left[ \cos{\left( \left( h_1 - \frac{m}{2} \right)t \right)} - \cos{\left( \left( h_1 + \frac{m}{2} \right)t \right)} \right] = \frac{1}{2}\left( \cos{(vt)} - \cos{(ut)} \right), \\
    & \alpha_3 = \cos{(h_1t)}\sin{\left( \frac{t}{2}m \right)} = \frac{1}{2}\left[ \sin{\left( \left( h_1 + \frac{m}{2} \right)t \right)} - \sin{\left( \left(  h_1 - \frac{m}{2} \right)t \right)}  \right] = \frac{1}{2}\left( \sin{(ut)} - \sin{(vt)} \right),\\
    & \alpha_4 = \sin{\left( h_1 t \right)}\cos{\left( \frac{t}{2}m \right)} = \frac{1}{2}\left[ \sin{\left( \left( h_1 + \frac{m}{2} \right)t \right)} + \sin{\left( \left( h_1 - \frac{m}{2} \right)t \right)} \right] = \frac{1}{2}\left( \sin{(ut)} + \sin{(vt)} \right).
    \end{align*}

    Substitute the obtained relations into the formulas for the trajectories:

    \begin{equation*}
    \begin{cases}
    x_1(t) = \frac{1}{2}\left[ b_{11}\left( \cos{(ut)} + \cos{(vt)} \right) + b_{12}\left( \cos{(vt)} - \cos{(ut)} \right) + b_{13}\left( \sin{(ut)} - \sin{(vt)} \right) + b_{14}\left( \sin{(ut)} + \sin{(vt)} \right) \right],\\
    y_1(t) = \frac{1}{2}\left[ b_{21}\left( \cos{(ut)} + \cos{(vt)} \right) +b_{22}\left( \cos{(vt)} - \cos{(ut)} \right) + b_{23}\left( \sin{(ut)} - \sin{(vt)} \right) + b_{24}\left( \sin{(ut)} + \sin{(vt)} \right) \right],\\
    x_2(t) = \frac{1}{2}\left[ b_{31}\left( \cos{(ut)} + \cos{(vt)} \right) + b_{32}\left( \cos{(vt)} - \cos{(ut)} \right) + b_{33} \left( \sin{(ut)} - \sin{(vt)} \right) + b_{34}\left( \sin{(ut)} + \sin{(vt)} \right) \right],\\
    y_2(t) = \frac{1}{2}\left[ b_{41}\left( \cos{(ut)} + \cos{(vt)} \right) + b_{42}\left( \cos{(vt)} - \cos{(ut)} \right) + b_{43} \left( \sin{(ut)} - \sin{(vt)} \right) + b_{44}\left( \sin{(ut)} + \sin{(vt)} \right) \right].
    \end{cases}
    \end{equation*}

    Collect like terms:

    \begin{equation*}
    \begin{cases}
    x_1(t) = \frac{1}{2}\left[ ( b_{11} - b_{12} )\cos{(ut)} + ( b_{11} + b_{12} )\cos{(vt)} + ( b_{13} + b_{14} )\sin{(ut)} + ( -b_{13} + b_{14} )\sin{(vt)} \right],\\
    y_1(t) = \frac{1}{2}\left[ ( b_{21} - b_{22} )\cos{(ut)} + ( b_{21} + b_{22} )\cos{(vt)} + ( b_{23} + b_{24} )\sin{(ut)} + ( -b_{23} + b_{24} )\sin{(vt)} \right],\\
    x_2(t) = \frac{1}{2}\left[ ( b_{31} - b_{32} )\cos{(ut)} + ( b_{31} + b_{32} )\cos{(vt)} + ( b_{33} + b_{34} )\sin{(ut)} + ( -b_{33} + b_{34} )\sin{(vt)} \right],\\
    y_2(t) = \frac{1}{2}\left[ ( b_{41} - b_{42} )\cos{(ut)} + ( b_{41} + b_{42} )\cos{(vt)} + ( b_{43} + b_{44} )\sin{(ut)} + ( -b_{43} + b_{44} )\sin{(vt)} \right].
    \end{cases}
    \end{equation*}

    As a result, we obtain:

    \begin{equation*}
    \begin{cases}
    x_1(t) = c_{11}\cos{(ut)} + c_{12}\cos{(vt)} + c_{13}\sin{(ut)} + c_{14}\sin{(vt)},\\
    y_1(t) = c_{21}\cos{(ut)} + c_{22}\cos{(vt)} + c_{23}\sin{(ut)} + c_{24}\sin{(vt)},\\
    x_2(t) = c_{31}\cos{(ut)} + c_{32}\cos{(vt)} + c_{33}\sin{(ut)} + c_{34}\sin{(vt)},\\
    y_2(t) = c_{41}\cos{(ut)} + c_{42}\cos{(vt)} + c_{43}\sin{(ut)} + c_{44}\sin{(vt)}.
    \end{cases}
    \end{equation*}

    If in each row only one of the constants is non-zero, then the trajectory is periodic (and it is important that the second index of the non-zero constants has the same parity in all rows).
    If in each row exactly two constants, whose second index has the same parity in all rows, are equal to zero, then such trajectories are also periodic.

    We will consider cases when the constants before the functions with arguments $ut$ and $vt$ do not vanish, and show that there exists an infinite set of $\theta_0$ for which the trajectories are periodic.

    The general form of each expression ($x_j$, $y_l$, $k=1, 2$, $l = 1, 2$) is the following:
    $$
    F(t) = a\cos{(ut)} + b\cos{(vt)} + c\sin{(ut)} + d\sin{(vt)}, 
    $$
    where $a$, $b$, $c$, $d$, $u$, $v \in \mathbb{R}$, $u \neq v$ are non-zero.

    There are two cases when this expression is periodic:

    \begin{itemize}
        \item[1)] when there exists such $T \in \mathbb{R}$ that $\forall t \in \mathbb{R}$ we have $F(t) = F\left(t+T\right)$ :
        \begin{equation}
        \label{full_period}
        a\cos{(ut)} + b\cos{(vt)} + c\sin{(ut)} + d\sin{(vt)} = a\cos{\left(u\left(t+T\right)\right)} + b\cos{\left(v\left(t+T\right)\right)} + c\sin{\left(u\left(t+T\right)\right)} + d\sin{\left(v\left(t+T\right)\right)}.
        \end{equation}
        \item[2)] when there exists such $T \in \mathbb{R}$ that $\forall t \in \mathbb{R}$:
        $$
        f((t + T)u) = f(ut), \quad g((t + T)v) = g(vt),
        $$
        where $f(\cdot)$, $g(\cdot) = \cos{(\cdot)}$ or $\sin{(\cdot)}$.
    \end{itemize}

    Case 2) is a special case of case 1). Indeed, let, for example, $f(\cdot) = g(\cdot) = \cos{(\cdot)}$, and $\exists$ such $T \in \mathbb{R}$ that $\forall t \in \mathbb{R}$:

    $$
    \cos{\left( \left(t + T\right)u\right)} = \cos{(ut)}, \quad \cos{\left(\left(t + T\right)v\right)} = \cos{(vt)}.
    $$
    
    Differentiating each of the relations with respect to $t$ and dividing by the non-zero factors $u$ and $v$, respectively, we obtain:

    $$
    \sin{\left( \left(t + T\right)u\right)} = \sin{(ut)}, \quad \sin{\left(\left(t + T\right)v\right)} = \sin{(vt)}.
    $$
    
    Multiplying each of the relations by the corresponding coefficient $a$, $b$, $c$, $d \in \mathbb{R}$ and summing them all, we obtain (\ref{full_period}). Relation (\ref{full_period}) is derived analogously if $f(\cdot) = \sin{(\cdot)}$, $g(\cdot) = \cos{(\cdot)}$; $f(\cdot) = \cos{(\cdot)}$, $g(\cdot) = \sin{(\cdot)}$; $f(\cdot) = g(\cdot) = \sin{(\cdot)}$.

    However, case 2) is interesting in itself, since, for example, if in each row only one constant is non-zero, but the parities of the second indices of these constants do not coincide, then the trajectory may not be periodic. However, when case 2) is satisfied, we just obtain the periodicity of the trajectory.

    Further considerations give an understanding that only case 2) makes sense. Case 2) is achieved when there are integer or rational relations between $u$ and $v$ ($u = kv$ or $u = \frac{k}{n}v$, $k$, $n \in \mathbb{Z}$, $k$, $n \neq 0$), which is shown in Lemma \ref{prop1}. In turn, when there is an irrational relation between $u$ and $v$ ($u = \alpha v$, $\alpha \in \mathbb{R} \setminus \mathbb{Q}$), even case 1) is not possible. Lemma \ref{prop2} considers a function of the form $F(t)$ with two non-zero coefficients. Lemma \ref{4funkcii} considers a function of the form $F(t)$ with four non-zero coefficients. The case with three non-zero coefficients is considered similarly to the case of four (two) non-zero coefficients, therefore its proof is not given.

    In our problem, $u$ and $v$ have specific functional dependencies on $\theta_0 \in \mathbb{R}$. We further show that the corresponding integer relations between $u$ and $v$ hold for a finite set of $\theta_0$ in Lemma \ref{celochisl}, and rational relations between $u$ and $v$ for an infinite set of $\theta_0$ in Lemma \ref{racion}. Thus, we prove that there exists an infinite set of $\theta_0$ for which the extremal trajectories are periodic, without being one-parameter subgroups.
    
    \end{proof}

    \begin{lemma}
    \label{prop1}
    Let $u$, $v \in \mathbb{R}$, $u \neq 0$, $v \neq 0$.
    \begin{itemize}
        \item[(1)] If $u = kv$, where $k \in \mathbb{Z}$, then $\forall t \in \mathbb{R}$ for $T = \frac{2\pi}{|v|}$ we have:
        $$
        f((t + T)u) = f(ut), \quad g((t + T)v) = g(vt),
        $$
        where $f(\cdot)$, $g(\cdot) = \cos{(\cdot)}$ or $\sin{(\cdot)}$;
        \item[(2)] If $u = \frac{k}{n}v$, where $n,\, k \in \mathbb{Z}$, $k \neq n$, $n \neq 0$, then $\forall t \in \mathbb{R}$ for $T = \frac{2|n| \pi}{|v|}$ we have:
        $$
        f((t + T)u) = f(ut), \quad g((t + T)v) = g(vt),
        $$
        where $f(\cdot)$, $g(\cdot) = \cos{(\cdot)}$ or $\sin{(\cdot)}$.
    \end{itemize}
    \end{lemma}

    \begin{proof}
    Since the period of $\sin{(t)}$ or $\cos{(t)}$ is $2\pi$, and the period of each of the terms separately is, respectively, $\frac{2\pi}{u}$ and $\frac{2\pi}{v}$, we will look for $T$ in the form $2\pi \beta$, $\beta \in \mathbb{R}$.
    $$
    (t + 2\pi \beta)u = tu + 2\pi \beta u, \quad (t + 2\pi \beta)v = tv + 2\pi \beta v.
    $$
    \begin{itemize}
        \item[(1)] Since $u = kv$, where $k \in \mathbb{Z}$, set $\beta = \frac{1}{v}$. Then
        $$
        tu + 2\pi \beta u = tu + \frac{2\pi k v}{v} = tu + 2\pi k, \quad tv + 2\pi \beta v  = tv + 2\pi.
        $$

        $$
        T = \frac{2\pi}{|v|}.
        $$

        \item[(2)] Since $u = \frac{k}{n} v$, where $n,\, k \in \mathbb{Z}$, $|k| < |n|$, set $\beta = \frac{n}{v}$. Then
        $$
        tu + 2\pi \beta u = tu + 2\pi\frac{ n }{v} \frac{k v}{n} = tu + 2\pi k, \quad tv + 2\pi \beta v  = tv + 2\pi v \frac{n}{v} = tv + 2\pi n.
        $$

        $$
        T = \frac{2\pi |n|}{|v|}.
        $$
    \end{itemize}
    \end{proof}

    \begin{lemma}
    \label{prop2}
    Consider the sum $F(t) = af(tu) \pm bg(tv)$, where $f(\cdot), g(\cdot) = \cos{(\cdot)}$ or $\sin{(\cdot)}$, $u$, $v \in \mathbb{R}$, $u$, $v$, $a$, $b \neq 0$, and $u = \alpha v$, where $\alpha \in \mathbb{R} \setminus \mathbb{Q}$. Then there does not exist such $T \in \mathbb{R}$, $T \neq 0$, that $\forall t \in \mathbb{R}$ the equality $F(t + T) = F(t)$ holds.
    \end{lemma}
    \begin{proof} 
    We will carry out the proof by contradiction.
        \begin{itemize}
        \item First, assume the existence of such $T \in \mathbb{R}$ that:
        $$
        f((t + T)u) = f(ut), \quad g((t + T)v) = g(vt).
        $$
        This is equivalent to
        $$
        tu + Tu = tu + T\alpha v = tu + 2l\pi, \quad tv + Tv = tv + 2m\pi, \quad l,\, m \in \mathbb{Z}.
        $$

        That is, on one hand, $T = \frac{2l\pi}{\alpha v}$. On the other hand, $T = \frac{2m \pi}{v}$. Let us find out if this is possible for any $l$, $m \in \mathbb{Z}$. We obtain the equality

        $$
        \frac{2l\pi}{\alpha v} = \frac{2m \pi}{v},
        $$

        which is equivalent to

        $$
        \alpha = \frac{l}{m},
        $$

        which cannot be due to the irrationality of $\alpha$.

        \item Consider another possibility. Assume the existence of such $T \in \mathbb{R}$ that:

        $$
        f((t+T)u) \pm g((t+T)v) = f(tu) \pm g(tv).
        $$
         
        \begin{itemize}
        
        \item Let $f = g = \sin{t}$. That is,
         \begin{equation}
        \label{sinsin1}
         \sin{(tu)} \pm \sin{(tv)} = \sin{((t+T)u)} \pm \sin{((t+T)v)}
         \end{equation}

         For $t = 0$:
         \begin{equation}
            \label{sinsin1_t0}
            0 = \sin{(Tu)} \pm \sin{(Tv)}.
         \end{equation}

         Differentiate (\ref{sinsin1}) with respect to $t$:
         $$
         u\cos{(tu)} \pm v\cos{(tv)} = u\cos{((t + T)u)} \pm v\cos{((t + T)v)},
         $$
         and differentiate once more with respect to $t$:
         $$
         -u^2\sin{(tu)} \mp v^2\sin{(tv)} = -u^2\sin{((t+T)u)} \mp v^2\sin{((t+T)v)}.
         $$

         Substitute $t = 0$ into the obtained equation:

         \begin{equation}
         \label{sinsin2_t0}
         0 = -u^2\sin{(Tu)} \mp v^2\sin{(Tv)}.
         \end{equation} 
         
        Multiply (\ref{sinsin1_t0}) by $u^2$ and add to (\ref{sinsin2_t0}):

        $$
        0 = \pm( u^2 - v^2 )\sin{(Tv)},
        $$

        which is equivalent to

        $$
        Tv = l\pi, \ l \in \mathbb{Z},
        $$

        since $u \neq v$.
        
        Multiply (\ref{sinsin1_t0}) by $v^2$ and add to (\ref{sinsin2_t0}):

        $$
        0 = (v^2 - u^2)\sin{(Tu)},
        $$

        which is equivalent to

        $$
        Tu = m\pi, \ m \in \mathbb{Z},
        $$

        since $u \neq v$. Thus:

        $$
        T = \frac{l\pi}{v} = \frac{m\pi}{u}, \ l,\, m \in \mathbb{Z},
        $$

        but $u = \alpha v$, $\alpha \in \mathbb{R} \setminus \mathbb{Q}$, therefore we obtain

        $$
        \frac{l}{v} = \frac{m}{\alpha v} \Leftrightarrow \alpha = \frac{m}{l}, \quad l,\, m \in \mathbb{Z},
        $$

        which contradicts our assumption.

        \item Let $f = g = \cos{t}$. That is,
         \begin{equation}
            \label{coscos1}
         \cos{(tu)} \pm \cos{(tv)} = \cos{((t+T)u)} \pm \cos{((t+T)v)}
         \end{equation}

         For $t = 0$:
         \begin{equation}
            \label{coscos1_t0}
            1 \pm 1 = \cos{(Tu)} \pm \cos{(Tv)}.
         \end{equation}

         Differentiate (\ref{coscos1}) with respect to $t$:
         $$
         -u\sin{(tu)} \mp v\sin{(tv)} = -u\sin{((t + T)u)} \mp v\sin{((t + T)v)},
         $$
         and differentiate once more with respect to $t$:
         $$
         -u^2\cos{(tu)} \mp v^2\cos{(tv)} = -u^2\cos{((t+T)u)} \mp v^2\cos{((t+T)v)}.
         $$

         Substitute $t = 0$ into the obtained equation:

         \begin{equation}
         \label{coscos2_t0}                
         -u^2 \mp v^2 = -u^2\cos{(Tu)} \mp v^2\cos{(Tv)}.   
         \end{equation} 
         
        Multiply (\ref{coscos1_t0}) by $u^2$ and add to (\ref{coscos2_t0}):

        $$
        \pm (u^2 - v^2) = \pm (u^2 - v^2)\cos{(Tv)},
        $$

        which is equivalent to
        $$
        Tv = 2l\pi, \ l \in \mathbb{Z},
        $$

        since $u \neq v$.

        Multiply (\ref{coscos1_t0}) by $v^2$ and add to (\ref{coscos2_t0}):

        $$
        (v^2 - u^2) = (v^2 - u^2)\cos{(Tu)},
        $$

        which is equivalent to
        $$
        Tu = 2m\pi, \ m \in \mathbb{Z},
        $$

        since $u \neq v$.

        \item Let $f = \sin{t}$, $g = \cos{t}$. That is,

        \begin{equation}
        \label{sincos1}
        \sin{(tu)} \pm \cos{(tv)} = \sin{((t+T)u)} \pm \cos{((t+T)v)}.
        \end{equation}

        For $t = 0$:

        \begin{equation}
        \label{sincos1_t0}
        \pm 1 = \sin{(Tu)} \pm \cos{(Tv)}.
        \end{equation}

        Differentiate (\ref{sincos1}) twice with respect to $t$:

        $$
        -u^2\sin{(tu)} \mp v^2\cos{(tv)} = -u^2\sin{((t+T)u)} \mp v^2\cos{((t+T)v)}.
        $$

        Substitute $t = 0$:

        \begin{equation}
        \label{sincos2_t0}
        \mp v^2 = -u^2\sin{(Tu)} \mp v^2\cos{(Tv)}.
        \end{equation}

        Multiply (\ref{sincos1_t0}) by $u^2$ and add to (\ref{sincos2_t0}):

        $$
        \pm (u^2 - v^2) = \pm (u^2 - v^2)\cos{(Tv)},
        $$

        which is equivalent to

        $$
        Tv = 2l\pi, \ l \in \mathbb{Z},
        $$

        since $u \neq v$.

        Multiply (\ref{sincos1_t0}) by $v^2$ and add to (\ref{sincos2_t0}):

        $$
        0 = (v^2 - u^2 )\sin{(Tu)},
        $$

        which is equivalent to

        $$
        Tu = m\pi, \ m \in \mathbb{Z},
        $$

        since $u \neq v$.

        \item Let $f = \cos{t}$, $g = \sin{t}$. That is,

        \begin{equation}
        \label{cossin1}
        \cos{(tu)} \pm \sin{(tv)} = \cos{((t + T)u)} \pm \sin{((t + T)v)}.
        \end{equation}

        For $t = 0$:

        \begin{equation}
        \label{cossin1_t0}
        1 = \cos{(Tu)} \pm \sin{(Tv)}.
        \end{equation}

        Differentiate (\ref{cossin1}) twice with respect to $t$:

        $$
        -u^2\cos{(tu)} \mp v^2\sin{(tv)} = -u^2\cos{((t + T)u)} \mp v^2 \sin{((t + T)v)}.
        $$

        Substitute $t = 0$:

        \begin{equation}
        \label{cossin2_t0}
        -u^2 = -u^2\cos{(Tu)} \mp v^2 \sin{(Tv)}.
        \end{equation}

        Multiply (\ref{cossin1_t0}) by $u^2$ and add to (\ref{cossin2_t0}):

        $$
        0 = \pm (u^2 - v^2)\sin{(Tv)},
        $$

        which is equivalent to

        $$
        Tv = l\pi, \ l \in \mathbb{Z},
        $$

        since $u \neq v$.

        Multiply (\ref{cossin1_t0}) by $v^2$ and add to (\ref{cossin2_t0}):

        $$
        v^2 - u^2 = (v^2 - u^2)\cos{(Tu)},
        $$

        which is equivalent to

        $$
        Tu = 2m\pi, \ m \in \mathbb{Z},
        $$

        since $u \neq v$.
        
        \end{itemize}
        
        \end{itemize}
    \end{proof}

    \begin{lemma}
    \label{4funkcii}
    Consider the function $F(t) = a\cos{(ut)} + b\cos{(vt)} + c\sin{(ut)} + d\sin{(vt)}$, where $a$, $b$, $c$, $d$, $u$, $v \in \mathbb{R}$ are non-zero. Moreover, $u = \alpha v$, $\alpha \in \mathbb{R} \setminus \mathbb{Q}$. Then the function $F(t)$ is non-periodic.
    \end{lemma}

    \begin{proof}
    Actually, the first case has already been considered: we assume the existence of such $T \in \mathbb{R}$ that:
        $$
        f\left((t+T)u\right) = f(ut),\ g\left((t+T)v\right) = g(vt),
        $$
        where $f$, $g$ are equal to $\cos$ or $\sin$. We arrive at a contradiction with the irrationality of the number $\alpha$.
        
    The second case consists in assuming the existence of such $T \in \mathbb{R}$ that $F(t) = F\left(t+T\right)$ $\forall t \in \mathbb{R}$:
        \begin{equation}
        \label{full_period1}
        a\cos{(ut)} + b\cos{(vt)} + c\sin{(ut)} + d\sin{(vt)} = a\cos{\left(u\left(t+T\right)\right)} + b\cos{\left(v\left(t+T\right)\right)} + c\sin{\left(u\left(t+T\right)\right)} + d\sin{\left(v\left(t+T\right)\right)}.
        \end{equation}
        We will obtain 4 equations, from which we will also derive a contradiction with the irrationality of the number $\alpha$.

        Differentiate (\ref{full_period}) 3 times with respect to $t$, and then substitute $t = 0$ into each obtained equation.

        $d/dt$(\ref{full_period1}):

        \begin{equation}
        \label{first_der}
        -au\sin{(ut)} - bv\sin{(vt)} + cu\cos{(ut)} + dv\cos{(vt)} = -au\sin{\left(u\left(t+T\right)\right)} - bv\sin{\left(v\left(t+T\right)\right)} + cu\cos{\left(u\left(t+T\right)\right)} + dv\cos{\left(v\left(t+T\right)\right)}.
        \end{equation}

        $d^2/dt^2$(\ref{full_period1}):

        \begin{equation}
        \label{sec_der}
        au^2\cos{(ut)} + bv^2\cos{(vt)} + cu^2\sin{(ut)} + dv^2\sin{(vt)} = au^2\cos{\left(u\left(t+T\right)\right)} + bv^2\cos{\left(v\left(t+T\right)\right)} + cu^2\sin{\left(u\left(t+T\right)\right)} + dv^2\sin{\left(v\left(t+T\right)\right)}.
        \end{equation}

        $d^3/dt^3$(\ref{full_period1}):

        \begin{equation}
        \label{third_der}
        au^3\sin{(ut)} + bv^3\sin{(vt)} - cu^3\cos{(ut)} - dv^3\cos{(vt)} = au^3\sin{\left(u\left(t+T\right)\right)} + bv^3\sin{\left(v\left(t+T\right)\right)} - cu^3\cos{\left(u\left(t+T\right)\right)} - dv^3\cos{\left(v\left(t+T\right)\right)}.
        \end{equation}

        Substitute $t=0$ into (\ref{full_period1}), (\ref{first_der}), (\ref{sec_der}), (\ref{third_der}):

        \begin{equation}
        \label{t0_zeroth}
        a + b = a\cos{\left(uT\right)} + b\cos{\left(vT\right)} + c\sin{\left(uT\right)} + d\sin{\left(vT\right)}.
        \end{equation}

        \begin{equation}
        \label{t0_first}
        cu+dv = -au\sin{\left(uT\right)} - bv\sin{\left(vT\right)} + cu\cos{\left(uT\right)} + dv\cos{\left(vT\right)}.
        \end{equation}

        \begin{equation}
        \label{t0_second}
        au^2+bv^2 = au^2\cos{\left(uT\right)} + bv^2\cos{\left(vT\right)} + cu^2\sin{\left(uT\right)} + dv^2\sin{\left(vT\right)}.
        \end{equation}

        \begin{equation}
        \label{t0_third}
        -cu^3 - dv^3 = au^3\sin{\left(uT\right)} + bv^3\sin{\left(vT\right)} - cu^3\cos{\left(uT\right)} - dv^3\cos{\left(vT\right)}.
        \end{equation}

        $u^2$(\ref{t0_zeroth}) $-$ (\ref{t0_second}):

        $$
        (u^2-v^2)b = (u^2 - v^2)(b\cos{\left( vT \right)} + d\sin{\left(vT \right)}) \Leftrightarrow b = b\cos{\left( vT \right)} + d\sin{\left(vT \right)}.
        $$

        $v^2$(\ref{t0_zeroth}) $-$ (\ref{t0_second}):

        $$
        (v^2-u^2)a = (v^2-u^2)\left( a\cos{\left(uT \right)} + c\sin{\left(uT \right)} \right) \Leftrightarrow a = a\cos{\left(uT \right)} + c\sin{\left(uT \right)}.
        $$

        $u^2$(\ref{t0_first}) $+$ (\ref{t0_third}):
        
        $$
        (u^2-v^2)vd = (u^2-v^2)v\left( d\cos{\left(vT \right)} - b\sin{\left(uT \right)} \right) \Leftrightarrow d = - b\sin{\left(uT \right)} + d\cos{\left(vT \right)}.
        $$

        $u^2$(\ref{t0_first}) $+$ (\ref{t0_third}):

        $$
        (v^2-u^2)uc = ( u^2 - v^2 )u\left( a\sin{\left(uT \right)} - c\cos{\left(uT \right)}\right) \Leftrightarrow c = -a\sin{\left(uT \right)} + c\cos{\left(uT \right)}.
        $$

        We have obtained the following relations:

        $$
        \begin{pmatrix}
        a \\
        c
        \end{pmatrix} = \begin{pmatrix}
        \cos{\left(uT \right)} & \sin{\left(uT \right)}\\
        -\sin{\left(uT \right)} & \cos{\left(uT \right)}
        \end{pmatrix}\begin{pmatrix}
        a \\
        c
        \end{pmatrix}; \quad \begin{pmatrix}
        b \\
        d
        \end{pmatrix} = \begin{pmatrix}
        \cos{\left(vT \right)} & \sin{\left(vT \right)}\\
        -\sin{\left(vT \right)} & \cos{\left(vT \right)}
        \end{pmatrix}\begin{pmatrix}
        b \\
        d
        \end{pmatrix}.
        $$

        That is, the pairs of numbers $(a, c)$ and $(b, d)$ are mapped onto themselves by rotation matrices through angles $uT$ and $vT$, respectively.

        From this we conclude:
        $$
        uT = 2l\pi, \quad vT = 2m\pi,\quad l,\ m \in \mathbb{Z}. 
        $$
        Thus, we obtain a contradiction with the fact that $\alpha \in \mathbb{R} \setminus \mathbb{Q}$.
    \end{proof}

    \begin{lemma}
    \label{celochisl}
    Let $u = \ch{\theta_0} + \frac{\sqrt{\ch{2\theta_0}}}{2}$, $v = \ch{\theta_0} - \frac{\sqrt{\ch{2\theta_0}}}{2} $, then there are integer relations $u = kv$, $k \in \mathbb{Z} \setminus {0}$ for
    $$
    \theta_0 = 0,\ k = 3, \quad \theta_0 = \pm \arcosh{\frac{5}{14}},\ k = 4, \quad \theta_0 = \pm \arcosh{\frac{3}{4}},\ k = 5.
    $$
    \end{lemma}

    \begin{proof}

    Is any of the following relations possible:

    \begin{equation*}
    \ch{\theta_0} + \frac{\sqrt{\ch{2\theta_0}}}{2} = k \left( \ch{\theta_0} - \frac{\sqrt{\ch{2\theta_0}}}{2} \right) \quad \text{or} \quad \ch{\theta_0} - \frac{\sqrt{\ch{2\theta_0}}}{2} = k\left( \ch{\theta_0} + \frac{\sqrt{\ch{2\theta_0}}}{2} \right), \quad k \in \mathbb{Z}?
    \end{equation*}

    \begin{align*}
    & \ch{\theta_0} + \frac{\sqrt{\ch{2\theta_0}}}{2} = k \left( \ch{\theta_0} - \frac{\sqrt{\ch{2\theta_0}}}{2} \right) \quad \quad \ch{\theta_0} - \frac{\sqrt{\ch{2\theta_0}}}{2} = k\left( \ch{\theta_0} + \frac{\sqrt{\ch{2\theta_0}}}{2} \right) \\
    & (k+1)\sqrt{\ch{2\theta_0}} = 2(k-1)\ch{\theta_0} \quad \quad \quad \quad \quad \quad 2(1-k)\ch{\theta_0} = (k+1)\sqrt{\ch{2\theta_0}} \\
    & \frac{k-1}{k+1} = \frac{\sqrt{\ch{2\theta_0}}}{2\ch{\theta_0}} \quad \quad \quad \quad \quad \quad \quad \quad \quad \quad \quad \quad \frac{1-k}{k+1} = \frac{\sqrt{\ch{2\theta_0}}}{2\ch{\theta_0}}.
    \end{align*}

\begin{itemize}

\item The relation $(k+1)\sqrt{\ch 2 \theta_0} = 2 (k-1) \ch \theta_0$ is possible:

\begin{align*}
&(k+1)\sqrt{\ch 2 \theta_0} = 2 (k-1) \ch \theta_0, \\
&(k+1)^2(2 \ch^2 \theta_0 - 1) = 4 (k-1)^2 \ch^2 \theta_0, \\
&(12 k - 2 k^2 - 2) \ch^2 \theta_0 = (k+1)^2.
    \end{align*}
    Let us find out for which $k \in \mathbb{Z}$ $\ch \theta_0 = \frac{k+1}{\sqrt{12 k - 2 k^2 - 2}} \geqslant 1$. To do this, consider an auxiliary real function:

    $$
    q(x) = \frac{x + 1}{\sqrt{-2x^2 + 12x - 2}}.
    $$

    First, let's deal with the domain of definition:
    \begin{align*}
    & -2x^2 + 12x - 2 \geqslant 0 \Leftrightarrow x^2 - 6x + 1 \leqslant 0.\\
    & x^2 - 6x + 1 = 0. \quad D/4 = 9 - 1 = 8, \quad x_{1,2} = 3 \pm 2\sqrt{2}.\\
    & 0 < 3-2\sqrt{2} < 1, \quad 5 < 3+2\sqrt{2} < 6.\\
    & -2x^2 + 12x - 2 \geqslant 0 \Leftrightarrow \\
    & \Leftrightarrow 3-2\sqrt{2} \leqslant x \leqslant 3 + 2\sqrt{2}.
    \end{align*}

    Domain of definition: $x \in \left(  3-2\sqrt{2}, 3 + 2\sqrt{2} \right)$.

    Let's find the limits of the function as it approaches the boundaries of the domain:

    \begin{align*}
    & \underset{x \rightarrow (3-2\sqrt{2}) + 0}{lim} q(x) = \underset{x \rightarrow (3-2\sqrt{2}) + 0}{lim} \frac{x + 1}{\sqrt{-2x^2 + 12x - 2}} = +\infty, \\
    & \underset{x \rightarrow (3+2\sqrt{2}) - 0}{lim} q(x) = \underset{x \rightarrow (3+2\sqrt{2}) - 0}{lim} \frac{x + 1}{\sqrt{-2x^2 + 12x - 2}} = +\infty.
    \end{align*}

    Then solve the equation $q(x) = 1$:

    \begin{align*}
    & \frac{x + 1}{\sqrt{-2x^2 + 12x - 2}} = 1 \Leftrightarrow x + 1 = \sqrt{-2x^2 + 12x - 2} \Leftrightarrow (x+1)^2 = -2x^2 +12x - 2, \quad 3-2\sqrt{2} \leqslant x \leqslant 3 + 2\sqrt{2} \Leftrightarrow \\
    & \Leftrightarrow x^2 + 2x + 1 = -2x^2 + 12x - 2 \Leftrightarrow 3x^2 - 10x + 3 = 0 \quad D/4 = 25 - 9 = 16, \quad x_{1,2} = \frac{5 \pm 4}{3}, \quad x_1 = \frac{1}{3},\ x_2 = 3.
    \end{align*}

    Finally, determine the intervals of monotonicity:

    \begin{align*}
    & \frac{d q(x)}{dx} = \frac{d}{dx}\frac{x + 1}{\sqrt{-2x^2 + 12x - 2}} = \frac{\sqrt{-2x^2 + 12x - 2} - (x+1)\frac{-4x + 12}{2\sqrt{-2x^2 + 12x - 2}}}{\left( \sqrt{-2x^2 + 12x - 2}\right)^2} = \frac{-2x^2 + 12x - 2 + (x+1)(2x-6)}{\left( -2x^2 + 12x - 2 \right)^{3/2}} = \\
    & = \frac{-2x^2 + 12x - 2 + 2x^2 - 4x - 6}{\left( -2x^2 + 12x - 2 \right)^{3/2}} = \frac{8x - 8}{\left( -2x^2 + 12x - 2 \right)^{3/2}}.
    \end{align*}

    We obtain that $x = 1$ is a local extremum point.

    $$
    q(1) = \frac{2}{\sqrt{-2 + 12 -2}} = \frac{\sqrt{2}}{2} < 1.
    $$

    For $x<1$, the function $q(x)$ decreases, and for $x > 1$, the function $q(x)$ increases. $x = 1$ is a minimum point, since as it approaches the boundaries, the function tends to $+\infty$.

    Thus,

    \begin{equation*}
    \frac{x + 1}{\sqrt{-2x^2 + 12x - 2}} \geqslant 1 \Leftrightarrow x \in [3-2\sqrt{2}, \frac{1}{3}] \cup [3, 3 + 2\sqrt{2}].
    \end{equation*}

    And the possible integer values are, respectively, $k = 3$, $4$, $5$. And the values of $\theta_0$ are thus:

    $$
    \theta_0 = \pm \arcosh{\frac{4}{\sqrt{36-20}}} = 0, \quad \theta_0 = \pm \arcosh{\frac{5}{\sqrt{48 - 34}}} = \pm \arcosh{\frac{5}{\sqrt{14}}}, \quad \theta_0 = \pm \arcosh{\frac{6}{\sqrt{60 - 52}}} = \pm \arcosh{\frac{3}{\sqrt{2}}}.
    $$

    \item Consider the other relation: $2(1-k)\ch{\theta_0} = (k+1)\sqrt{\ch{2\theta_0}}$. Here $k$ can only be $-1$, $0$, or $1$ due to the signs of the expressions on the right and left sides.

    \begin{itemize}
        \item $k = -1$ cannot be:
        $$
        4\ch{\theta_0} = 0.
        $$
        \item $k = 1$ cannot be:
        $$
        2 \sqrt{\ch{2\theta_0}} = 0.
        $$
        \item $k = 0$ cannot be:
        $$
        2\ch{\theta_0} = \sqrt{\ch{2\theta_0}} \Leftrightarrow 4\ch^2{\theta_0} = 2\ch^2{\theta_0} - 1 \Leftrightarrow 2\ch^2{\theta_0} = -1.
        $$
    \end{itemize}

    \end{itemize}

    \end{proof}

    \begin{lemma}
    \label{racion}
    Let $u = \ch{\theta_0} + \frac{\sqrt{\ch{2\theta_0}}}{2}$, $v = \ch{\theta_0} - \frac{\sqrt{\ch{2\theta_0}}}{2} $, then there are rational relations between $u$ and $v$ for an infinite set of $\theta_0 \in \mathbb{R}$.
    \end{lemma}

    \begin{proof}
    Let us find out if case (2) from Proposition \ref{prop1} is possible in our situation:

    \begin{equation}
    \label{chosinusy2}
    n \left( \ch{\theta_0} + \frac{\sqrt{\ch{2\theta_0}}}{2} \right) = k \left( \ch{\theta_0} - \frac{\sqrt{\ch{2\theta_0}}}{2} \right) \quad n,\, k \in \mathbb{Z}.
    \end{equation}

    \begin{equation*}
    (k-n)\ch{\theta_0} = (n+k)\frac{\sqrt{\ch{2\theta_0}}}{2}.
    \end{equation*}

        \begin{align*}
        & 4(k-n)^2\ch^2{\theta_0} = (n+k)^2 \ch{(2\theta_0)} \Leftrightarrow 4(k-n)^2 \ch^2{\theta_0} = (n+k)^2(2\ch^2{\theta_0} - 1) \Leftrightarrow \\
        & \Leftrightarrow \ch^2{\theta_0}\left( 4(k-n)^2 - 2(n+k)^2 \right) = -(n+k)^2 \Leftrightarrow \\
        & \Leftrightarrow \ch^2{\theta_0} = \frac{(n+k)^2}{2(n+k)^2 - 4(k-n)^2} = \frac{k^2 + n^2 - 2kn}{12nk - 2k^2 - 2n^2}.
        \end{align*} 

    We have reduced the problem to finding out for which $n$, $k \in \mathbb{Z}$

    $$
    \ch{\theta_0} = \frac{|k+n|}{\sqrt{12nk - 2k^2 - 2n^2}} \geqslant 1?
    $$

    Consider the corresponding function of two real variables:

    $$
    q(x,y) = \frac{|x+y|}{\sqrt{12xy - 2x^2 - 2y^2}}, \ x \neq 0, \ y \neq 0.
    $$

    First, let's deal with the domain of definition:

    $$
    6xy - x^2 - y^2 \geqslant 0 \Leftrightarrow y^2 - 6xy + x^2 \leqslant 0
    $$
    
    $$
    D/4 = 9x^2 - x^2 = 8x^2 \geqslant 0
    $$

    $$
    y^2 - 6xy + x^2 = 0 \Leftrightarrow y = \left( 3 \pm 2\sqrt{2} \right)x
    $$

    The domain of definition consists of the union of two regions:
    \begin{itemize}
    \item[(1)]
    $$
    \left( 3 - 2\sqrt{2} \right)x < y < \left( 3 + 2\sqrt{2} \right)x, \ x > 0.
    $$
    \item[(2)]
    $$
    \left( 3 + 2\sqrt{2} \right)x < y < \left( 3 - 2\sqrt{2} \right)x, \ x < 0.
    $$
    \end{itemize}

    Expand the modulus of our function $q(x,y)$:

    $$
    |x+y| = x+y \Leftrightarrow x+y \geqslant 0 \Leftrightarrow y \geqslant -x.
    $$

    $$
    |x+y| = -x-y \Leftrightarrow x+y < 0 \Leftrightarrow y < -x.
    $$

    Thus, on set (1) we have $|x+y| = x+y$, and on set (2) we have $|x+y| = -x-y$.

    We will act similarly on each of the regions as follows:

    \begin{itemize}
    \item[1.] Compute the limits of the function $q(x, y)$ on the boundary of the set;
    \item[2.] Compute the partial derivative $\partial_y q(x, y)$ and find intervals of monotonicity for each fixed $x$;
    \item[3.] Compare the values at the extremum points with the values on the boundaries and find out where $q(x, y) \geqslant 1$.
    \end{itemize}
    \begin{itemize}
        \item[(1)] 
        \begin{itemize}
            \item[1.] 
            \begin{align*}
            & \underset{x = const > 0,\ y \rightarrow (3-2\sqrt{2})x + 0}{lim} q(x,y) = \underset{x = const > 0,\ y \rightarrow (3-2\sqrt{2})x + 0}{lim}  \frac{x + (3-2\sqrt{2})x}{\sqrt{12x(3-2\sqrt{2})x - 2x^2 - 2(3-2\sqrt{2})^2x^2}} = +\infty;\\
            & \underset{x = const > 0,\ y \rightarrow (3+2\sqrt{2})x - 0}{lim}  q(x,y) = \underset{x = const > 0,\ y \rightarrow (3+2\sqrt{2})x - 0}{lim} \frac{x + \left( 3 + 2\sqrt{2} \right)x}{\sqrt{12x\left( 3 + 2\sqrt{2} \right)x - 2x^2 - 2\left( 3 + 2\sqrt{2} \right)^2x^2}} = + \infty.
            \end{align*}
            
            \item[2.]
            \begin{align*}
            & \partial_y q(x, y) = \frac{\sqrt{12xy - 2x^2 - 2y^2} - (x+y)\frac{6x-2y}{\sqrt{12xy - 2x^2 - 2y^2}}}{\left(\sqrt{12xy - 2x^2 - 2y^2}\right)^2} = \frac{12xy-2x^2-2y^2 - 6x^2-4xy + 2y^2}{\left( 12xy - 2x^2 - 2y^2 \right)^{3/2}} = \\
            & = \frac{-8x^2 + 8xy}{\left( 12xy - 2x^2 - 2y^2 \right)^{3/2}}.
            \end{align*}
            $$
            \partial_y q(x, y) = 0 \Leftrightarrow -8x^2 + 8xy = 0 \Leftrightarrow x(-x + y) = 0 \Leftrightarrow y = x.
            $$

            For $y > x$ and each fixed $x > 0$, $q(x, y)$ increases; for $y < x$, $q(x, y)$ decreases. Hence, $x$ is a local minimum point.

            $$
            q(x,x) = \frac{2x}{\sqrt{12x^2 - 2x^2 - 2x^2}} = \frac{2x}{2x\sqrt{2}} = \frac{\sqrt{2}}{2} < 1.
            $$
            
            \item[3.]
            Since the function tends to $+\infty$ as it approaches the boundary of the region, $\exists$ points $y_1 < x$ and $y_2 > x$ where $q(x, y)$ takes the value $1$. Accordingly, for $(3-2\sqrt{2})x < y \leqslant y_1$ and $y_2 \leqslant y < (3+2\sqrt{2})x$, we have $q(x, y) \geqslant 1$.

            \begin{align*}
            & \frac{x+y}{\sqrt{12xy - 2x^2 - 2y^2}} = 1 \Leftrightarrow x+y = \sqrt{12xy - 2x^2 - 2y^2} \Leftrightarrow x^2 + 2xy + y^2 = 12xy - 2x^2 - 2y^2 \Leftrightarrow \\
            & \Leftrightarrow 3y^2 - 10xy + 3x^2 = 0. \\
            & D/4 = 25x^2 - 9x^2 = 16x^2 \geqslant 0;\\
            & y = \frac{5 \pm 4}{3}x; \\
            & y_2 = 3x > x > y_1 = \frac{1}{3}x.
            \end{align*}
            
        \end{itemize}
        
        \item[(2)] 
        \begin{itemize}
            \item[1.]
            \begin{align*}
            & \underset{x = const < 0,\ y \rightarrow (3-2\sqrt{2})x - 0}{lim} q(x,y) = \underset{x = const < 0,\ y \rightarrow (3-2\sqrt{2})x - 0}{lim} \frac{-x - (3-2\sqrt{2})x}{\sqrt{12x(3-2\sqrt{2})x - 2x^2 - 2(3-2\sqrt{2})^2x^2}} = +\infty;\\
            & \underset{x = const < 0,\ y \rightarrow (3+2\sqrt{2})x + 0}{lim} q(x,y) = \underset{x = const < 0,\ y \rightarrow (3+2\sqrt{2})x + 0}{lim} \frac{-x - \left( 3 + 2\sqrt{2} \right)x}{\sqrt{12x\left( 3 + 2\sqrt{2} \right)x - 2x^2 - 2\left( 3 + 2\sqrt{2} \right)^2x^2}} = +\infty.
            \end{align*}
            \item[2.]
            \begin{align*}
            & \partial_y q(x,y) = \frac{-\sqrt{12xy - 2x^2 - 2y^2} + (x+y)\frac{6x-2y}{\sqrt{12xy - 2x^2 - 2y^2}}}{\left( \sqrt{12xy - 2x^2 - 2y^2} \right)^2} = \frac{-12xy+2x^2+2y^2 + 6x^2+4xy-2y^2}{\left( 12xy - 2x^2 - 2y^2 \right)^{3/2}} = \\
            & = \frac{8x^2 - 8xy}{\left( 12xy - 2x^2 - 2y^2 \right)^{3/2}}.
            \end{align*}
            $$
            \partial_y q(x, y) = 0 \Leftrightarrow 8x^2 - 8xy = 0 \Leftrightarrow x(x - y) = 0 \Leftrightarrow y = x.
            $$

            For $y > x$ and each fixed $x < 0$, the function $q(x, y)$ increases; for $y < x$, the function $q(x, y)$ decreases. Hence, $x$ is a local minimum point.

            $$
            q(x,x) = \frac{-2x}{\sqrt{12x^2 - 2x^2 - 2x^2}} = \frac{-2x}{-2x\sqrt{2}} = \frac{\sqrt{2}}{2} < 1.
            $$
            
            \item[3.]
            Since the function tends to $+\infty$ as it approaches the boundary of the region, $\exists$ points $y_1 < x$ and $y_2 > x$ where $q(x, y)$ takes the value $1$. Accordingly, for $(3+2\sqrt{2})x < y \leqslant y_1$ and $y_2 \leqslant y < (3-2\sqrt{2})x$, we have $q(x, y) \geqslant 1$.

            \begin{align*}
            & \frac{-x-y}{\sqrt{12xy - 2x^2 - 2y^2}} = 1 \Leftrightarrow -x-y = \sqrt{12xy - 2x^2 - 2y^2} \Leftrightarrow x^2 + 2xy + y^2 = 12xy - 2x^2 - 2y^2 \Leftrightarrow \\
            & \Leftrightarrow 3y^2 - 10xy + 3x^2 = 0. \\
            & D/4 = 25x^2 - 9x^2 = 16x^2 \geqslant 0;\\
            & y = \frac{5 \pm 4}{3}x; \\
            & y_1 = 3x < x < y_2 = \frac{1}{3}x.
            \end{align*}
            
        \end{itemize}
                    
    \end{itemize}

    Now let's show that the obtained union of regions contains an infinite number of integer pairs.

    Consider one of the regions, for example, $ 3x < y < (3+2\sqrt{2})x$, $x > 0$. Select integer $x = k \in \mathbb{Z}$. Then $3k < y < (3+2\sqrt{2})k$.

    The difference between the left and right boundaries is $2k\sqrt{2}$, and $2 < 2\sqrt{2} < 3$. Therefore, as $k$ increases, the number of integers $y$ belonging to this interval increases. Hence, there are infinitely many such pairs.

    \end{proof}

\section{Sectional Curvature}

We have the orthonormal case: $-g(X_1,X_1) = g(X_2,X_2) = g(X_3,X_3) = 1$. And also: $[X_1, X_2] = X_3$, $[X_3, X_1] = X_2$, $[X_2, X_3] = X_1$, i.e. $c_{12}^3 = 1 = c_{31}^2 = c_{23}^1$, $c_{21}^3 = -1 = c_{13}^2 = c_{32}^1$, $c_{ij}^k = 0$ others. Note that $c_{32}^1 - c_{21}^3 - c_{13}^2 = -1 + 1 + 1 = 1$.

We use the formula obtained in \cite{curv_lor}.

Take vector fields $v = v^1X_1 + v^2X_2 + v^3X_3$, $w = w^1X_1 + w^2X_2 + w^3X_3$, where $v^j = \const$, $w^j = \const$, $j=1,2,3$.

\begin{theorem}
\label{sect_curv}
The sectional curvature of the group $\SU(2)$ on the section spanned by the vector fields $v$ and $w$ is given by the following formula:
\begin{equation}
\label{sect_kriv}
K = \frac{w^1w^1\left( \frac{v^2v^2}{4} - \frac{v^3v^3}{4} \right) + w^2w^2\left( -\frac{v^1v^1}{4} + \frac{3v^3v^3}{4} \right) + w^3w^3\left( -\frac{v^1v^1}{4} + \frac{7v^2v^2}{4} \right) - w^1w^3\frac{v^1v^3}{2} - w^2w^3\frac{5v^2v^3}{2}}{\left( -(v^1)^2 + \sum_{k=2}^3(v^k)^2 \right) \cdot \left( -(w^1)^2 + \sum_{k=2}^3(w^k)^2 \right) - \left( -v^1w^1 + v^2w^2 + v^3w^3 \right)^2}.
\end{equation}
\end{theorem}

\begin{example}
$v = X_1$, $w = X_2$:
\begin{equation*}
K = \frac{-\frac{1}{4}}{-1} = \frac{1}{4}.
\end{equation*}
$v = X_2$, $w = X_1$:
\begin{equation*}
K = \frac{\frac{1}{4}}{-1} = -\frac{1}{4}.
\end{equation*}
\end{example}

\begin{example}
$v = X_1$, $w = X_3$:
\begin{equation*}
K = \frac{-\frac{1}{4}}{-1} = \frac{1}{4}.
\end{equation*}
$v = X_3$, $w = X_1$:
\begin{equation*}
K = \frac{-\frac{1}{4}}{-1} = \frac{1}{4}.
\end{equation*}
\end{example}

\begin{example}
$v = X_2$, $w = X_3$:
\begin{equation*}
K = \frac{\frac{7}{4}}{1} = \frac{7}{4}.
\end{equation*}
$v = X_3$, $w = X_2$:
\begin{equation*}
K = \frac{\frac{7}{4}}{1} = \frac{7}{4}.
\end{equation*}
\end{example}

\begin{corollary}
The group $\SU(2)$ with the considered Lorentzian structure has variable sectional curvature, therefore it is locally non-isometric to Lorentzian manifolds of constant curvature: Minkowski space, de Sitter space, and anti-de Sitter space.
\end{corollary}

\begin{proof}[Proof of Theorem \ref{sect_curv}]
The general formula from \cite{curv_lor} is as follows:
$$
K = \frac{g(R_{v,w}v, w)}{Q(v,w)} = \frac{-w^1R_{v, w, v}^1 + w^2R_{v,w,v}^2 + w^3R_{v,w,v}^3}{\left( -(v^1)^2 + \sum_{k=2}^3(v^k)^2 \right) \cdot \left( -(w^1)^2 + \sum_{k=2}^3(w^k)^2 \right) - \left( -v^1w^1 + v^2w^2 + v^3w^3 \right)^2},
$$
where for $k=1,2,3$:

$$
R_{v,w,v}^k = \sum_{l,m = 1}^3\left( w^lD_{v,v,l,m}^k - v^lD_{w,v,l,m}^k \right) + \sum_{l=1}^3\left( \sum_{m=1}^2\sum_{m<j\leqslant 3}(v^mw^j - v^jw^m)c_{m,j}^l + \sum_{m=1}^3\left( \left( (v^mX_m)w^l \right) - w^m\left( X_mv^l \right) \right) \right)D_{v,l}^k;
$$

\begin{align*}
    & D_{v,1}^1 = 0,\qquad D_{v,1}^2 = - \frac{v^3}{2},\qquad D_{v,1}^3 = \frac{3v^2}{2};\\
    & D_{v,2}^1 =  \frac{v^3}{2},\qquad D_{v,2}^2 = 0,\qquad D_{v,2}^3 = \frac{v^1}{2},\\
    & D_{v,3}^1 = -\frac{v^2}{2},\qquad D_{v,3}^2 = -\frac{v^1}{2},\qquad D_{v,3}^3 = 0;\\
    & D_{w,v,1,1}^1 = 0,\qquad D_{w,v,1,1}^2 = - \frac{3w^1v^2}{4},\qquad  D_{w,v,1,1}^3 =  -\frac{3w^1v^3}{4}; \\ 
    & D_{w,v,2,1}^1 =  \frac{3w^1v^2}{4},\qquad D_{w,v,2,1}^2 = 0,\qquad  D_{w,v,2,1}^3 = 0; \\ & D_{w,v,3,1}^1 = \frac{w^1v^3}{4},\qquad D_{w,v,3,1}^2 = 0,\qquad  D_{w,v,3,1}^3 = 0;\\
    & D_{w,v,1,2}^1 = 0,\qquad D_{w,v,1,2}^2 = - \frac{w^2v^1}{4},\qquad D_{w,v,1,2}^3 = 0;\\
    & D_{w,v,2,2}^1 = \frac{w^2v^1}{4},\qquad D_{w,v,2,2}^2 = 0,\qquad D_{w,v,2,2}^3 = \frac{w^2v^3}{4};\\
    & D_{w,v,3,2}^1 = 0,\qquad D_{w,v,3,2}^2 = -\frac{w^2v^3}{4},\qquad D_{w,v,3,2}^3 = 0;\\
    & D_{w,v,1,3}^1 = 0,\qquad D_{w,v,1,3}^2 = 0,\qquad D_{w,v,1,3}^3 = -\frac{3w^3v^1}{4};\\
    & D_{w,v,2,3}^1 = 0,\qquad D_{w,v,2,3}^2 = 0,\qquad D_{w,v,2,3}^3 = -\frac{w^3v^2}{4};\\
    & D_{w,v,3,3}^1 = \frac{w^3v^1}{4},\qquad D_{w,v,3,3}^2 = \frac{w^3v^2}{4},\qquad D_{w,v,3,3}^3 = 0;\\
    & D_{v,v,1,1}^1 = 0,\qquad D_{v,v,1,1}^2 = - \frac{3v^1v^2}{4},\qquad D_{v,v,1,1}^3 =  -\frac{3v^1v^3}{4}; \\
    & D_{v,v,2,1}^1 =  \frac{3v^1v^2}{4},\qquad D_{v,v,2,1}^2 = 0,\qquad D_{v,v,2,1}^3 = 0; \\
    & D_{v,v,3,1}^1 = \frac{v^1v^3}{4},\qquad D_{v,v,3,1}^2 = 0,\qquad D_{v,v,3,1}^3 = 0;\\
    & D_{v,v,1,2}^1 = 0,\qquad D_{v,v,1,2}^2 = - \frac{v^2v^1}{4},\qquad D_{v,v,1,2}^3 = 0;\\
    & D_{v,v,2,2}^1 = \frac{v^2v^1}{4},\qquad D_{v,v,2,2}^2 = 0,\qquad D_{v,v,2,2}^3 = \frac{v^2v^3}{4};\\
    & D_{v,v,3,2}^1 = 0,\qquad D_{v,v,3,2}^2 = -\frac{v^2v^3}{4},\qquad D_{v,v,3,2}^3 = 0;\\
    & D_{v,v,1,3}^1 = 0,\qquad D_{v,v,1,3}^2 = 0,\qquad D_{v,v,1,3}^3 = -\frac{3v^3v^1}{4};\\
    & D_{v,v,2,3}^1 = 0,\qquad D_{v,v,2,3}^2 = 0,\qquad D_{v,v,2,3}^3 = -\frac{v^3v^2}{4};\\
    & D_{v,v,3,3}^1 = \frac{v^3v^1}{4},\qquad D_{v,v,3,3}^2 = \frac{v^3v^2}{4},\qquad D_{v,v,3,3}^3 = 0;
    \end{align*}

    \begin{align*}
    R_{v,w,v}^1 & = \sum_{l,m = 1}^3\left( w^lD_{v,v,l,m}^1 - v^lD_{w,v,l,m}^1 \right) + \sum_{l=1}^3\left( \sum_{m=1}^2\sum_{m<j\leqslant 3}(v^mw^j - v^jw^m)c_{m,j}^l + \sum_{m=1}^3\left( \left( (v^mX_m)w^l \right) - w^m\left( X_mv^l \right) \right) \right)D_{v,l}^1 = \\
    & = \sum_{m=1}^3\left( w^1D_{v,v,1,m}^1 - v^1D_{w,v,1,m}^1 \right) + \sum_{m=1}^3\left( w^2D_{v,v,2,m}^1 - v^2D_{w,v,2,m}^1 \right) + \sum_{m=1}^3\left( w^3D_{v,v,3,m}^1 - v^3D_{w,v,3,m}^1 \right) + \\
    & + (v^1w^3 - v^3w^1)c_{1,3}^2D_{v,2}^1 + (v^1w^2 - v^2w^1)c_{1,2}^3D_{v,3}^1 =  \left( w^2\frac{3v^1v^2}{4} + w^2\frac{v^2v^1}{4} - v^2\frac{3w^1v^2}{4} - v^2\frac{w^2v^1}{4} \right) + \\
    & + \left( w^3\frac{v^1v^3}{4}+w^3\frac{v^3v^1}{4} - v^3\frac{w^1v^3}{4}-v^3\frac{w^3v^1}{4} \right) + \frac{v^3}{2}(-1)\left( v^1w^3 - v^3w^1 \right) + \left( -\frac{v^2}{2} \right)\left( v^1w^2 - v^2w^1 \right) = \\
    & = w^1\left( -\frac{3v^2v^2}{4} - \frac{v^3v^3}{4} + \frac{v^3v^3}{2} + \frac{v^2v^2}{2} \right) + w^2\left( \frac{3v^1v^2}{4} + \frac{v^1v^2}{4} - \frac{v^1v^2}{4} - \frac{v^1v^2}{2} \right) + w^3\left( \frac{v^1v^3}{4} + \frac{v^1v^3}{4} - \frac{v^1v^3}{4} - \frac{v^1v^3}{2} \right) = \\
    & = w^1\left( -\frac{v^2v^2}{4} + \frac{v^3v^3}{4} \right) + w^2\frac{v^1v^2}{4} - w^3\frac{v^1v^3}{4}.
    \end{align*}

    \begin{align*}
    R_{v,w,v}^2 & = \sum_{l,m = 1}^3\left( w^lD_{v,v,l,m}^2 - v^lD_{w,v,l,m}^2 \right) + \sum_{l=1}^3\left( \sum_{m=1}^2\sum_{m<j\leqslant 3}(v^mw^j - v^jw^m)c_{m,j}^l + \sum_{m=1}^3\left( \left( (v^mX_m)w^l \right) - w^m\left( X_mv^l \right) \right) \right)D_{v,l}^2 = \\
    & = \sum_{m=1}^3\left( w^1D_{v,v,1,m}^2 - v^1D_{w,v,1,m}^2 \right) + \sum_{m=1}^3\left( w^2D_{v,v,2,m}^2 - v^2D_{w,v,2,m}^2 \right) + \sum_{m=1}^3\left( w^3D_{v,v,3,m}^2 - v^3D_{w,v,3,m}^2 \right) + \\
    & + (v^2w^3 - v^3w^2)c_{2,3}^1D_{v,1}^2 + (v^1w^2 - v^2w^1)c_{1,2}^3D_{v,3}^2 = \\
    & = \left(w^1\left( -\frac{3v^1v^2}{4} \right) + w^1\left( -\frac{v^2v^1}{4} \right) - v^1\left( -\frac{3w^1v^2}{4} \right) - v^1 \left( -\frac{w^2v^1}{4} \right) \right) + \\
    & + \left(w^3\left( -\frac{v^2v^3}{4} \right) + w^3\left( \frac{v^3v^2}{4} \right) - v^3\left( -\frac{w^2v^3}{4} \right) - v^3\left( \frac{w^3v^2}{4} \right) \right) + (v^2w^3 - v^3w^2)\left( - \frac{v^3}{2} \right) + (v^1w^2 - v^2w^1)\left(- \frac{v^1}{2} \right) = \\
    & = w^1\left( -\frac{3v^1v^2}{4} - \frac{v^1v^2}{4} + \frac{3v^1v^2}{4} + \frac{v^1v^2}{2} \right) + w^2\left( \frac{v^1v^1}{4} + \frac{v^3v^3}{4} + \frac{v^3v^3}{2} - \frac{v^1v^1}{2} \right) + w^3\left( -\frac{v^2v^3}{4} + \frac{v^2v^3}{4} - \frac{v^2v^3}{4} - \frac{v^2v^3}{2} \right) = \\
    & = w^1\left( \frac{v^1v^2}{4} \right) + w^2\left( -\frac{v^1v^1}{4} + \frac{3v^3v^3}{4} \right) + w^3\left( -\frac{3v^2v^3}{4}  \right).
    \end{align*}

    \begin{align*}
    R_{v,w,v}^3 & = \sum_{l,m = 1}^3\left( w^lD_{v,v,l,m}^3 - v^lD_{w,v,l,m}^3 \right) + \sum_{l=1}^3\left( \sum_{m=1}^2\sum_{m<j\leqslant 3}(v^mw^j - v^jw^m)c_{m,j}^l + \sum_{m=1}^3\left( \left( (v^mX_m)w^l \right) - w^m\left( X_mv^l \right) \right) \right)D_{v,l}^3 = \\
    & = \sum_{m=1}^3\left( w^1D_{v,v,1,m}^3 - v^1D_{w,v,1,m}^3 \right) + \sum_{m=1}^3\left( w^2D_{v,v,2,m}^3 - v^2D_{w,v,2,m}^3 \right) + \sum_{m=1}^3\left( w^3D_{v,v,3,m}^3 - v^3D_{w,v,3,m}^3 \right) + \\
    & + \frac{3v^2}{2}c_{23}^1\left( v^2w^3 - v^3w^2 \right) + \frac{v^1}{2}c_{13}^2\left( v^1w^3 - v^3w^1 \right) = \\
    & = \left( w^1\left( -\frac{3v^1v^3}{4} \right) + w^1 \left( -\frac{3v^3v^1}{4} \right) - v^1\left( -\frac{3w^1v^3}{4} \right) - v^1\left( -\frac{3w^3v^1}{4} \right) \right) + \\
    & + \left( w^2\frac{v^2v^3}{4} + w^2\left( -\frac{v^3v^2}{4} \right) - v^2\frac{w^2v^3}{4} - v^2 \left( -\frac{w^3v^2}{4} \right) \right) + \frac{3v^2}{2}\left( v^2w^3 - v^3w^2 \right) - \frac{v^1}{2}\left( v^1w^3 - v^3w^1 \right) = \\
    & = w^1\left( -\frac{3v^1v^3}{4} - \frac{3v^1v^3}{4} + \frac{v^1v^3}{4} + \frac{v^1v^3}{2}  \right) + w^2\left( \frac{v^2v^3}{4} - \frac{v^2v^3}{4} - \frac{v^2v^3}{4} - \frac{3v^2v^3}{2}    \right) + w^3\left( \frac{v^1v^1}{4} + \frac{v^2v^2}{4} + \frac{3v^2v^2}{2} - \frac{v^1v^1}{2} \right) = \\
    & = w^1\left( -\frac{3v^1v^3}{4} \right) + w^2\left( -\frac{7v^2v^3}{4} \right) + w^3\left( -\frac{v^1v^1}{4} + \frac{7v^2v^2}{4} \right).
    \end{align*}

    \begin{align*}
    & -w^1R_{v, w, v}^1 + w^2R_{v,w,v}^2 + w^3R_{v,w,v}^3 = \\
    & = -w^1\left( w^1\left( -\frac{v^2v^2}{4} + \frac{v^3v^3}{4} \right) + w^2\frac{v^1v^2}{4} - w^3\frac{v^1v^3}{4} \right) + w^2\left( w^1\left( \frac{v^1v^2}{4} \right) + w^2\left( -\frac{v^1v^1}{4} + \frac{3v^3v^3}{4} \right) + w^3\left( -\frac{3v^2v^3}{4}  \right) \right) + \\
    & + w^3\left( w^1\left( -\frac{3v^1v^3}{4} \right) + w^2\left( -\frac{7v^2v^3}{4} \right) + w^3\left( -\frac{v^1v^1}{4} + \frac{7v^2v^2}{4} \right) \right) = w^1w^1\left( \frac{v^2v^2}{4} - \frac{v^3v^3}{4} \right) + w^2w^2\left( -\frac{v^1v^1}{4} + \frac{3v^3v^3}{4} \right) + \\
    & + w^3w^3\left( -\frac{v^1v^1}{4} + \frac{7v^2v^2}{4} \right) + w^1w^2\left( -\frac{v^1v^2}{4} + \frac{v^1v^2}{4} \right) + w^1w^3\left( \frac{v^1v^3}{4} - \frac{3v^1v^3}{4} \right) + w^2w^3\left( -\frac{3v^2v^3}{4} - \frac{7v^2v^3}{4} \right) = \\
    & = w^1w^1\left( \frac{v^2v^2}{4} - \frac{v^3v^3}{4} \right) + w^2w^2\left( -\frac{v^1v^1}{4} + \frac{3v^3v^3}{4} \right) + w^3w^3\left( -\frac{v^1v^1}{4} + \frac{7v^2v^2}{4} \right) - w^1w^3\frac{v^1v^3}{2} - w^2w^3\frac{5v^2v^3}{2}.
    \end{align*}

    \end{proof}


\begin{thebibliography}{99}

\bibitem{lie_control}
Yu. L. Sachkov, {\it Control Theory on Lie Groups}, SMFN, 2008, Volume 27, 5–59

\bibitem{filippov}
A. F. Filippov, {\it Introduction to the Theory of Differential Equations}, M.: LENAND, 2022. --- 240 p.

\bibitem{notes}
Agrachev A.A., Sachkov Yu.L. Geometric Control Theory. – M.: Fizmatlit, 2005.

\bibitem{curv_lor}
A.Z. Ali, Yu.L. Sachkov,
Sectional Curvature and Structural Functions of a Two-Dimensional or Three-Dimensional Lorentzian Manifold, {\em submitted for publication}.
\end{thebibliography}
\end{document}